\documentclass{amsart}

\usepackage{amsmath,amssymb,amsthm}
\usepackage[mathscr]{euscript}
\usepackage[utf8]{inputenc}
\usepackage{shuffle}
\usepackage{tabularx}
    \newcolumntype{L}{>{\raggedright\arraybackslash}X}
\usepackage{hyperref}
\usepackage{tikz, tikz-cd}
\usetikzlibrary{arrows,positioning} 
\tikzset{
    >=stealth',
    punkt/.style={
           rectangle,
           rounded corners,
           draw=black, very thick,
           text width=6.5em,
           minimum height=2em,
           text centered},
    pil/.style={
           ->,
           thick,
           shorten <=2pt,
           shorten >=2pt,}
}

\allowdisplaybreaks

\newtheorem{theorem}{Theorem}[section]
\newtheorem{lemma}[theorem]{Lemma}
\newtheorem{proposition}[theorem]{Proposition}
\newtheorem{corollary}[theorem]{Corollary}
\newtheorem*{theorem*}{Theorem}
\newtheorem{conjecture}{Conjecture}
\newtheorem{problem}[conjecture]{Problem}

\theoremstyle{definition}
\newtheorem{definition}[theorem]{Definition}
\newtheorem{example}[theorem]{Example}

\newtheorem*{remark*}{Remark}

\theoremstyle{remark}
\newtheorem{remark}[theorem]{Remark}

\numberwithin{equation}{section}

\DeclareMathOperator{\P1minus}{\mathbb{P}^1\backslash\left\{0,1,\infty\right\}}
\DeclareMathOperator{\M11}{\mathcal{M}_{1,1}}
\DeclareMathOperator{\Tatecurve}{\mathcal{E}^{\times}_{\partial/\partial \mathit{q}}}

\DeclareMathOperator{\dch}{dch}
\DeclareMathOperator{\Zm}{\mathcal{Z}^{\mm}}
\DeclareMathOperator{\ugeom}{\mathfrak{u}^{\geom}}
\DeclareMathOperator{\Ugeom}{\mathcal{U}^{\geom}}

\DeclareMathOperator{\ab}{ab}
\DeclareMathOperator{\mm}{\mathfrak{m}}
\DeclareMathOperator{\mot}{mot}
\DeclareMathOperator{\dR}{dR}

\DeclareMathOperator{\B}{B}
\DeclareMathOperator{\rel}{rel}

\DeclareMathOperator{\geom}{geom}
\DeclareMathOperator{\op}{op}
\DeclareMathOperator{\id}{id}
\DeclareMathOperator{\odd}{odd}
\DeclareMathOperator{\even}{even}

\DeclareMathOperator{\res}{res}
\DeclareMathOperator{\bp}{bp}

\def\adr{\mathsf{a}}
\def\bdr{\mathsf{b}}

\DeclareMathOperator{\gr}{gr}
\DeclareMathOperator{\GrL}{\mathscr{G}}

\DeclareMathOperator{\Co}{Co}
\DeclareMathOperator{\Fil}{Fil}
\DeclareMathOperator{\Hom}{Hom}
\DeclareMathOperator{\im}{im}
\DeclareMathOperator{\per}{per}
\DeclareMathOperator{\Aut}{Aut}
\DeclareMathOperator{\Out}{Out}
\DeclareMathOperator{\Lie}{Lie}
\DeclareMathOperator{\Spec}{Spec}
\DeclareMathOperator{\Der}{Der}

\DeclareMathOperator{\Isom}{Isom}
\DeclareMathOperator{\Sym}{Sym}
\DeclareMathOperator{\Gal}{Gal}
\DeclareMathOperator{\ad}{ad}
\DeclareMathOperator{\Ad}{Ad}

\DeclareMathOperator{\depth}{depth}
\DeclareMathOperator{\weight}{weight}

\DeclareMathOperator{\Alg}{\mathsf{Alg}}
\DeclareMathOperator{\MT}{\mathsf{MT}(\mathbb{Z})}
\DeclareMathOperator{\Sch}{\mathsf{Sch}}
\DeclareMathOperator{\Set}{\mathsf{Set}}
\DeclareMathOperator{\Rep}{\mathsf{Rep}}
\DeclareMathOperator{\Vect}{\mathsf{Vect}}
\DeclareMathOperator{\Loc}{\mathsf{LS}}
\DeclareMathOperator{\Con}{\mathsf{Con}}

\begin{document}

\title{Multiple zeta values and iterated Eisenstein integrals}

\author{Alex Saad}
\address{Department of Mathematics, University of Oxford}
\email{saad@maths.ox.ac.uk}
\thanks{The author is partly supported by EPSRC grant EP/M024830/1 and ERC grant 724638.}
\date{\today}
\dedicatory{This paper is dedicated to Geoffrey John}
\keywords{Multiple zeta values, iterated Eisenstein integrals, relative completion}

\begin{abstract}
Brown showed that the affine ring of the motivic path torsor $\pi_1^{\mot}(\P1minus, \vec{1}_0, -\vec{1}_1)$, whose periods are multiple zeta values, generates the Tannakian category $\MT$ of mixed Tate motives over $\mathbb{Z}$ \cite{brownmtm}. Brown also introduced multiple modular values \cite{mmv}, which are periods of the relative completion of the fundamental group of the moduli stack $\M11$ of elliptic curves \cite{hainmoduli}. We prove that all motivic multiple zeta values may be expressed as $\mathbb{Q}[2 \pi i]$-linear combinations of motivic iterated Eisenstein integrals along elements of $\pi_1 (\M11) \cong SL_2(\mathbb{Z})$, which are examples of motivic multiple modular values. This provides a new modular generator for $\MT$. We also explain how the coefficients in this linear combination may be partially determined using the motivic coaction.
\end{abstract}

\maketitle

\tableofcontents

\specialsection*{Introduction}

Multiple zeta values (MZVs) are real numbers defined by series of the form
\begin{equation*}
\zeta(k_1, \dots, k_r) = \sum_{0 < n_1 < \dots < n_r} \frac{1}{n_1^{k_1}\cdots n_r^{k_r}},
\end{equation*}
where $k_i\in\mathbb{N}$ and $k_r\geq 2$ to ensure convergence. The integer $k_1 + \dots + k_r$ is called the \emph{weight}, and the \emph{depth} is $r$. It is well-known that MZVs are periods of the motivic fundamental path torsor of $X:=\P1minus$ \cite{brownmtm}. They may be expressed as certain iterated integrals on $X$ of length equal to the weight.

In this paper we prove that all multiple zeta values may be expressed as $\mathbb{Q}[2\pi i]$-linear combinations of \emph{multiple modular values} \cite{mmv}. These are periods of the relative completion of the fundamental group of the moduli stack $\M11$ of elliptic curves \cite{hainhodge}, whose fundamental group is $\Gamma := SL_2(\mathbb{Z})$. The $\mathbb{Q}$-algebra of multiple modular values is very rich, containing the $L$-values of modular forms for $\Gamma$ (i.e. modular forms of level $1$), amongst other quantities (e.g. \cite[\S $7$]{brownihara}).

More specifically, we show that every MZV of weight $n$ and depth $r$ can be expressed as $(2\pi i)^n$ times a  $\mathbb{Q}$-linear combination of iterated integrals (along a specific element of $\pi_1 (\M11)$) of level-$1$ Eisenstein series
\begin{equation*}
\mathbb{G}_{2k} (\tau) = -\frac{B_{2k}}{4k} + \sum_{n\geq 1}\sigma_{2k-1} (n) q^n, \quad q := e^{2\pi i \tau}.
\end{equation*}
Moreover, the length of these iterated integrals is $s\leq r$ and the total modular weight is at most $n+s$.\footnote{In fact we will prove a stronger statement, where periods are replaced by motivic periods \cite{motivicperiods}. The advantage of working with motivic periods is that they are equipped with an action of a ``motivic'' Galois group, with the caveat that this group is only known to be truly \emph{motivic} (rather than Hodge-theoretic) when working with periods of mixed Tate motives. See  \S \ref{mpsection}.} With an appropriate choice of tangential basepoint $\vec{1}_\infty$ at the cusp \cite[\S $4$]{mmv}, the element of $\pi_1(\M11)\cong \Gamma$ we integrate along corresponds to
\begin{equation*}
S = \left(\begin{matrix} 0 & -1 \\ 1 & 0 \end{matrix} \right).
\end{equation*}
By identifying $S$ with the imaginary axis on the upper half plane $\mathfrak{H}$, our main result (Theorem \ref{maintheorem}) implies the following Theorem. (Here we must regularise with respect to the unit tangent vectors on $\mathfrak{H}$ based at $0$ and $i\infty$ \cite[\S $4$]{mmv}.)

\begin{theorem*}
Every MZV of weight $n$ and depth $r$ can be expressed as a $\mathbb{Q}$-linear combination of iterated integrals on $\mathfrak{H}$ of the form
\begin{equation*}
(2\pi i)^{n}\int_{\vec{1}_0}^{\vec{1}_\infty} \mathbb{G}_{2n_1+2} (\tau_1) \tau_1^{b_1} d\tau_1 \cdots \mathbb{G}_{2n_s+2}(\tau_s) \tau_s^{b_s} d\tau_s,
\end{equation*}
where $s\leq r$, the total modular weight $N:=(2n_1+2) + \dots + (2 n_s+2)$ is bounded above by $n+s$, and $0\leq b_i \leq 2n_i$.
\end{theorem*}
This follows from Theorem \ref{maintheorem} by applying the period map. Note that Theorem \ref{maintheorem} is stated in the de Rham normalisation, which alters the power of $2\pi i$. See Remark \ref{bettidrintegrals} for a more precise statement.

\subsubsection*{Examples}

The simplest nontrivial example is the case of $\zeta(3)$. It is a multiple zeta value of weight $3$ and depth $1$, and can therefore be expressed as an iterated integral on $X$ of length $3$. The well-known formula is
\begin{equation*}
\zeta(3) = \int_{0<x<y<z<1} \frac{dx dy dz}{(1-x)y z},
\end{equation*}
which may be verified by expanding $1/(1-x)$ as a geometric series and integrating term-by-term. By our result it may be written as a single integral on $\mathfrak{H}$. Indeed,
\begin{equation*}
\zeta(3) = -(2\pi i)^3 \int_{\vec{1}_0}^{\vec{1}_\infty} \mathbb{G}_4 (\tau) d\tau.
\end{equation*}
After expanding $\mathbb{G}_4$ in $q$ this instead expresses $\zeta(3)$ as a rapidly converging Lambert series. Similar formulae exist for all odd zeta values \cite{ramanujan}; for example
\begin{equation*}
    \zeta(5) = -\frac{1}{12} (2 \pi i)^5 \int_{\vec{1}_0}^{\vec{1}_\infty} \mathbb{G}_6(\tau) d\tau.
\end{equation*}

A more involved combination occurs in depth $2$. Brown gave the first example of an expression of this type for the value $\zeta(3,5)$ \cite[Example $7.2$]{brownihara}. His formula is\footnote{There is a difference in normalisation between Brown's formulae and ours; namely $\Lambda(\mathbb{G}_{k_1}, \dots, \mathbb{G}_{k_s};b_1, \dots, b_s) = i^{b_1 +\dots + b_s} \int_{\vec{1}_0}^{\vec{1}_1} \mathbb{G}_{k_1} (\tau_1) \tau_1^{b_1} d\tau_1 \cdots \mathbb{G}_{k_s}(\tau_s) \tau_s^{b_s} d\tau_s$. This introduces the extra factor of $(-1)$ here that is not present in Brown's formula.}
\begin{equation} \label{zeta35intro}
\zeta(3,5) = -\frac{5}{12}(2\pi i)^8 \int_{\vec{1}_0}^{\vec{1}_\infty}\mathbb{G}_6 (\tau_1) d\tau_1 \mathbb{G}_4  (\tau_2) d\tau_2 + \frac{503}{2^{13} 3^5 5^2 7} (2\pi i)^8.
\end{equation}
This agrees with our result, which states that $\zeta(3,5)$ may be expressed as a linear combination of at most double Eisenstein integrals of total modular weight $\leq 10$. Compare this to the expression for $\zeta(3,5)$ as an iterated integral on $X$ of length $8$.\footnote{In general the depth of an MZV is always smaller than its weight. Our result shows that there is a substantial simplification to the length filtration on iterated integrals in passing from $X$ to $\M11$.}

However, not all linear combinations of iterated Eisenstein integrals are equal to MZVs. The formula
\begin{align*}
& 600\pi \int_{\vec{1}_0}^{\vec{1}_\infty} \mathbb{G}_4 (\tau_1) \tau_1 d\tau_1 \mathbb{G}_{10} (\tau_2) \tau_2^4 d\tau_2 + 480 \pi \int_{\vec{1}_0}^{\vec{1}_\infty} \mathbb{G}_4 (\tau_1) \tau_1^2 d\tau_1  \mathbb{G}_{10} (\tau_2) \tau_2^3 d\tau_2 \\
& = \int_0^{i\infty} \Delta (\tau) \tau^{11} d\tau = \Lambda(\Delta, 12),
\end{align*}
given in \cite[Example $7.3$]{brownihara}, exhibits a linear combination of iterated Eisenstein integrals equal to a noncritical completed $L$-value of the Ramanujan cusp form $\Delta\in S_{12} (\Gamma)$. This multiple modular value is not expected to be an MZV.

\subsubsection*{Algebraic structure}

The space spanned by iterated Eisenstein integrals has a rich algebraic structure. The subspace consisting of MZVs is generated by periods of a certain pro-nilpotent Lie subalgebra $\ugeom\subseteq \Der \Lie(\adr, \bdr)$.\footnote{Here we have slightly abused the concept of a period; to be more precise we are referring to periods of the affine ring of the associated pro-unipotent group scheme, which is an ind-object in the category of mixed Tate motives over $\mathbb{Z}$.} It is generated by derivations $\varepsilon_{2n+2}^\vee$, for each $n\geq 1$, which were originally studied by Tsunogai \cite{tsunogai} in the $\ell$-adic context. They are defined by
\begin{align*}
\varepsilon_{2n+2}^\vee (\adr) &= \ad(\adr)^{2n+2} (\bdr) \\
\varepsilon_{2n+2}^\vee (\bdr) & = \frac{1}{2} \sum_{i+j = 2n+1} (-1)^i \left[\ad(\adr)^i (\bdr), \ad(\adr)^j (\bdr) \right].
\end{align*}
There are many arithmetic relations in $\ugeom$, some of which were studied by Pollack in his masters' thesis \cite{pollack}. For each depth in the lower central series filtration on $\ugeom$ there is a family of relations whose coefficients are connected to period polynomials of cusp forms for $\Gamma$. Some examples in depth $2$ are
\begin{align*}
[\varepsilon_{10}^\vee, \varepsilon_4^\vee] - 3 [\varepsilon_8^\vee, \varepsilon_6^\vee] &=0 \\
2[\varepsilon_{14}^\vee, \varepsilon_4^\vee]  - 7[\varepsilon_{12}^\vee, \varepsilon_6^\vee] + 11 [\varepsilon_{10}^\vee, \varepsilon_8^\vee] &=0.
\end{align*}

Geometrically, $\ugeom$ is the infinitesimal image of the monodromy representation of the unipotent radical of the relative de Rham fundamental group $\mathcal{G}_{1,1}^{\dR}:=\pi_1^{\rel, \dR}(\M11)$ on the de Rham fundamental group of the infinitesimal punctured Tate curve (see \S \ref{monodromysection}). The group $\mathcal{G}_{1,1}^{\dR}$ is generated by symbols corresponding to a basis of all modular forms for $\Gamma$, tensored with irreducible representations of $SL_2/\mathbb{Q}$. The derivations $\varepsilon_{2n+2}^\vee$ are the images of Eisenstein generators, while the cuspidal symbols act trivially.

The Lie algebra $\ugeom$ is a pro-object in the category $\MT$ of mixed Tate motives over the integers. The de Rham fiber functor $\omega^{\dR} \colon \MT\to\Vect_{\mathbb{Q}}$ sending a mixed Tate motive to its de Rham realisation (algebraic de Rham cohomology) equips $\MT$ with the structure of a neutral Tannakian category over $\mathbb{Q}$ \cite{deligne, brownmtm}. Consequently, the \emph{motivic Galois group} $G_{\MT}^{\dR}:= \Aut^{\otimes}_{\MT} (\omega^{\dR})$, and its unipotent radical $U_{\MT}^{\dR}$, act on $\ugeom$.

As a corollary of our result, which crucially relies upon its validity at the level of motivic periods, we obtain the following theorem (Theorem \ref{motivetheorem}):
\begin{theorem*}
The action of $G_{\MT}^{\dR}$ on $\ugeom$ is faithful.
\end{theorem*}
This confirms a conjecture of Brown \cite[Remark $14.6$]{mmv}, suggesting that every mixed Tate motive over $\mathbb{Z}$ may be constructed from modular forms. It is a ``modular'' analogue of Brown's result \cite{brownmtm}, which implies that $G_{\MT}^{\dR}$ acts faithfully on the motivic fundamental path torsor $\pi_1^{\mot}(X, \vec{1}_0, -\vec{1}_1)$, or of Belyi's Theorem \cite{belyi}, which implies that $\Gal(\bar{\mathbb{Q}}/\mathbb{Q})$ acts faithfully on the profinite fundamental group $\pi_1^{\text{\'{e}t}} (X\times_{\mathbb{Q}} \bar{\mathbb{Q}},b)$ for any rational basepoint $b$.

The faithfulness of the Galois action on $\ugeom$ is connected to the Pollack relations. Let $\Lie \Pi_Y^{\dR}$ be the Lie algebra of the de Rham fundamental group of $Y:=\Tatecurve$, the infinitesimal smoothing of the punctured Tate curve. It is canonically isomorphic to the completed free Lie algebra $\Lie(\adr, \bdr)^\wedge$.

Although $\Tatecurve$ is a topological object, its de Rham fundamental group $\Pi_Y^{\dR}$ (or equivalently, its Lie algebra $\Lie(\adr, \bdr)^\wedge$) has a motivic structure, and $\Lie(\adr, \bdr)^\wedge$ is a pro-object of $\MT$ \cite{brownzetaelements, umem}. Consequently, it has an action by the Lie algebra $\mathfrak{k} = \Lie(U_{\MT}^{\dR})$, which is non-canonically isomorphic to the completed free Lie algebra on generators $\sigma_{2n+1}$ for all $n\geq 1$ \cite{brownmtm}. This action is described by a homomorphism
\begin{equation*}
\rho \colon \mathfrak{k}\to \Der \Lie(\adr, \bdr).
\end{equation*}
In \cite{brownzetaelements}, Brown proved that $\rho$ is injective and that there is a choice of generators $\sigma_{2n+1}\in\mathfrak{k}$ that act to lowest order through $\ugeom\subseteq \Der \Lie(\adr, \bdr)$ via
\begin{equation}\label{nakamurader1}
\rho(\sigma_{2n+1}) \equiv \varepsilon_{2n+2}^\vee \pmod{W_{-2n-3}},
\end{equation}
where $W_\bullet \Der \Lie (\adr, \bdr)$ is the negative of the lower central series filtration. The image of $\rho$ is contained in the normaliser of $\ugeom$ within $\Der \Lie(\adr, \bdr)$. This follows because $\ugeom$ is also a pro-object of $\MT$ and the actions of $\mathfrak{k}$ on $\ugeom$ and on $\Lie(\adr, \bdr)^\wedge$ must be compatible: the $\mathfrak{k}$-action on $\ugeom$ is described by a homomorphism $\tilde{\rho} \colon \mathfrak{k}\to \Der(\ugeom)$ satisfying $\tilde{\rho}(\sigma)(\varepsilon) = [\rho(\sigma), \varepsilon]$ for any $\varepsilon \in \ugeom$.

If $\ugeom$ had no relations, \eqref{nakamurader1} would trivially imply that $\tilde{\rho}$ is injective (i.e. that the $\mathfrak{k}$-action on $\ugeom$ is faithful). However, the Pollack relations in $\ugeom$ prevent this, and can be viewed as a potential cuspidal obstruction to a faithful Galois action on $\ugeom$. Our result implies that $\tilde{\rho}$ is still injective despite this obstruction. This result can be viewed as ``orthogonal'' to a consequence of Oda's conjecture\footnote{Now a theorem by \cite{takao, brownmtm}.} \cite{oda}, which implies a different injectivity result for $\mathfrak{k}$. This is discussed in \S \ref{odasection}.

\subsubsection*{Context}

There are close links between MZVs and other modular and elliptic periods. For example, every MZV may be written as a $\mathbb{Q}[\log(2)^\pm]$-linear combination of iterated integrals of certain weight $2$ modular forms for $\Gamma_0 (4)$ \cite[Theorem $8.1$]{brownLvalues}.

In a different vein, Lochak-Matthes-Schneps \cite{lochakmatthesschneps} showed that the algebra of MZVs is contained within the algebra of elliptic multiple zeta values modulo $2\pi i$. These are functions on $\mathfrak{H}$ given by \emph{indefinite} iterated Eisenstein integrals.

\subsubsection*{Proof idea}

In this section we explain the heuristics of the proof, leaving technicalities and precise definitions for later.

Firstly, we make use of a slightly modified moduli scheme $\mathcal{M}_{1,\vec{1}}$ classifying elliptic curves $E$ together with a choice of tangential basepoint $\vec{v}$ at the origin. It is an affine scheme, and its associated analytic space is a complex manifold whose topological fundamental group is isomorphic to the braid group $B_3$ on three strands. Forgetting the tangential basepoint induces a morphism $\mathcal{M}_{1,\vec{1}}\to \M11$ equipping $\mathcal{M}_{1,\vec{1}}$ with the structure of a principal $\mathbb{G}_m$-bundle over $\M11$. In turn, this induces a natural homomorphism of fundamental groups $B_3 \to \Gamma$.

The benefits of working with $\mathcal{M}_{1,\vec{1}}$ are twofold. Firstly, the space of iterated integrals on $\mathcal{M}_{1,\vec{1}}$ is essentially the same as that on $\M11$; the only additional numbers one obtains are powers of $2\pi i$ (see \eqref{vectordecomp}). Secondly, the addition of basepoint data to the moduli problem equips $\pi_1 (E^\times, \vec{v})$ with a natural ``monodromy'' action by $\pi_1(\mathcal{M}_{1,\vec{1}})$. In particular, choosing $E = \mathcal{E}_{\partial/\partial q}$ to be the infinitesimal Tate curve and $\vec{v}$ to be the unit tangent vector $\vec{1}_O$ at the origin defines an action
\begin{equation*}
    \pi_1 (\mathcal{M}_{1,\vec{1}}) \times \pi_1(\Tatecurve, \vec{1}_O) \to \pi_1(\Tatecurve, \vec{1}_O).
\end{equation*}

This action may be explicitly described using generators and relations (see \S \ref{topmonodromysection}). The group  $\pi_1(\Tatecurve, \vec{1}_O)$ is free on two standard generators $\alpha$ and $\beta$. The group $\pi_1 (\mathcal{M}_{1,\vec{1}})$ is generated by certain (explicitly defined) elements of the braid group, called $\tilde{S}$ and $\tilde{T}$. Under the map $B_3\to \Gamma$ these are sent to the generators
\begin{equation*}
S = \left(\begin{matrix} 0 & -1 \\ 1 & 0 \end{matrix} \right), \quad T = \left(\begin{matrix} 1 & 1 \\ 0 & 1 \end{matrix} \right) \in \Gamma
\end{equation*}
respectively. The elements $\tilde{S}$ and $\tilde{T}$ act on $\alpha$ and $\beta$ via combinations of Dehn twists, and this action may be computed explicitly. In particular we have
\begin{equation} \label{fromonesourceallthingsdepend}
    \tilde{S}(\beta) = \alpha^{-1}.
\end{equation}
The importance of this equation is that it relates iterated integrals on $\Tatecurve$ along $\alpha$ to those along $\beta$ via iterated integrals on $\mathcal{M}_{1, \vec{1}}(\mathbb{C})$ along $\tilde{S}$.

We show in \S \ref{genus0andTate} that all multiple zeta values occur as iterated integrals along both $\alpha$ and $\beta$, but that their distributions within these two spaces of iterated integrals differ. By \eqref{fromonesourceallthingsdepend}, the \emph{difference} between these two spaces of iterated integrals on $\Tatecurve$ is contained in the space of iterated integrals on $\mathcal{M}_{1,\vec{1}}$ along $\tilde{S}$. In this way we are able to show that all MZVs occur as natural iterated integrals along the element $\tilde{S}$.

Finally, we must show that the MZVs within the space of iterated integrals along $\tilde{S}$ are contained within the subspace of iterated Eisenstein integrals. This uses the structure of the relative completion of $\Gamma$ and its monodromy action; in particular, it uses the fact that the cuspidal generators act trivially (see Proposition \ref{explicitmonodromy}).

The heuristics of the proof are summarised in the following diagram:
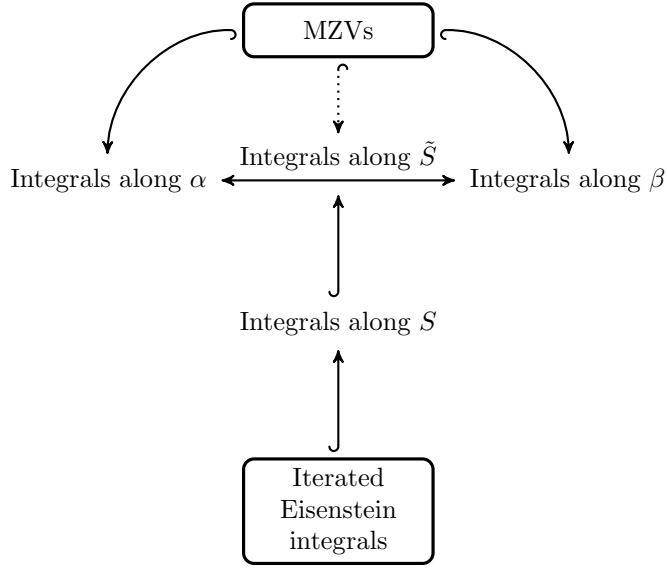
\begin{equation*}
\begin{tikzpicture}[node distance=1.5cm, auto,]
 %nodes
 \node[punkt] (mzv) {MZVs};
 % We make a dummy figure to make everything look nice.
 \node[below=of mzv] (dummy) {}
    edge[pil, <- right hook, dotted, shorten <= 0.5cm] (mzv.south);
 \node[right=of dummy] (t) {Integrals along $\beta$}
   edge[pil, <- left hook , bend right=45] (mzv.east);
 \node[left=of dummy] (g) {Integrals along $\alpha$}
    edge[pil, <->] node[auto] {Integrals along $\tilde{S}$} (t)
   edge[pil, <- right hook, bend left=45] (mzv.west);
 \node[below=of dummy] (s) {Integrals along $S$}
    edge[pil, right hook ->, ] (dummy);
 \node[punkt, below=of s] (e) {Iterated Eisenstein integrals}
    edge[pil, right hook ->,] (s);
\end{tikzpicture}
\end{equation*}

In order to formalise this argument and to formally work with the ``space of iterated integrals along a path'', we make judicious use of the notion of the (relative) de Rham fundamental group of a space $X/\mathbb{Q}$. This is an affine group scheme over $\mathbb{Q}$. The points of its unipotent radical with values in the algebra $\mathcal{P}^{\mm}_{\mathcal{H}}$ of motivic $\mathcal{H}$-periods receive a map\footnote{When the fundamental group is itself unipotent, this map is a homomorphism. In the general case it is a cocycle for $\pi_1 (X)$.} from the topological fundamental group of $X(\mathbb{C})$. Suitably interpreted, this map sends $\gamma \in \pi_1 (X)$ to a noncommutative formal generating series for motivic iterated integrals on $X(\mathbb{C})$ along $\gamma$. Below we define the main spaces, topological paths and generating series of periods used in the argument.
\begin{itemize}
    \item Let $X = \P1minus$. The straight line path $\dch$ between the tangential basepoints at $0$ and $1$ is mapped to the \emph{Drinfeld associator} $\Phi_{0 1}^{\mm} = \sum_w \zeta^{\mm} (w) w$. It is a power series in variables $\mathsf{x}_0, \mathsf{x}_1$.
    \item Let $X = \Tatecurve$. The two natural generators $\alpha, \beta$ for the fundamental group of the punctured torus are mapped to power series $\alpha^{\mm}, \beta^{\mm}$ in variables $\adr, \bdr$. They are products of exponentials and Drinfeld associators.
    \item Let $X = \M11$. The element $S \in \Gamma \cong \pi_1(\M11)$ is mapped to an element $\mathcal{C}_S^{\mm}$. It is a power series in symbols corresponding to a basis for modular forms for $\Gamma$, with coefficients in $SL_2$-representations. It is the value at $S$ of a certain ``canonical cocycle'' \cite[Definition $15.4$]{mmv}, and by definition its coefficients are motivic \emph{multiple modular values} \cite{mmv}.
    \item Let $X = \mathcal{M}_{1,\vec{1}}$. The element $\tilde{S} \in \pi_1 (\mathcal{M}_{1, \vec{1}})$ is mapped to $\Psi = \exp(\eta \mathbf{e}_2) \mathcal{C}_S^{\mm}$. The symbol $\mathbf{e}_2$ corresponds to the Eisenstein series $\mathbb{G}_2$ of weight $2$; $\eta$ is a motivic period whose value we compute as $\eta = \mathbb{L}/8$ in Corollary \ref{qvalue}. 
\end{itemize}

The relationship $\tilde{S}(\beta) = \alpha^{-1}$ implies a similar equation holds between these generating series; namely,\footnote{The automorphism $S^{\mm} \colon (\adr, \bdr) \mapsto (-\mathbb{L}^{-1} \bdr, \mathbb{L}\adr)$ in \eqref{introfundamentaleq} is present for technical reasons.}
\begin{equation} \label{introfundamentaleq}
    \mu(\Psi)(S^{\mm}(\beta^{\mm})) = (\alpha^{\mm})^{-1}.
\end{equation}
Here $\mu$ is a morphism of group schemes called the monodromy morphism (defined in \S \ref{monodromysection}). Its image consists of noncommutative power series in elements of $\ugeom \subseteq \Der \Lie(\adr, \bdr)$ together with an additional central derivation $\varepsilon_2$. Since $\mu$ kills cuspidal symbols (see Lemma \ref{explicitmonodromy}), the coefficients of $\mu(\Psi)$ are motivic iterated Eisenstein integrals. Moreover these coefficients must be contained within the subalgebra of motivic periods of mixed Tate motives because the image of $\mu$ is a pro-object in $\MT$ (Proposition \ref{motivicremark}). In other words, $\mu(\Psi)$ is a generating series for all motivic iterated Eisenstein integrals equal to $\mathbb{Q}[\mathbb{L}^\pm]$-linear combinations of motivic MZVs.

Our goal is to use \eqref{introfundamentaleq} to understand the coefficients of $\mu(\Psi)$ in terms of those of $S^{\mm}(\beta^{\mm})$ and $(\alpha^{\mm})^{-1}$. We show that the coefficients of these series are motivic MZVs, and that \emph{all} motivic MZVs occur. We exhibit explicit bounds for the growth of these coefficients in terms of natural filtrations defined in terms of the degrees in $\adr$ and $\bdr$ (Lemmas \ref{betacoeffs} and \ref{alphacoeffs}). To do this formally we introduce the notion of a coefficient space (Definition \ref{coeffspacedef}).

We then compare the difference in the filtered coefficient spaces of $S^{\mm}(\beta^{\mm})$ and $(\alpha^{\mm})^{-1}$. By equation \eqref{introfundamentaleq}, the difference is contained within the filtered coefficient space of $\mu(\Psi)$. In this way we are able to show that all motivic MZVs occur within the coefficients of $\mu(\Psi)$. We conclude the proof by relating the MZV and modular weights by a Hodge-theoretic argument.

\subsubsection*{Coefficients}

Our proof is nonconstructive; we are able to show that every motivic MZV may be written as a linear combination of iterated Eisenstein integrals of certain lengths and modular weights, but \emph{a priori} we cannot determine the coefficients in this expression. Nevertheless, in \S \ref{coefficientsection} we show that some coefficients may be determined by comparing the additional information coming from $f$-alphabet decompositions of motivic periods \cite[\S $5.4$]{motivicperiods} \cite[\S $22$]{mmv}.

As an example, consider the formula \eqref{zeta35intro} expressing $\zeta^{\mm}(3,5)$ as a double Eisenstein integral. The coefficient $-5/12$ of the longest word in this expression may be computed as follows.

Theorem \ref{maintheorem} implies that $\zeta^{\mm}(3,5)$ may be written as a linear combination of motivic iterated Eisenstein integrals of length at most $2$ and total modular weight at most $10$. There are many possible iterated Eisenstein integrals satisfying these criteria, but the space they span in fact has many relations. By a computation given in Example \ref{example2}, we are able to show that
\begin{equation} \label{introexamplemotivic}
    \zeta^{\mm}(3,5) = A \mathbb{L}^6 I_{6, 4}^{0,0} + B \mathbb{L}^8,
\end{equation}
where $A, B \in \mathbb{Q}$ must be determined and where $I_{2n_1+2, \dots, 2n_s+2}^{b_1, \dots, b_s}$ is the motivic analogue of the iterated integral
\begin{align*}
    &\int_S \mathbb{G}_{2n_1+2}(q_1) \log(q_1)^{b_1} \frac{dq_1}{q_1} \cdots \mathbb{G}_{2n_s+2}(q_s) \log(q_s)^{b_s} \frac{dq_s}{q_s} \\
    = & (2 \pi i)^{b_1 + \dots + b_s + s} \int_{\vec{1}_0}^{\vec{1}_\infty} \mathbb{G}_{2n_1+2} (\tau_1) \tau_1^{b_1} d\tau_1 \cdots \mathbb{G}_{2n_s+2}(\tau_s) \tau_s^{b_s} d\tau_s.
\end{align*}
To leading order in the coradical filtration \cite[\S $2.5$, \S $3.8$]{motivicperiods} the $f$-alphabet decomposition of $\zeta^{\mm}(3,5)$ is $-5 f_5 f_3$ \cite[\S $3$]{browndecomposition}. By Lemma \ref{falphabet} the decomposition of $I_{6, 4}^{0,0}$ is $12 \mathbb{L}^{-6} f_5 f_3$. Comparing coefficients in \eqref{introexamplemotivic} determines $A = -5/12$. The coefficient $B$ cannot be determined in this manner because of the inherent ambiguity associated with splitting the coradical filtration. This computation, together with the outline of a general procedure for determining coefficients, is covered in \S \ref{coefficientsection}.

\subsubsection*{Structure of the paper}

This paper is structured as follows: \S $1$-$4$ should be standard to experts. \S $1$ establishes notation; \S $2$ consists of background algebraic geometry and a review of motivic periods; \S $3$ generalities and specifics on fundamental groups; and \S $4$ is a review of $\ugeom$.

In \S $5$ we study filtrations on fundamental groups and the induced filtrations on coefficient spaces, and introduce several essential pieces of notation.

\S $6$ studies links between moduli spaces and the infinitesimal punctured Tate curve via the monodromy action.

\S $7$ defines the canonical cocycle, series $\Psi$, motivic multiple modular values and motivic iterated Eisenstein integrals.

\S $8$ studies links between $\P1minus$ and the infinitesimal punctured Tate curve via the Hain morphism. We define the series $\alpha^{\mm}$ and $\beta^{\mm}$ and compute bounds on their coefficients in terms of the depth and weight of motivic MZVs.

\S $9$ consists of the proof of Theorem \ref{maintheorem}, the main result of this paper.

\S $10$ explores Galois-theoretic consequences of Theorem \ref{maintheorem}. We prove that the motivic Galois group of the category of mixed Tate motives over $\mathbb{Z}$ acts faithfully on $\ugeom$, and discuss the related Oda Conjecture.

\S $11$ returns to the depth $1$ case of Theorem \ref{maintheorem}. We use Brown's explicit formula for $\mathcal{C}_S^{\mm}$ in length $1$ and explicit formulae for $\phi$ and $\mu$ to compute the unknown motivic period $\eta$ occuring in $\Psi$, and compare this to a formula of Matthes \cite{matthesE2, matthesquasi}.

In \S $12$ we show how an $f$-alphabet decomposition can be used to partially determine the coefficients in an expression for a multiple zeta value in terms of iterated Eisenstein integrals.

\section*{Acknowledgements}

The author would like to thank F. Brown for his extensive guidance and support throughout this project. Many thanks are also due to N. Matthes, M. Luo and A. Keilthy for fruitful discussions and thorough readings of early versions of this text, and to R. Hain for his comments regarding Oda's conjecture.

\section{Notation and conventions}

\subsection{Representations of $SL_2$} \label{SL2reps}

For $n\geq 0$, let $V_n$ be the $\mathbb{Q}$-vector space of homogeneous polynomials in $X$ and $Y$ of degree $n$. It is equipped with a \emph{right} action of the group scheme $SL_2$: for $R$ a $\mathbb{Q}$-algebra, $p(X,Y) \in V_n \otimes_{\mathbb{Q}} R$, and
\begin{equation*}
\gamma = \left(\begin{matrix} a & b \\ c & d \end{matrix} \right) \in SL_2 (R),
\end{equation*}
this is defined by $p(X,Y)\vert_\gamma = p(aX + bY, cX + dY)$.

There is also a de Rham version of $V_n$, denoted by $V_n^{\dR}$, generated by de Rham indeterminates $\mathsf{X}, \mathsf{Y}$. They are related to the Betti generators $X,Y$ via $(X,Y)\mapsto (\mathsf{X}, \mathbb{L} \mathsf{Y})$, where $\mathbb{L}$ is the motivic period analogue of $2\pi i$ (see \S\ref{mpsection}). The action $(-)\vert^{\dR}$ of the (Betti) $SL_2$ on the de Rham generators is therefore twisted by powers of $\mathbb{L}$; in particular, given $\gamma$ as before and $p(\mathsf{X}, \mathsf{Y})\in V_n^{\dR}\otimes R$, we have
\begin{equation*}
p(\mathsf{X}, \mathsf{Y})\vert^{\dR}_\gamma = p(\mathsf{X}, \mathsf{Y})\vert_{\epsilon \gamma \epsilon^{-1}} \text{ where } \epsilon = \left(\begin{matrix} 1 & 0 \\ 0 & \mathbb{L}^{-1} \end{matrix} \right).
\end{equation*}
Equivalently, one can interpret this as two different forms of $SL_2/\mathbb{Q}$ (Betti and de Rham forms, denoted $SL_2^{\B}$ and $SL_2^{\dR}$), which are isomorphic over $\mathbb{Q}[\mathbb{L}, \mathbb{L}^{-1}]$ via $\gamma\mapsto \epsilon \gamma \epsilon^{-1}$. In particular, the images of the two standard generators $S,T\in SL_2(\mathbb{Z})$ under the composition $SL_2(\mathbb{Z}) \hookrightarrow SL_2^{\B}(\mathbb{Q}) \to SL_2^{\dR} (\mathbb{Q}[\mathbb{L}, \mathbb{L}^{-1}])$ are
\begin{equation*}
S^{\mm} = \left(\begin{matrix} 0 & -\mathbb{L} \\ \mathbb{L}^{-1} & 0 \end{matrix} \right), \quad T^{\mm} = \left(\begin{matrix} 1 & \mathbb{L} \\ 1 & 0 \end{matrix} \right).
\end{equation*}

\subsection{Modular forms}

Let $\mathfrak{H} = \left\{\tau\in\mathbb{C}:\Im(\tau)>0\right\}$ and let $q := \exp(2\pi i \tau)$ be the corresponding coordinate on the punctured $q$-disc $D^\times = \left\{q\in\mathbb{C}: 0<\lvert q\rvert  < 1\right\}$.

We only consider modular forms $f$ for the full modular group $\Gamma = SL_2(\mathbb{Z})$. Such modular forms are of \emph{level} $1$.

The Hecke-normalised Eisenstein series of weight $2k \geq 4$ is the modular form $\mathbb{G}_{2k}$ with $q$-expansion
\begin{equation} \label{Eisdef}
\mathbb{G}_{2k} (\tau) = -\frac{B_{2k}}{4k} + \sum_{n\geq 1}\sigma_{2k-1} (n) q^n.
\end{equation}
Here the arithmetic function $\sigma_r$ is defined as $\sigma_r (n) := \sum_{d\vert n} d^r$.

\subsection{Semidirect products} \label{sdpconvention}

Let $G$ and $\Pi$ be groups, and suppose that $G$ acts on $\Pi$ on the \emph{left} via $\pi\mapsto g(\pi)$. The semidirect product $\Pi \rtimes G$ with respect to this action is the group with underlying set $\Pi\times G$ and product $(\pi_1, g_1)(\pi_2, g_2) = (\pi_1 g_1 (\pi_2), g_1 g_2)$.

Let us suppose instead that $G$ acts on $\Pi$ on the \emph{right} via $\pi \mapsto \pi\vert_g$. The semidirect product $G\ltimes \Pi$ with respect to this action is the group with underlying set $G\times \Pi$ and product $(g_1, \pi_1) (g_2, \pi_2) = (g_1 g_2, \pi_1 \vert_{g_2} \pi_2)$.

\subsection{Fundamental groups}

We use the topologists' convention regarding path multiplication in fundamental groups. If $\alpha, \beta$ are two elements in a fundamental group, the product $\alpha \beta$ is homotopic to the result of first traversing $\alpha$ and then $\beta$.

\subsection{Hopf algebras}

If $H$ is a Hopf algebra over a field of characteristic zero, let
\begin{equation*}
\GrL(H) = \left\{x\in H : x\text{ invertible and } \Delta(x) = x\otimes x\right\}
\end{equation*}
be its group of grouplike elements.

\subsubsection{The shuffle algebra} \label{hopfshufflealg}

Let $Z$ be a set and $R$ a $\mathbb{Q}$-algebra. Let $R\langle Z \rangle$ be the $R$-span of all words $w$ in the alphabet $Z$, including the empty word $\varnothing$. It is a commutative $\mathbb{Q}$-algebra equipped with the shuffle product $\shuffle$. It can be equipped with the structure of a Hopf algebra over $\mathbb{Q}$ with the deconcatenation coproduct.

\subsubsection{The Hopf algebra of noncommutative formal power series} \label{hopfalgnoncomm}

Let $Z$ and $R$ be as above. The ring $R\langle\langle Z \rangle\rangle$ of formal power series with noncommuting indetermines in $Z$ consists all formal power series with coefficients in $R$ whose indeterminates are words in the elements of $Z$, equipped with the concatenation product. It has the structure of a Hopf algebra over $\mathbb{Q}$ when equipped with the coproduct for which elements of $Z$ are primitive.

\subsection{Tangential basepoints}

We regularly use fundamental groups with \emph{tangential basepoints} \cite[\S $15$]{deligne}. For completeness we include the following algebraic definition of a tangential basepoint, attributed to Nakamura \cite[Definition $1.1$]{nakamuratangential}.

\begin{definition}[Tangential basepoint] \label{tangentialbasepoint}
Let $X$ be a connected scheme and let $K$ be a field of characteristic $0$. A $K$-rational tangential basepoint is a morphism $\vec{v} \colon \Spec K((q))\to X$. For notational reasons, define
\begin{equation*}
X(K)_{\bp} := X(K)\cup\left\{K\text{-rational tangential basepoints}\right\}.
\end{equation*}
If $X$ is a curve, the unit tangent vector at $p\in X(K)$ is denoted $\vec{1}_p\in X(K)_{\bp}$.
\end{definition}

\section{Background notions}

In this section we recall background notions from algebraic geometry, topology and the Tannakian theory of motivic periods.

\subsection{Geometry}

In this subsection we define the moduli spaces $\M11$ and $\mathcal{M}_{1,\vec{1}}$, the universal family $\mathcal{E}\to \M11$ and the infinitesimal punctured Tate curve $\Tatecurve$. We give both an algebraic and analytic description of these spaces.

\subsubsection{Algebraic definition of the moduli spaces $\M11$ and $\mathcal{M}_{1,\vec{1}}$} \label{modulialg}

Let $\M11$ be the moduli stack of elliptic curves. It is the stack over $\Spec(\mathbb{Z})$ whose points $\M11(X)$ classify isomorphism classes of elliptic curves $(E,O)$ over a scheme $X$. Here $E\to X$ is an elliptic curve over $X$ and $O \colon X\to E$ is a distinguished section.

Consider the functor $F \colon \Sch^{\op}\to \Set$ that assigns to a scheme $X$ the set of isomorphism classes $[E,O, \omega]$ of triples $(E,O,\omega)$, where:
\begin{itemize}
\item $(E,O)$ is an elliptic curve over $X$, as above;
\item $\omega\in\Omega^1_{E/X}(E)$ is a nonzero differential.
\end{itemize}
Over $S = \Spec \mathbb{Z}[1/6]$, this functor is representable by the affine scheme
\begin{equation*}
\mathcal{M}_{1,\vec{1}}: = \Spec \mathbb{Z}[1/6][u,v,\Delta^{-1}], \quad \text{where } \Delta = u^3 - 27 v^2,
\end{equation*}
because every isomorphism class $[E,O,\omega]\in F (X)$ can be locally represented over $X$ by a Weierstrass equation
\begin{equation*}
E: y^2 = 4 x^3 - g_2 x - g_3
\end{equation*}
where $g_2, g_3$ are local sections of $\mathcal{O}_X$ with $g_2^3 - 27 v^2$ invertible. This Weierstrass equation determines a morphism $X \to \mathcal{M}_{1,\vec{1}}$ via $u \mapsto g_2, v\mapsto g_3$; it also determines $\omega = dx/y$ and $O$ uniquely  \cite[\S $8.3$]{hainmoduli}. 

There is a natural embedding $\mathcal{M}_{1,\vec{1}}\hookrightarrow \mathbb{A}^2$. The multiplicative group $\mathbb{G}_m$ acts on $\mathbb{A}^2$ via $\lambda \cdot (u,v) = (\lambda^4 u, \lambda^6 v)$. The formula $\lambda \cdot \Delta = \lambda^{12} \Delta$ implies that the $\mathbb{G}_m$-action restricts to an action on $\mathbb{A}^2\backslash V(\Delta) = \mathcal{M}_{1,\vec{1}}$.

It follows that $\M11$ is the stack quotient $\M11\cong [\mathbb{G}_m\backslash \mathcal{M}_{1,\vec{1}}]$ \cite[Remark $8.5$]{hainmoduli}. The quotient morphism $\mathcal{M}_{1,\vec{1}}\to \M11$ is the map $[E,O,\omega]\mapsto [E,O]$ and is evidently a principal $\mathbb{G}_m$-bundle.

\begin{remark}
The notation $\mathcal{M}_{1,\vec{1}}$ is suggestive of the fact that this is also the moduli space of isomorphism classes of pairs $((E,O), \vec{v})$, where $(E,O)$ is as before and $\vec{v}$ is a nonzero tangential basepoint at $O$.
\end{remark}

\subsubsection{Analytic defintion of the moduli spaces $\M11$ and $\mathcal{M}_{1,\vec{1}}$}

To study periods it is more useful to have an analytic description of the associated analytic spaces $\mathcal{M}_{1,\vec{1}}^{an}$ and $\mathcal{M}_{1,1}^{an}$. They are examples of complex-analytic orbifolds \cite{hainmoduli}. As we generally work with moduli spaces as orbifolds rather than as algebraic stacks, we abuse notation and from this point forward use the same symbols to denote the stacks and their associated orbifolds.

Let $\mathfrak{H} = \left\{z\in\mathbb{C}:\Im(z)>0\right\}$. The modular group $\Gamma:=SL_2(\mathbb{Z})$ acts on $\mathfrak{H}$ by
\begin{equation} \label{mobiusaction}
\gamma\cdot \tau := \frac{a\tau +b}{c\tau+d} \quad \text{where }\gamma = \left(\begin{matrix} a & b \\ c & d \end{matrix}\right), \quad \tau\in\mathfrak{H}.
\end{equation}

The analytic moduli space of elliptic curves is the orbifold quotient $\mathcal{M}_{1,1}:=[\Gamma \backslash \mathfrak{H}]$ under the action \eqref{mobiusaction}. Define a left-action of $\Gamma$ on $\mathbb{C}^\times \times \mathfrak{H}$ by
\begin{equation} \label{bundleaction}
\gamma \cdot  (\xi,\tau) = \left((c\tau+d)^{-1}\xi, \gamma \cdot \tau \right),
\end{equation}
where $\gamma\in\Gamma$ acts on $\tau\in\mathfrak{H}$ as in \eqref{mobiusaction}. Then $\mathcal{M}_{1,\vec{1}}$ is the orbifold quotient
\begin{equation*}
\mathcal{M}_{1,\vec{1}} := [\Gamma \backslash (\mathbb{C}^\times \times \mathfrak{H})].
\end{equation*}
The action \eqref{bundleaction} has no fixed points, so $\mathcal{M}_{1,\vec{1}}$ is an analytic variety \cite[\S $14$]{hainhodge}.

Projection onto the second factor induces a morphism $\mathcal{M}_{1,\vec{1}}\to \M11$. This is a $\mathbb{C}^\times$-bundle over $\M11$; let $\mathcal{L}$ denote the associated analytic line bundle over $\M11$. The global sections of $\mathcal{L}^{\otimes k}$ are modular forms of weight $k$.

The moduli space $\M11$ has a Deligne-Mumford compactification $\overline{\mathcal{M}}_{1,1}$. This corresponds to compactifying the orbifold $\M11$ by patching basic orbifolds together to include the \emph{cusp} of $\M11$ \cite[\S $4$]{hainmoduli}.

\subsubsection{The universal elliptic curve} \label{universalell}

The identity $\id\in \Hom(\M11,\M11)$ corresponds to a universal family $\mathcal{E}\to \M11$ with a canonical section $O \colon \M11\to\mathcal{E}$. The universal punctured elliptic curve is $\mathcal{E}^\times := \mathcal{E}\backslash \left\{O\right\}$. An analytic description is given in \cite[\S $1$]{hainkzb}.

For any $\tau\in\M11$, the fiber $\mathcal{E}_\tau$ is isomorphic to the elliptic curve $E_\tau$ whose complex points are isomorphic to the quotient $\mathbb{C}/\Lambda_\tau$, where $\Lambda_\tau = \mathbb{Z}\oplus\mathbb{Z}\tau$. The fiber $\mathcal{E}^\times_\tau$ is isomorphic to the affine curve $E_\tau^{\times}$. The moduli space $\mathcal{M}_{1,\vec{1}}$ is the normal bundle of the image of $O \colon \M11\to \mathcal{E}$.

It is possible to extend $\mathcal{E}\to\M11$ over $\overline{\mathcal{M}}_{1,1}$ to obtain a generalised universal elliptic curve $\overline{\mathcal{E}}\to \overline{\mathcal{M}}_{1,1}$. The fiber over the cusp of $\overline{\mathcal{M}}_{1,1}$ is the nodal cubic.

\subsubsection{The Tate curve}

The Tate curve is the elliptic curve over $\mathbb{Q}((q))$ whose affine part is given by the equation
\begin{equation*}
y^2 = 4x^3 -g_2(q) x - g_3(q),
\end{equation*}
where $ g_2(q) := 20 \mathbb{G}_4(q)$, $g_3(q) := \frac{7}{3}\mathbb{G}_6(q)$ and $\mathbb{G}_{2k} (q) \in \mathbb{Q}[[q]]$ is the $q$-expansion of the Hecke-normalised Eisenstein series of weight $2k$ defined in \eqref{Eisdef}, considered as a formal power series  \cite[Chapter V, \S $3$]{silverman}.

Following Nakamura \cite[Example $4$]{nakamuratangential} we note that the Tate curve over $\mathbb{Q}((q))$ is equivalent to a $\mathbb{Q}$-rational tangential basepoint
\begin{equation} \label{tatecurvebasepoint}
\vec{1}_\infty \colon \Spec \mathbb{Q}((q))\to \M11.
\end{equation}

Hain provides a topological model for the infinitesimal punctured Tate curve in \cite[\S $16$]{hainkzb}. It is more tractable to work with from the point of view of fundamental groups, although it should be noted that the resulting space is not algebraic. Following Hain, we identify the tangential basepoint \eqref{tatecurvebasepoint} with the positive unit tangent vector at the cusp. A local coordinate at the cusp is $q = \exp(2\pi i \tau)$, and $\vec{1}_\infty = \partial/\partial q$. The infinitesimal punctured Tate curve is the fiber $Y := \Tatecurve$ of $\mathcal{E}^\times\to\M11$ over the tangent vector $\partial/\partial q$.

By equipping $\Tatecurve$ with the tangential basepoint $\vec{1}_O$ at the punctured origin $O$, we determine a trivialisation for $\Omega^1_{\mathcal{E}_{\partial/\partial q}}$ and hence a tangential basepoint on $\mathcal{M}_{1,\vec{1}}$. We continue to denote this tangential basepoint by $\vec{1}_\infty:=[\mathcal{E}_{\partial/\partial q}, \vec{1}_O]$. It is a lift of the tangential basepoint $\vec{1}_\infty$ on $\M11$ under the morphism $\mathcal{M}_{1,\vec{1}}\to\M11$.

\subsection{Motivic periods} \label{mpsection}

In this section we recall basic background about motivic periods, with particular focus on periods of mixed Tate motives. The reader is referred to \cite{motivicperiods} for more detail.

\subsubsection{Tannakian theory of motivic periods} \label{tannakianperiods}

Let $\MT$ be the category of mixed Tate motives over $\mathbb{Z}$ \cite{delignegoncharov}, and let $\mathcal{H}$ be the category of ``Hodge-theoretic triples'' \cite[\S $1$]{deligne}\cite[\S $3.1$]{motivicperiods}. These categories are equipped with Betti and de Rham fiber functors
\begin{equation*}
\omega^{\B}, \omega^{\dR} \colon\MT\to\Vect_{\mathbb{Q}}, \quad \omega_{\mathcal{H}}^{\B}, \omega_{\mathcal{H}}^{\dR} \colon \mathcal{H}\to\Vect_{\mathbb{Q}},
\end{equation*}
with respect to which they are neutral Tannakian over $\mathbb{Q}$. There is also a $\mathcal{H}$-realisation functor $\omega^{\mathcal{H}} \colon \MT\to\mathcal{H}$ satisfying $\omega^\bullet = \omega_{\mathcal{H}}^\bullet \circ \omega^{\mathcal{H}}$ for $\bullet \in\left\{\B, \dR\right\}$. It is a deep result that $\omega^{\mathcal{H}}$ is fully faithful \cite{delignegoncharov}.

To each of these categories one may associate a $\mathbb{Q}$-algebra of motivic periods, denoted $\mathcal{P}^{\mm}_{\MT}$ and $\mathcal{P}^{\mm}_{\mathcal{H}}$ respectively. A \emph{motivic} $\mathcal{H}$-period is an element of $\mathcal{P}^{\mm}_{\mathcal{H}}$ arising from the cohomology of an algebraic variety \cite[\S $3.5$]{motivicperiods}. Since $\omega^{\mathcal{H}}$ is fully faithful, there is an inclusion
\begin{equation} \label{motivicperiodinclusion}
\mathcal{P}^{\mm}_{\MT}\hookrightarrow \mathcal{P}^{\mm}_{\mathcal{H}}.
\end{equation}
The period algebras are equipped with actions of their respective motivic Galois groups
\begin{equation*}
G_{\MT}^{\dR}:=\Aut_{\MT}^\otimes (\omega^{\dR}), \quad G_{\mathcal{H}}^{\dR}:=\Aut_{\mathcal{H}}^\otimes (\omega_{\mathcal{H}}^{\dR}).
\end{equation*}
They are also equipped with canonical period maps
\begin{equation*}
\per \colon \mathcal{P}^{\mm}_{\MT}\to\mathbb{C}, \quad \per \colon \mathcal{P}^{\mm}_{\mathcal{H}}\to\mathbb{C}.
\end{equation*}
All of these features are compatible with the inclusion \eqref{motivicperiodinclusion}.

\begin{remark} \label{motivicperiodinclusionremark}
The embedding \eqref{motivicperiodinclusion} means that $\mathcal{P}^{\mm}_{\MT}$, together with its Galois action, may be studied within the algebra of motivic $\mathcal{H}$-periods, which has an elementary description in terms of \emph{matrix coefficients} \cite[\S $2.2$]{motivicperiods}. For this reason, all motivic periods will be viewed as contained in $\mathcal{P}^{\mm}_{\mathcal{H}}$, and when considering periods of mixed Tate motives we use \eqref{motivicperiodinclusion} implicitly.
\end{remark}

\subsubsection{The universal comparison isomorphism} \label{unicompsection}

The identity map on $\mathcal{P}^{\mm}_{\mathcal{H}}$ can be viewed as a canonical point $c^{\mm}\in \Isom_{\mathcal{H}}^{\otimes} (\omega^{\dR}_{\mathcal{H}}, \omega_{\mathcal{H}}^{\B})(\mathcal{P}^{\mm}_{\mathcal{H}})$ i.e. a natural isomorphism
\begin{equation*}
c^{\mm} \colon  \omega^{\dR}_{\mathcal{H}} \otimes \mathcal{P}^{\mm}_{\mathcal{H}} \xrightarrow{\sim} \omega^{\B}_{\mathcal{H}} \otimes \mathcal{P}^{\mm}_{\mathcal{H}}.
\end{equation*}
The component at $M = (M_{\B}, M_{\dR}, c_M)\in\mathcal{H}$ gives a \emph{universal} comparison isomorphism\footnote{If $M\in\MT$, this restricts to an isomorphism with $\mathcal{P}^{\mm}_{\MT}$ replacing $\mathcal{P}^{\mm}_{\mathcal{H}}$.}
\begin{equation*}
c_M^{\mm} \colon M_{\dR}\otimes_{\mathbb{Q}}\mathcal{P}^{\mm}_{\mathcal{H}} \xrightarrow{\sim}M_{\B}\otimes_{\mathbb{Q}}\mathcal{P}^{\mm}_{\mathcal{H}}.
\end{equation*}
It satisfies $(\id\otimes \per)\circ c_M^{\mm} = c_M \circ (\id\otimes \per)$.

\subsection{Mixed Tate periods}

\subsubsection{The Lefschetz period}

The most basic example of a motivic period is the Lefschetz period $\mathbb{L}$. It satisfies $\per(\mathbb{L}) = 2\pi i$. We recall the definition here as it is widely used in this paper.

\begin{definition}[Lefschetz period]
The Lefschetz period is
\begin{equation*}
\mathbb{L} = \mathbb{L}^{\mm} := [H^1 (\mathbb{G}_m), \left[dz/z\right], \gamma_0]^{\mm},
\end{equation*}
where $H^1 (\mathbb{G}_m)\cong\mathbb{Q}(-1)\in \MT$ and $\gamma_0\in H_1(\mathbb{C}^\times, \mathbb{Q})\cong H^1_{\B} (\mathbb{G}_m)^\vee$ is the class of a small counterclockwise loop circling the origin in $\mathbb{C}^\times$.
\end{definition}

\subsubsection{Motivic multiple zeta values} \label{motivicmzvsection}

Let $X = \P1minus$, and let 
\begin{equation*}
_0 \Pi^{\mot}_1 := \pi_1^{\mot} (X, \vec{1}_0, -\vec{1}_1)
\end{equation*}
denote the motivic path torsor between the tangential basepoints $\vec{1}_0$ and $\vec{1}_1$. Its affine ring $\mathcal{O} ({_0 \Pi^{\mot}_1})$ is an ind-object of $\MT$ \cite{deligne}. Using $\omega^{\mathcal{H}} \colon \MT \hookrightarrow \mathcal{H}$ it may be considered as an ind-object of $\mathcal{H}$ via
\begin{equation*}
\mathcal{O} ({_0 \Pi_1^{\mot}}) = (\mathcal{O} ({_0 \Pi_1^{\B}}), \mathcal{O} ({_0 \Pi_1^{\dR}}), c).
\end{equation*}
Let $\dch \in{ _0 \Pi_1^{\B}} (\mathbb{R})$ denote the straight line path from $\vec{1}_0$ to $-\vec{1}_1$ (see \S \ref{motivicdrinfeldassociators}).

\begin{definition}[Motivic multiple zeta value]
The $\mathbb{Q}$-algebra of motivic multiple zeta values is the ring $\Zm$ generated by matrix coefficients $[\mathcal{O}({_0 \Pi^{\mot}_1}), w, \dch]^{\mm}$, where $w$ ranges over elements of $\mathcal{O}({_0 \Pi^{\dR}_1})$.\footnote{The ind-object $\mathcal{O}({_0 \Pi^{\mot}_1})$ has a weight filtration by finite-dimensional subobjects $M_r \mathcal{O}({_0 \Pi^{\mot}_1})\in \MT$. A motivic MZV $\zeta^{\mm}(w)$ of \emph{weight} $n$ can be written as the matrix coefficient $[M_n\mathcal{O}({_0 \Pi^{\mot}_1}), w, \dch]^{\mm}$.}
\end{definition}

There is a canonical isomorphism $\mathcal{O}({_0 \Pi^{\dR}_1}) \cong \mathbb{Q}\langle e_0, e_1\rangle$ with the shuffle algebra on two symbols \cite{delignegoncharov}. Here $e_0$ and $e_1$ are formal symbols corresponding to $\omega_0 := dz/z, \omega_1 := dz/(1-z) \in \Omega^1 (X)$ respectively. Their classes span $H^1_{\dR}(X)$.

\begin{remark}[Admissible words] \label{admissibledef}
Let $w\in\mathbb{Q}\langle e_0, e_1 \rangle$ be a single word. The motivic multiple zeta value  $[\mathcal{O}(_0 \Pi^{\mot}_1), w, \dch]^{\mm}\in \Zm$ is denoted $\zeta^{\mm}(w)$. If $w$ is \emph{admissible}, meaning it is of the form $w = e_1 v e_0$ for some word $v$, then we may write
\begin{equation} \label{admissibleword}
w = e_1 e_0^{k_1 - 1} \cdots e_1 e_0^{k_r - 1} \quad \text{with } k_1 \dots k_{r-1}\geq 1 \text{ and } k_r\geq2.
\end{equation}
We use the notation $\zeta^{\mm} (k_1, \dots, k_r) := \zeta^{\mm} (w)$ for an admissible motivic MZV.
\end{remark}

\begin{remark}[Shuffle-regularisation]
There is a unique shuffle-algebra homomorphism $\zeta^{\mm} \colon  \mathbb{Q}\langle e_0, e_1\rangle \to \mathcal{P}^{\mm}_{\mathcal{H}}$, such that whenever $w$ is admissible of the form \eqref{admissibleword} we have $\zeta^{\mm}(w) = \zeta^{\mm}(k_1, \dots, k_r)$, and such that $\zeta^{\mm}(e_0) = \zeta^{\mm} (e_1) = 0$ \cite[Proposition $1.173$]{fresan}. We have $\per \zeta^{\mm} (w) = \zeta(w)$.
\end{remark}

\subsubsection{Brown's theorem} \label{browntheorem}

Brown proved that $G^{\dR}_{\MT}$ acts faithfully on $\mathcal{O}( {_0 \Pi^{\mot}_1})$ \cite{brownmtm}. This is equivalent to the statement that the Tannakian subcategory of $\MT$ generated by the ind-object $\mathcal{O}( {_0 \Pi^{\mot}_1})$ is equivalent to $\MT$, and implies that there is an isomorphism $\Zm[\mathbb{L}^{-1}] \xrightarrow{\sim} \mathcal{P}^{\mm}_{\MT}$.

\section{Fundamental groups} \label{fgsection}

For us, fundamental groups are a convenient tool for packaging and manipulating generating series of periods. The relevant fundamental groups for this purpose are Betti and de Rham fundamental groups, as well as the usual topological fundamental group. For technical reasons, we will also have to consider \emph{relative} versions of these groups. These notions are defined and outlined in the next sections.

\subsection{The topological fundamental group} \label{topfg}

Let $X$ be a scheme\footnote{When $X$ is an algebraic stack, $X(\mathbb{C})$ is a complex-analytic orbifold and the fundamental group here refers to the \emph{orbifold fundamental group} \cite{hainmoduli}.} and let $x\in X(\mathbb{C})_{\bp}$. The topological fundamental group of $X$, based at $x$, is
\begin{equation*}
\pi_1^{\text{top}}(X,x):=\pi_1(X(\mathbb{C}),x).
\end{equation*}

\begin{table}[ht]
\centering
\caption{Description of fundamental groups}\label{fg}
\renewcommand\arraystretch{1.5}
\noindent\[
\begin{tabularx}{\linewidth}{|c|c|L|} 
\hline
Space & Fundamental group & Description of generators \\
\hline
$\P1minus$&$\langle \gamma_0, \gamma_1, \gamma_\infty \vert \gamma_0\gamma_1\gamma_\infty = 1\rangle$&$\gamma_p =$ small positively oriented loop around puncture at $p$\\
$E^\times$ & $\langle\alpha, \beta\rangle$ & Cycles on punctured torus\\
$\M11$ & $SL_2(\mathbb{Z}) $ &Orientation-preserving automorphisms of $H_1 (E^\times(\mathbb{C}), \mathbb{Z})$ \\
$\mathcal{M}_{1,\vec{1}}$ & $B_3$ & Automorphisms of $\pi_1^{\text{top}} (E^\times)$ acting via Dehn twists along $\alpha, \beta$\\
\hline
\end{tabularx}
\]
\end{table}
We record several important fundamental groups in Table \ref{fg}. (To ease notation, the basepoint $x$ is omitted.) Here $E^\times$ denotes the fiber of $\mathcal{E}^\times\to\M11$ over any point in $\M11(\mathbb{C})_{\bp}$, and in particular may refer to the infinitesimal punctured Tate curve $Y = \Tatecurve$. The element $\alpha$ is the image under $\mathbb{C}\mapsto\mathbb{C}/\Lambda_\tau\cong E(\mathbb{C})$ of the path $\hat{\alpha} \colon (0,1)\mapsto \mathbb{C}, t\mapsto t$, and $\beta$ is the image of $\hat{\beta} \colon (0,1)\mapsto\mathbb{C}, t\mapsto t\tau$.

Under the isomorphism $ \pi_1^{\text{top}}(\M11, \vec{1}_\infty) \cong \Gamma$ the element $S$ corresponds to the imaginary axis on $\mathcal{H}$, or equivalently to a loop based at $\vec{1}_\infty$ encircling the point $i$.\footnote{This loop is not nullhomotopic because the stabiliser of $[i]\in\M11$ is $\mathbb{Z}/4\mathbb{Z}$.} The element $T$ corresponds to a small positively oriented loop around the cusp.

In Table \ref{fg} the braid group on $3$ strands has presentation
\begin{equation} \label{braidpresentation}
B_3 = \langle t_A, t_B  \vert  t_A t_B t_A = t_B t_A t_B \rangle.
\end{equation}
The elements $t_A$ and $t_B$ can be identified with Dehn twists on cycles $A, B \subseteq E(\mathbb{C})$ that intersect transversally at one point. One such choice is (the images of) $\alpha$ and $\beta$. This point of view realises $\pi_1^{\text{top}}(\mathcal{M}_{1,\vec{1}},x)$ as the mapping class group of a genus $1$ surface with one puncture. See \S \ref{topmonodromysection} for a detailed description of the monodromy action of this mapping class group.

\subsection{The relative fundamental group}

The relative completion of a discrete group \cite{hainmalcev} generalises the notion of unipotent completion. It can be used as a substitute in cases when the unipotent completion is trivial.

Here we use a geometric variant of relative completion, due to Brown \cite[\S $12$]{mmv}. The advantage of this method is that it is possible to define all realisations of the relative completion simultaneously from the following categorical construction.

\begin{definition}[Abstract relative completion] \label{abstractrelcompletion}
Let $(\mathsf{C},\omega)$ be a neutral Tannakian category over a field $K$ of characteristic zero, and let $\mathsf{S}\hookrightarrow \mathsf{C}$ be a full semisimple Tannakian subcategory with fiber functor $\omega\vert_\mathsf{S}$. Define a subcategory $\mathsf{F}(\mathsf{C},\mathsf{S}) \hookrightarrow \mathsf{C}$ whose objects are objects $V\in\mathsf{C}$ equipped with a (finite, exhaustive, separated) increasing filtration
\begin{equation*}
0 = V_0\subseteq V_1 \subseteq \cdots \subseteq V_n = V
\end{equation*}
by $\mathsf{C}$-subobjects of $V$, such that each graded quotient $V_{i+1}/V_i$ is isomorphic to an object in $\mathsf{S}$. The morphisms in $\mathsf{F}(\mathsf{C}, \mathsf{S})$ are the $\mathsf{C}$-morphisms compatible with the filtrations. Then $\mathsf{F}(\mathsf{C},\mathsf{S})$ is a Tannakian category equipped with the fiber functor $\omega \vert_{\mathsf{F}(\mathsf{C}, \mathsf{S})}$. We define the fundamental group of $(\mathsf{C},\omega)$ relative to $\mathsf{S}$ to be the Tannakian fundamental group
\begin{equation*}
\mathcal{G} = \pi_1(\mathsf{C},\mathsf{S}, \omega) := \underline{\Aut}^{\otimes}_{\mathsf{F}(\mathsf{C},\mathsf{S})} \left(\omega \vert_{\mathsf{F}(\mathsf{C}, \mathsf{S})}\right).
\end{equation*}
\end{definition}

Let $S := \underline{\Aut}^{\otimes}_{\mathsf{S}} (\omega\vert_{\mathsf{S}})$. It is pro-reductive because $\mathsf{S}$ is semisimple \cite[Proposition $2.23$]{delignemilne}. There is an inclusion $\mathsf{S}\hookrightarrow \mathsf{F}(\mathsf{C},\mathsf{S})$ given by equipping an object $V\in\mathsf{S}$ with the trivial filtration $0 = V_0\subseteq V_1 = V$. This inclusion induces a faithfully flat morphism of affine group schemes $\mathcal{G}\to S$ whose kernel is a pro-unipotent group $\mathcal{U}$, exhibiting $\mathcal{G}$ as an extension
\begin{equation} \label{abstractext}
1\to \mathcal{U}\to \mathcal{G} \to S \to 1.
\end{equation}

Let $X$ be a scheme\footnote{In one case we must consider the relative fundamental group of the \emph{algebraic stack} $X = \M11$. As described in \S \ref{modulialg}, this is the stack quotient $[\mathbb{G}_m \backslash \mathcal{M}_{1,\vec{1}}]$. We define its relative fundamental groups by looking at $\mathbb{G}_m$-equivariant objects of $\mathsf{LS}(\mathcal{M}_{1,\vec{1}})$ and $\Con (\mathcal{M}_{1, \vec{1}})$ and then applying the same Tannakian machinery as for schemes.} over a field $K\subseteq\mathbb{C}$, and let $x\in X(K)_{\bp}$. To this data we can associate two neutral Tannakian categories:
\begin{itemize}
\item The category $\Loc(X) = \Loc_K(X)$ of local systems $V$ of finite-dimensional $K$-vector spaces on $X$, equipped with the \emph{Betti} fiber functor $\omega_x^{\B}$ sending each local system $V$ to the stalk $V_x$.
\item The category $\Con(X) = \Con_K(X)$ of algebraic $K$-vector bundles with a flat connection $(\mathcal{V},\nabla)$ on $X$ and regular singularities at infinity, equipped with the \emph{de Rham} fiber functor $\omega_x^{\dR}$ sending $(\mathcal{V},\nabla)$ to the fiber $\mathcal{V}(x) = \mathcal{V}_x/\mathfrak{m}_x \mathcal{V}_x$.
\end{itemize}

Let $\mathsf{S}^{\B}\hookrightarrow \Loc(X)$ (resp. $\mathsf{S}^{\dR}\hookrightarrow \Con(X)$) be a full semisimple Tannakian subcategory of $\Loc(X)$ (resp. $\Con(X)$) with fiber functor given by the restriction of $\omega_x^{\B}$ (resp. $\omega_x^{\dR}$). Denote their Tannaka groups by $S^{\B}$ and $S^{\dR}$ respectively.

\begin{definition}[Relative fundamental group] \label{relfgdef}
Applying Definition \ref{abstractrelcompletion} to the above setup produces two affine group schemes over $K$,
\begin{equation}
\pi_1^{\rel,\B}(X,x):=\pi_1(\Loc(X),\mathsf{S}^{\B}, \omega^{\B}_x), \quad \pi_1^{\rel,\dR}(X,x):=\pi_1(\Con(X),\mathsf{S}^{\dR}, \omega^{\dR}_x),
\end{equation}
respectively called the \emph{Betti} and \emph{de Rham} relative fundamental groups of $X$ with basepoint $x$. These depend on the choice of the semisimple subcategories.
\end{definition}

\begin{remark}
A standard way to define the semistable subcategory $\mathsf{S}^{\B}$ (resp. $\mathsf{S}^{\dR}$) is by taking the semistable Tannakian subcategory generated by an object $V \in \mathsf{LS}(X)$ (resp. $\mathcal{V} \in \Con(X)$) inducing a polarised variation of MHS over $X$ \cite{hainmalcev}. In this case the associated reductive group is the symplectic group $S^{\B} = Sp (V_x)$ (resp. $S^{\dR} = Sp(\mathcal{V}(x))$), considered as an affine group scheme over $K$.
\end{remark}

\subsubsection{Splitting}

For relative completions of fundamental groups, the exact sequence \eqref{abstractext} always splits \cite[Proposition $3.1$]{hainhodge}. This means that for $\bullet\in\left\{\B,\dR\right\}$ the relative fundamental group is isomorphic to a semidirect product
\begin{equation} \label{relativesplitting}
\pi_1^{\rel, \bullet}(X,x) \cong \mathcal{U}^\bullet\rtimes S^{\bullet},
\end{equation}
where $S^\bullet$ is pro-reductive and $\mathcal{U}^\bullet$ is pro-unipotent. This decomposition depends upon the choice of splitting.

\subsubsection{Comparison isomorphism} \label{complexcomparisoniso}

Let $X_\mathbb{C} = X\times_{K} \mathbb{C}$ and let $\hat{x}\in X_{\mathbb{C}}(\mathbb{C})$ be a lift of $x$. The Riemann-Hilbert correspondence induces a tensor-equivalence $\Con_{\mathbb{C}}(X_{\mathbb{C}}) \xrightarrow{\sim} \Loc_{\mathbb{C}}(X_\mathbb{C})$. Assume that this induces an equivalence $ \mathsf{S}^{\dR}_{\mathbb{C}} \xrightarrow{\sim} \mathsf{S}^{\B}_{\mathbb{C}}$. Tannakian duality yields a canonical isomorphism $\pi_1^{\rel, B}(X_{\mathbb{C}},\hat{x}) \xrightarrow{\sim} \pi_1^{\rel, \dR}(X_{\mathbb{C}},\hat{x})$. Relative completion commutes with base change \cite[\S $3.3$]{hainhodge} so we obtain a canonical isomorphism
\begin{equation} \label{comparisonisorel}
\pi_1^{\rel, B}(X,x)\times_{K}{\mathbb{C}} \xrightarrow{\sim} \pi_1^{\rel, \dR}(X,x)\times_{K}{\mathbb{C}}.
\end{equation}

\begin{definition}[Comparison isomorphism] \label{comparisoniso}
The isomorphism \eqref{comparisonisorel} is called the comparison isomorphism between (relative) Betti and de Rham fundamental groups.
\end{definition}

The isomorphism \eqref{comparisonisorel} is induced by viewing the affine rings as ind-objects in $\mathcal{H}\otimes_{\mathbb{Q}} K$, whose ring of motivic periods is $\mathcal{P}^{\mm}_{\mathcal{H}}\otimes K$. By \S \ref{unicompsection} it lifts to a universal comparison isomorphism
\begin{equation*}
\pi_1^{\rel, B}(X,x) \times_{K} ({\mathcal{P}^{\mm}_{\mathcal{H}}}\otimes_{\mathbb{Q}} K) \xrightarrow{\sim} \pi_1^{\rel, \dR}(X,x)\times_{K}({\mathcal{P}^{\mm}_{\mathcal{H}}}\otimes_{\mathbb{Q}} K).
\end{equation*}

\begin{remark}[Generating series]
The Betti fiber functor $\omega_x^{\B} \colon \Loc(X)\to\Vect_K$ factors through the category $\Rep_K(\pi_1^{\text{top}}(X,x))$ of $K$-representations of the topological fundamental group, and the augmented functor $\Loc(X)\to\Rep_K(\pi_1^{\text{top}}(X,x))$ is a Tannakian equivalence.\footnote{The category $\Rep_K(\pi_1^{\text{top}}(X,x))$ is equipped with the forgetful fiber functor.} This induces a canonical Zariski-dense homomorphism
\begin{equation} \label{canonicaldensehom}
\pi_1^{\text{top}}(X,x)\to\pi_1^{\rel,\B}(X,x)(K).
\end{equation}
Consider the composition
\begin{equation} \label{universaldrmap}
\pi_1^{\text{top}}(X,x)\to \pi_1^{\rel, \B}(X,x)(K)\to \pi_1^{\rel, \dR}(X,x)(\mathcal{P}^{\mm}_{\mathcal{H}}\otimes_{\mathbb{Q}} K)
\end{equation}
where the first map is \eqref{canonicaldensehom} and the second is induced by the universal comparison isomorphism. This composition can be viewed as the map sending an element $\gamma$ of the topological fundamental group to the generating series of all homotopy-invariant motivic iterated integrals on $X$ along $\gamma$.
\end{remark}

\subsubsection{Relation to unipotent completion}

Let $\bullet \in\left\{ \B, \dR\right\}$. When $\mathsf{S}^{\bullet}$ is a trivial Tannakian subcategory, the relative fundamental group $\pi_1^{\rel, \bullet}(X,x)$ is pro-unipotent. In this case we denote $\pi_1^{\rel, \bullet}(X,x)$ by $\pi_1^{\bullet}(X,x)$.

When $\bullet = \B$, $\pi_1^{\B}(X,x)$ is the unipotent completion of $\pi_1^{\text{top}}(X,x)$. The canonical homomorphism \eqref{canonicaldensehom} is none other than the natural map from $\pi_1^{\text{top}}(X,x)$ to the rational points of its unipotent completion.

Throughout this paper, we only make use of a nontrivial semisimple subcategory $\mathsf{S}^\bullet$ in the cases $(X,x) = (\M11, \vec{1}_\infty)$ or $(\mathcal{M}_{1,\vec{1}}, \vec{1}_\infty)$, in which case $\mathsf{S}^\bullet$ is generated by an explicit object depending on the universal family over $\M11$. This is because the unipotent completion of $\Gamma$ is trivial. In all other situations we only require relative completion with respect to a trivial $\mathsf{S}$.

\begin{remark}
If a discrete group $\Pi$ is the topological fundamental group of a space $X$, we generally write $\Pi^{\B}$ (resp. $\Pi^{\dR}$) for the associated Betti (resp. de Rham) fundamental group.
\end{remark}

\subsection{De Rham fundamental groups of $\P1minus$ and $\Tatecurve$}

In this section we describe the de Rham fundamental groups of $X=\P1minus$ and $Y=\Tatecurve$ in more detail.

\subsubsection{$R$-points}

The groups $\Pi_X := \pi_1^{\text{top}}(X,\vec{1}_1)$ and $\Pi_Y := \pi_1^{\text{top}}(Y,\vec{1}_O)$ are free on two generators (see \S  \ref{topfg}); we have $\Pi_X = \langle \gamma_0, \gamma_1\rangle$ and $\Pi_Y = \langle \alpha, \beta\rangle$. Let $R$ be a $\mathbb{Q}$-algebra, and consider the Hopf algebras of formal noncommutative power series $R\langle\langle \mathsf{x}_0, \mathsf{x}_1 \rangle\rangle$, $R\langle\langle \adr, \bdr \rangle\rangle$ defined in \S \ref{hopfalgnoncomm}. The $R$-points of the de Rham fundamental groups may be described as the groups of grouplike elements
\begin{equation} \label{drpoints}
\Pi_X^{\dR}(R) \cong \GrL(R\langle\langle \mathsf{x}_0, \mathsf{x}_1\rangle\rangle ), \quad\Pi_Y^{\dR}(R) \cong \GrL(R\langle\langle \adr, \bdr\rangle\rangle ).
\end{equation}
The elements $\mathsf{x}_p, \adr, \bdr$ correspond to the abelianisations of logarithms of the associated elements $\gamma_p, \alpha, \beta$ up to a normalisation coming from the comparison isomorphism. In other words, under the natural map $\Pi_X \to (\Pi_X^{\dR})^{\ab} (\mathcal{P}^{\mm}_{\MT})$ we have $\gamma_p \mapsto \exp(\mathbb{L} \mathsf{x}_p)$. Under $\Pi_Y \to (\Pi_Y^{\dR})^{\ab} (\mathcal{P}^{\mm}_{\mathcal{H}})$ we have $\alpha \mapsto \exp(\mathbb{L}\adr), \beta\mapsto \exp(-\bdr)$. The coefficients reflect the MHS on the fundamental group \cite{haingeometry}; $\mathsf{x}_p$ spans a copy of $\mathbb{Q}(1)$, $\adr$ spans $\mathbb{Q}(1)$ and $\bdr$ spans $\mathbb{Q}(0)$.

\subsubsection{Lie algebras}

The Lie algebras of $\Pi_X^{\dR}$ and $\Pi_Y^{\dR}$ are non-canonically isomorphic to the completed free Lie algebras $\Lie(\mathsf{x}_0, \mathsf{x}_1)^\wedge$ and $\Lie(\adr,\bdr)^\wedge$ respectively.

\subsubsection{Motivic Drinfeld associators} \label{motivicdrinfeldassociators}

There are six tangential basepoints on $X$ that have good reduction modulo every prime: each puncture $p\in\left\{0,1,\infty\right\}$ is equipped with two unit tangent vectors $\pm\vec{1}_p$ based at $p$.

Let $i,j\in\left\{0,1,\infty\right\}$ be distinct. There is a natural ``straight line path''
\begin{equation*}
\dch_{i j}\in {_i \Pi_j} :=\pi_1^{\text{top}}(X,\vec{1}_i, -\vec{1}_j)
\end{equation*}
whose image is contained within $X(\mathbb{R})$. Under the map ${_i \Pi_j}\to {_i \Pi_j^{\dR}}(\mathcal{P}^{\mm}_{\MT})$ described in \eqref{universaldrmap}, the element $\dch_{i j}$ is sent to an element $\Phi_{i j}^{\mm}$. By \eqref{drpoints}, this element is a grouplike power series $\Phi_{i j}^{\mm} \in\mathcal{P}^{\mm}_{\MT} \langle\langle \mathsf{x}_0, \mathsf{x}_1 \rangle\rangle$.\footnote{Technically, $\Phi_{i j}^{\mm}$ is an element of a de Rham path torsor. However the canonical de Rham splitting for mixed Tate motives implies that the basepoint plays no role in $\pi_1^{\dR} (X)$.
\par
Here and onward we also make an obvious abuse of notation and write $\zeta^{\mm}(w)$ for $w\in M(\mathsf{x}_0, \mathsf{x}_1)$, even though $\zeta^{\mm}(w)$ is strictly defined for $w\in M(e_0, e_1)$.} As in the case for the real Drinfeld associator $\Phi_{0 1}\in\mathbb{R}\langle\langle \mathsf{x}_0, \mathsf{x}_1\rangle\rangle$, $\Phi_{0 1}^{\mm}$ is the generating series for motivic multiple zeta values
\begin{equation} \label{drinfeldgeneratingseries}
\Phi_{0 1}^{\mm} = \Phi^{\mm}(\mathsf{x}_0, \mathsf{x}_1) = \sum_{w\in M(\mathsf{x}_0, \mathsf{x}_1)} \zeta^{\mm}(w) w
\end{equation}
and satisfies $\per(\Phi_{0 1}^{\mm}) = \Phi_{0 1}$. The other series $\Phi_{i j}^{\mm}$ are obtained making by the change of variables $\Phi^{\mm}_{i j} := \Phi^{\mm} (\mathsf{x}_i, \mathsf{x}_j)$. Here we have defined $\mathsf{x}_\infty := - (\mathsf{x}_0 + \mathsf{x}_1)$.

\subsection{Relative fundamental groups of $\M11$ and $\mathcal{M}_{1,\vec{1}}$} \label{specificfundamentalgroups}

Let $\pi \colon \mathcal{E}\to\M11$ be the universal family of \S \ref{universalell}. Let $V:= R^1 \pi_* \mathbb{Q}$ be the local system over $\M11$ whose fiber over $\tau$ is the Betti cohomology $H_{\B}^1 (E_\tau, \mathbb{Q})$. Associated to $V$ is a vector bundle $\mathcal{V}$ over $\M11$ equipped with the Gauss-Manin connection $\nabla$. 

Let $\mathsf{S}^{\B}$ be the full semisimple subcategory of $\Loc_{\mathbb{Q}} (\M11)$ generated by $V$, and let $\mathsf{S}^{\dR}$ be the full subcategory of $\Con_{\mathbb{Q}}(\M11)$ generated by $(\mathcal{V}, \nabla)$. Definition \ref{relfgdef} yields the Betti and de Rham relative fundamental groups with respect to $\mathsf{S}^{\B}$, $\mathsf{S}^{\dR}$:
\begin{equation*}
\mathcal{G}_{1,1}^{\B} :=\pi_1^{\rel,\B} (\M11,\vec{1}_\infty), \quad \mathcal{G}_{1,1}^{\dR} :=\pi_1^{\rel,\dR} (\M11,\vec{1}_\infty).
\end{equation*}
Let $V'$ (resp. $\mathcal{V}'$) denote the pullback of $V$ (resp. $\mathcal{V}$) by $\mathcal{M}_{1,\vec{1}}\to \M11$. The same procedure defines the relative fundamental groups of $\mathcal{M}_{1,\vec{1}}$ with respect to $V'$ (resp. $\mathcal{V}'$) to give
\begin{equation*}
\mathcal{G}_{1,\vec{1}}^{\B} :=\pi_1^{\rel,\B} (\mathcal{M}_{1,\vec{1}} ,\vec{1}_\infty), \quad \mathcal{G}_{1,\vec{1}}^{\dR} :=\pi_1^{\rel,\dR} (\mathcal{M}_{1,\vec{1}} ,\vec{1}_\infty).
\end{equation*}
The decomposition \eqref{relativesplitting} produces non-canonical isomorphisms
\begin{equation*}
\mathcal{G}_{1,1}^{\bullet} \cong \mathcal{U}_{1,1}^{\bullet}\rtimes SL_2, \quad \mathcal{G}_{1,\vec{1}}^{\bullet} \cong \mathcal{U}_{1,\vec{1}}^{\bullet}\rtimes SL_2,
\end{equation*}
where $\mathcal{U}_{1,1}^{\bullet}$ and $\mathcal{U}_{1,\vec{1}}^{\bullet}$ are pro-unipotent. By \cite[Proposition $14.2$]{hainhodge} we also have
\begin{equation} \label{vectordecomp}
\mathcal{U}_{1,\vec{1}}^{\dR}\cong \mathcal{U}_{1,1}^{\dR}\times\mathbb{G}_a(1).
\end{equation}
In order to describe $\mathcal{G}_{1,1}^{\dR}$ and $\mathcal{G}_{1,\vec{1}}^{\dR}$ it is therefore sufficient to describe $\mathfrak{u}_{1,1} :=\Lie(\mathcal{U}_{1,1}^{\dR})$, together with its $SL_2$-action. The Lie algebra $\mathfrak{u}_{1,\vec{1}} := \Lie(\mathcal{U}_{1,\vec{1}}^{\dR})$ is isomorphic to the direct product $\mathfrak{u}_{1,1} \oplus \mathbb{Q}\mathbf{e}_2$ with a central generator $\mathbf{e}_2$ spanning a copy of $\mathbb{Q}(1)$. This generator is acted on trivially by $SL_2$.

\subsubsection{Explicit description of generators}

It remains to describe $\mathfrak{u}_{1,1}$. It is equipped with a limiting mixed Hodge structure (MHS) \cite{hainhodge}, and therefore has two different weight filtrations, $M$ and $W$. These filtrations will not be used in the sequel, but note that the filtrations on $\ugeom$ defined in \S  \ref{ugeomsection} are their image under the monodromy morphism of \S \ref{monodromysection}. The Lie algebra $\mathfrak{u}_{1,1}$, together with its MHS, is described explicitly in \cite[\S $13.5$]{mmv}. It is completed free Lie algebra on its abelianisation, which is isomorphic to
\begin{equation*}
\prod_{n\geq 2} H^1_{\dR}(\M11, \Sym^n \mathcal{V})^\vee\otimes V_n^{\dR}.
\end{equation*}
By describing $H^1_{\dR} (\M11, \Sym^n \mathcal{V})$ \cite{mmv, brownhain2018}, Brown shows that $\mathfrak{u}_{1,1}$ is the completed free Lie algebra generated by terms of the form
\begin{equation} \label{liealgconcrete}
\mathbf{e}_{2n+2} \mathsf{X}^k \mathsf{Y}^{2n-k}, \quad  \mathbf{e}'_f \mathsf{X}^k\mathsf{Y}^{2n-k}, \quad \mathbf{e}''_f \mathsf{X}^k\mathsf{Y}^{2n-k}
\end{equation}
where $n\geq 1$, $0\leq k\leq 2n$ and $f$ ranges over normalised Hecke eigenforms of weight $2n+2$.

The isomorphism $\mathfrak{u}_{1,\vec{1}}\cong \mathfrak{u}_{1,1}\oplus{\mathbb{Q} \mathbf{e}_2}$ implies that $\mathfrak{u}_{1,\vec{1}}$ is the completed free Lie algebra generated by the terms in \eqref{liealgconcrete}, where we also allow $n\geq 0$. An element of $\mathcal{U}^{\dR}_{1,\vec{1}}(R)$ can be written in the form $u \cdot \exp(r\mathbf{e}_2)$, where $u\in \mathcal{U}^{\dR}_{1,1}(R)$ and $r\in R$.

The symbols $\mathbf{e}_{2n+2}$ each span a copy of $\mathbb{Q}(1)$, while $\mathbf{e}_f',\mathbf{e}_f''$ is a basis for a rank-$2$ mixed Hodge structure associated to the motive of $f$ \cite{scholl}. The symbols $\mathsf{X}$ and $\mathsf{Y}$ span copies of $\mathbb{Q}(0)$ and $\mathbb{Q}(1)$ respectively.

\subsubsection{Left and right actions} \label{leftrightactions}

The reductive group $SL_2^{\dR}$ acts on $\mathfrak{u}_{1,1}$ naturally on the \emph{right} as follows: let $R$ be a $\mathbb{Q}$-algebra and let
\begin{equation*}
\gamma = \left(\begin{matrix} a & b \\ c & d \end{matrix} \right) \in SL_2^{\dR}(R).
\end{equation*}
Then $\gamma$ acts on $\mathfrak{u}_{1,1}\otimes_{\mathbb{Q}} R$ via
\begin{equation} \label{rightSL2action}
P(\mathsf{X},\mathsf{Y})\vert_\gamma := P(a \mathsf{X} + b\mathsf{Y},c\mathsf{X}+d\mathsf{Y}).
\end{equation}
It also acts on $\Pi_Y^{\dR}:=\pi_1^{\dR}(\Tatecurve, \vec{1}_O)$ via right multiplication on the frame $(\adr, \bdr)$:
\begin{equation} \label{SL2actionPi}
(\adr, \bdr) \vert_\gamma := (\adr, \bdr)\gamma = (a \adr + c\bdr, b\adr + d\bdr)
\end{equation}

In contrast, the unipotent group $\mathcal{U}_{1, \vec{1}}^{\dR}$ acts on $\Pi_Y^{\dR}$ on the \emph{left}, via the monodromy homomorphism of \S \ref{monodromysection}, by applying certain derivations on $\Lie(\adr, \bdr)$ to power series in $\adr,\bdr$. These derivations are defined in \S  \ref{geometricderivationssection}.

It is convenient to consider only left actions, so we let $SL_2^{\dR}$ act on the \emph{left} by taking the inverse action of each right action. For $\gamma\in SL_2^{\dR}(R)$ and $x$ an $R$-point of either $\Pi_Y^{\dR}$, $\mathcal{U}_{1,\vec{1}}^{\dR}$ or $\mathcal{U}_{1,1}^{\dR}$, this is defined by $\gamma(x) = x\rvert_{\gamma^{-1}}$.

The group $\mathcal{G}_{1,\vec{1}}^{\dR}$ (resp. $\mathcal{G}_{1,1}^{\dR}$) splits non-canonically as the semidirect product
\begin{equation} \label{sdsplit}
\mathcal{G}_{1,\vec{1}}^{\dR} \cong  \mathcal{U}_{1,\vec{1}}^{\dR} \rtimes SL_2^{\dR} \quad \text{(resp. }\mathcal{G}_{1,1}^{\dR} \cong  \mathcal{U}_{1,1}^{\dR} \rtimes SL_2^{\dR} \text{)},
\end{equation}
where $SL_2^{\dR}$ acts on the left by the above inverse action.

\subsubsection{Induced $\mathfrak{sl}_2$-actions} \label{sl2action}

The $SL_2$-action on $\mathfrak{u}_{1,1}$ described in \S  \ref{leftrightactions} induces an action of $\Lie(SL_2) = \mathfrak{sl}_2$ on $\mathfrak{u}_{1,1}$. It acts via the operators $\mathsf{X}\partial/\partial\mathsf{Y}$ and $-\mathsf{Y}\partial/\partial\mathsf{X}$.

\section{The Lie algebra of geometric derivations} \label{geometricders}

\subsection{The derivations $\varepsilon_{2n+2}^\vee$ and $\varepsilon_{2n+2}$} \label{geometricderivationssection}

Let $L=\Lie(\adr,\bdr)$ denote the free Lie algebra on two letters $\adr$ and $\bdr$ over $\mathbb{Q}$ \cite{freeliealgebras}. Its lower central series completion $L^\wedge$ is canonically isomorphic to $\Lie(\Pi_Y^{\dR})$.

Let $\theta = [\adr,\bdr]$ and consider the Lie subalgebra $\Der^\theta (L)\subseteq \Der(L)$ consisting of derivations $\delta$ for which $\delta(\theta) = 0$. Such a derivation is determined by its value on either $\adr$ or $\bdr$. Within $\Der^\theta (L)$ there is a distinguished family of derivations $\varepsilon_{2n+2}^\vee$, for each $n\geq -1$ \cite{tsunogai}.

\begin{definition}[Geometric derivations] \label{epsderdef}
For $n\geq -1$ define $\varepsilon_{2n+2}^\vee\in \Der^\theta (L)$ by
\begin{align*}
\varepsilon_{2n+2}^\vee (\adr) &= \ad(\adr)^{2n+2} (\bdr) \\
\varepsilon_{2n+2}^\vee (\bdr) & = \frac{1}{2} \sum_{i+j = 2n+1} (-1)^i \left[\ad(\adr)^i (\bdr), \ad(\adr)^j (\bdr) \right].
\end{align*}
The derivation $\varepsilon_0^\vee$ can also be written as $\varepsilon_0^\vee = \bdr\partial/\partial \adr$.
\end{definition}
\subsubsection{Action of $SL_2$ and $\mathfrak{sl}_2$} \label{sl2actionugeom}

The algebraic group $SL_2$ acts on the right of $L$ by right-multiplying the row vector $(\adr,\bdr)$ (see \S \ref{leftrightactions}). For all $n\geq -1$ we define
\begin{equation} \label{dualderdef}
\varepsilon_{2n+2} := (-)\vert_{S} \circ \varepsilon_{2n+2}^\vee \circ (-)\vert_{ S^{-1}}.
\end{equation}
The element $\varepsilon_0$ may also be written as $\varepsilon_0 = -\adr\partial/\partial \bdr$.

There is an inner action of $\mathfrak{sl}_2$ on $\Der^\theta (L)$ by $\varepsilon_0^\vee$ and $\varepsilon_0$. The following proposition is \cite[Lemma $5.2$]{browneisenstein}.

\begin{proposition} \label{epsilonbasis}
The derivation $\varepsilon_{2n+2}^\vee$ is a highest-weight vector and generates an irreducible representation of $\mathfrak{sl}_2$ of dimension $2n+1$, with basis
\begin{equation*}
\left\{ \ad(\varepsilon_0^\vee)^k (\varepsilon_{2n+2}^\vee): 0\leq k\leq 2n\right\}.
\end{equation*}
\end{proposition}

For later use we note that the highest and lowest weight vectors are related by a simple formula.

\begin{lemma} \label{highestlowest}
The highest and lowest vectors are related via
\begin{equation} \label{highestlowestrelation}
\ad(\varepsilon_0^\vee)^k (\varepsilon_{2n+2}^\vee) = \frac{k!}{(2n-k)!} \ad(\varepsilon_0)^{2n-k}(\varepsilon_{2n+2}).
\end{equation}
\end{lemma}
\begin{proof}
By comparing $\mathfrak{sl}_2$-weights, one knows that the quantities $\ad(\varepsilon_0^\vee)^k (\varepsilon_{2n+2}^\vee)$ and $\ad(\varepsilon_0)^{2n-k}(\varepsilon_{2n+2})$ must be proportional. The scale factor can be found by applying both sides of \eqref{highestlowestrelation} to either $\adr$ and $\bdr$.
\end{proof}

\subsection{The Lie algebra $\ugeom$} \label{ugeomsection}

The geometric derivations assemble into a Lie subalgebra of $\Der^\theta (L)$.

\begin{definition} \label{ugeomdef}
Let $\ugeom$ denote the Lie subalgebra of $\Der^\theta (L)$ generated by the collection of all $\mathfrak{sl}_2$ representations in Proposition \ref{epsilonbasis}. It is a bigraded Lie subalgebra of $\Der^\theta(L)$ generated by $\ad(\varepsilon_0^\vee)^k (\varepsilon_{2n+2}^\vee)$ for $n\geq 1$ and $k\geq 0$, and carries a canonical $\mathfrak{sl}_2$-action via the adjoint action of $\varepsilon_0^\vee$ and $\varepsilon_0$.
\end{definition}

\begin{remark}
The Lie algebra $\ugeom$ is the graded Lie algebra of the image $\Ugeom = \im(\mu \colon \mathcal{U}_{1,1}^{\dR}\to \Aut(\Pi_Y^{\dR}))$ under the monodromy morphism of \S \ref{monodromysection}. The image is a pro-unipotent subgroup of $\Aut( \Pi_Y^{\dR})$. The inclusion $\ugeom\subseteq \Der^\theta (L)$ is a consequence of the fact that elements of $\Ugeom$ fix a small loop around the puncture on $\Tatecurve$, whose logarithm corresponds to $\theta$.

Note that $\ugeom$ does not contain $\varepsilon_2^\vee = \varepsilon_2 = -\ad([a,b])$. This element is central in $\Der^\theta (L)$. It corresponds to the logarithm of an element of $\mathcal{U}_{1,\vec{1}}^{\dR}$ that acts on $(\Tatecurve, \vec{1}_O)$ by rotating the tangent vector $\vec{1}_O$.
\end{remark}

\subsubsection{Mixed Hodge structure}

The fundamental group $\Pi_Y^{\dR}$ and its automorphism group are equipped with limiting mixed Hodge-Tate structures \cite[\S $13.6$]{mmv}. Therefore $\ugeom$ is equipped with a limiting MHS, and has two increasing weight filtrations ($M$ and $W$) and a decreasing Hodge filtration $F$. Because it is mixed Tate, the $M$-filtration is canonically split by the Hodge filtration. We will briefly describe the weight filtrations $M$ and $W$.

The Lie algebra $\ugeom$ is generated by elements $\delta_{2n+2}^{(k)} := \ad(\varepsilon_0^\vee)^k (\varepsilon_{2n+2}^\vee)$ for all $n\geq 1$ and $0\leq k\leq 2n$. The element $\delta_{2n+2}^{(k)}$ lies in $M_{2k-2-4n}\cap W_{-2n-2}$.

The monodromy weight filtration $W$ is the negative of the lower central series filtration on $\ugeom$. One should note that $\ugeom$ is the image of $\mathfrak{u}_{1,1} = \Lie(\mathcal{U}_{1,1}^{\dR})$ under the monodromy homomorphism $\mu$ (see \S  \ref{monodromysection}). It is known that $\mu$ is a morphism of mixed Hodge structures \cite[Proposition $15.1$]{hainhodge}, so the MHS on $\ugeom$ is inherited from that of $\mathfrak{u}_{1,1}$.

\subsubsection{Motivic structure}

The following general statement may be proven using the Tannakian formalism.

\begin{lemma} \label{motivicinclusion}
Let $(\mathsf{C}, \omega_{\mathsf{C}})$ and $(\mathsf{D}, \omega_{\mathsf{D}})$ be neutral Tannakian categories over $\mathbb{Q}$, and let $F \colon \mathsf{C}\to\mathsf{D}$ be a fully faithful additive tensor functor satisfying $\omega_{\mathsf{D}} \circ F = \omega_{\mathsf{C}}$. Let $M\in\mathsf{C}$. Then $F(\langle M\rangle_\otimes) \simeq \langle F(M)\rangle_\otimes$.
\end{lemma}

Lemma \ref{motivicinclusion} implies the following important fact about the Lie algebra $\ugeom$.

\begin{proposition} \label{motivicremark}
The Lie algebra $\ugeom$ is the $\mathcal{H}$-realisation of a pro-object in $\MT$. Equivalently, $\mathcal{O}(\Ugeom)$ is the $\mathcal{H}$-realisation of an ind-object of $\MT$.
\end{proposition}
\begin{proof}
The $\mathcal{H}$-realisation functor $\omega^{\mathcal{H}} \colon \MT\to \mathcal{H}$ is fully faithful and compatible with the respective de Rham fiber functors. The Lie algebra $L^\wedge$ is a pro-object within its essential image \cite{brownzetaelements, umem}. Consequently,
\begin{equation*}
\Der(L^\wedge)\cong \Hom (\mathbb{Q}\adr\oplus\mathbb{Q}\bdr, L^\wedge)\cong \Hom(\mathbb{Q}(1)\oplus\mathbb{Q}(0), L^\wedge)
\end{equation*}
is a pro-object in the essential image of $\omega^{\mathcal{H}}$. By definition, $\ugeom$ is a $\mathcal{H}$-subobject of $\Der(L^\wedge)$. Applying Lemma \ref{motivicinclusion} with $\mathsf{C} = \MT$, $\mathsf{D}=\mathcal{H}$, $F = \omega^{\mathcal{H}}$ and $F(M) = \Der(L^\wedge)$ implies that $\ugeom$ is the $\mathcal{H}$-realisation of a pro-object of $\MT$.
\end{proof}

In Theorem \ref{motivetheorem} we prove that the action of $U_{\MT}^{\dR}$ on $\ugeom$ is faithful.

\section{Filtrations}

The fundamental groups of \S \ref{fgsection} are equipped with various \emph{decreasing} filtrations by identifying their points with subgroups of various noncommutative power series rings. Truncating a series in these filtrations defines \emph{increasing} filtrations on the space of coefficients of this series. The technical core of this paper uses this idea to relate the depth and weight filtrations on motivic multiple zeta values to the length filtration on motivic iterated integrals of Eisenstein series.

In this section we review some the available filtrations. We also introduce some new notation, such as coefficient spaces.

\subsection{Formalities}

Throughout, $R$ is a $\mathbb{Q}$-algebra. If $Z$ is a set, the free monoid on $Z$ is denoted by $M(Z) = M(z:z\in Z)$. For each $z \in Z$ there is a monoid morphism $\deg_{z} \colon M (Z)\to (\mathbb{Z}_{\geq 0}, +)$ for which $\deg_z (z) = 1$ and $\deg_z (z') = 0$ for $z'\neq z$. From this one can define a total degree monoid morphism $\deg = \sum_{z\in Z} \deg_{z}$. These degree functions determine gradings on $M(Z)$. 

Associated to $Z$ is the $R$-algebra $R \langle\langle Z\rangle\rangle$ of formal power series whose noncommuting indeterminates are the elements of $Z$. We can write elements $s\in R\langle\langle Z\rangle\rangle$ as formal power series
\begin{equation*}
s = \sum_{w\in M(Z)} s_w w, \quad s_w \in R.
\end{equation*}

Let $\mathbb{Q}\langle Z\rangle$ denote the free shuffle algebra on $Z$. It is a Hopf algebra equipped with the shuffle product and deconcatenation coproduct. There is an isomorphism of $\mathbb{Q}$-vector spaces
\begin{equation} \label{seriesisovec}
\Hom_{\Vect_{\mathbb{Q}}^{\infty}}(\mathbb{Q}\langle Z\rangle, R)\xrightarrow{\sim} R\langle\langle Z\rangle\rangle, \quad s\mapsto \sum_{w\in M(Z)} s(w) w,
\end{equation}
where the homomorphism space on the left is in the category $\Vect_{\mathbb{Q}}^\infty$ of all $\mathbb{Q}$-vector spaces (not necessarily finite-dimensional).

Let $\mathcal{A}$ be an $R$-algebra. In this section we will be concerned with decreasing filtrations $F^\bullet \mathcal{A}$ that are compatible with the algebra structure on $\mathcal{A}$. Where there is no danger of confusion, we will omit reference to the underlying algebra in the notation and write $F^\bullet$ instead of $F^\bullet \mathcal{A}$.

\begin{remark} \label{transferfiltration}
The (commutative) shuffle algebra $\mathbb{Q}\langle Z \rangle$ and (noncommutative) power series ring $R\langle\langle Z \rangle\rangle$ are naturally constructed from $M(Z)$. One can define decreasing filtrations on each of these algebras using the degree functions on $M(Z)$ in a natural way.

Using the isomorphism \eqref{seriesisovec}, it is possible to pass between filtrations on $\mathbb{Q}\langle Z \rangle$ and $R\langle\langle Z \rangle\rangle$ as follows: given $F^\bullet R\langle\langle Z \rangle\rangle$, define $F^\bullet\mathbb{Q}\langle Z\rangle := F^\bullet R\langle\langle Z\rangle\rangle\cap \mathbb{Q}\langle Z\rangle$. This uses the vector space inclusion $\mathbb{Q}\langle Z\rangle \subseteq R\langle\langle Z\rangle\rangle$.\footnote{If $R\langle \langle Z \rangle \rangle$ contains the $R$-points of a fundamental group and $\mathbb{Q}\langle Z \rangle$ is its affine ring, this inclusion is an abuse of notation whereby elements are implicitly identified with their duals. It is therefore incompatible with the algebra structure.} Conversely, given $F^\bullet \mathbb{Q}\langle Z \rangle$, define $F^\bullet R\langle\langle Z \rangle\rangle$ as the image of $F^\bullet \Hom_{\Vect_{\mathbb{Q}}^\infty} (\mathbb{Q}\langle Z\rangle, R)$ under \eqref{seriesisovec}, where
\begin{equation*}
F^r \Hom_{\Vect_{\mathbb{Q}}^\infty} (\mathbb{Q}\langle Z\rangle, R) := \im\left(\Hom_{\Vect_{\mathbb{Q}}^\infty} (F^r\mathbb{Q}\langle Z\rangle, R)\to\Hom_{\Vect_{\mathbb{Q}}^\infty} (\mathbb{Q}\langle Z\rangle, R) \right)
\end{equation*}
is the image of extension by zero.
\end{remark}

\subsubsection{Filtered pieces of series}

Let $\mathcal{A}$ be any $R$-algebra equipped with a separated and exhaustive decreasing filtration $F^\bullet$, and let $s\in\mathcal{A}$. Write
\begin{equation*} \label{filteredseries}
s = \sum_{k\geq 0} s_k, \quad \text{where } s_k\in F^k\backslash F^{k+1}.
\end{equation*}
\begin{definition}[Filtered piece of series] \label{filteredpiece}
The $r$th $F$-filtered piece of $s$ is
\begin{equation*}
\Fil_F^r (s):=\sum_{k=0}^r s_k.
\end{equation*}
This implies that $s\equiv \Fil_F^r (s) \pmod{F^{r+1}}$.
\end{definition}

\subsubsection{Induced filtrations} \label{inducedfiltration}

Let $\mathcal{A}$ be an $R$-algebra equipped with a decreasing filtration $F^\bullet \mathcal{A}$. Let $E$ be an $R$-module acting $R$-linearly on $\mathcal{A}$. There is an induced filtration $F^\bullet E$ defined by $F^r E := \left\{ e \in E: e\left(F^k \mathcal{A} \right) \subseteq F^{k+r} \mathcal{A}\right\}$.

\subsection{Decreasing filtrations on fundamental groups}

In this section we define decreasing filtrations on $\Pi_X^{\dR}$, $\Pi_Y^{\dR}$, $\mathcal{U}_{1,1}^{\dR}$, $\mathcal{U}_{1,\vec{1}}^{\dR}$ and $\Ugeom$.

\subsubsection{Filtrations on $\Pi_X^{\dR}$}

Let $X = \P1minus$. We have seen that
\begin{equation*}
\Pi_X^{\dR} (R) \cong \GrL \left(R\langle\langle \mathsf{x}_0, \mathsf{x}_1\rangle\rangle \right). 
\end{equation*}
The filtrations on $\Pi_X^{\dR} (R)$ are induced by the following filtrations on the power series ring.
\begin{definition} \label{filtrationsonP1minus}
The power series ring $R\langle\langle \mathsf{x}_0, \mathsf{x}_1\rangle\rangle$ is equipped with two natural decreasing filtrations: the \emph{weight} filtration $W^\bullet R\langle\langle \mathsf{x}_0, \mathsf{x}_1 \rangle \rangle$ is the decreasing filtration on the total degree in both $\mathsf{x}_0$ and $\mathsf{x}_1$, and the \emph{depth} filtration $D^\bullet R\langle\langle \mathsf{x}_0, \mathsf{x}_1 \rangle \rangle$ is the decreasing filtration on the $\mathsf{x}_1$-degree.
\end{definition}

\begin{remark}
The depth filtration is induced by the inclusion $X\hookrightarrow \mathbb{G}_m$ in the sense that $D^1\Pi_X^{\dR}$ is the kernel of the induced morphism $\Pi_X^{\dR} \to \pi_1^{\dR} (\mathbb{G}_m, \vec{1}_1)$.
\end{remark}

\subsubsection{Filtrations on $\Pi_Y^{\dR}$}

We have seen that
\begin{equation*}
\Pi_Y^{\dR} (R) \cong  \GrL \left(R\langle\langle \adr , \bdr\rangle\rangle \right).
\end{equation*}
The filtrations on $\Pi_Y^{\dR} (R)$ are induced by the following filtrations on the power series ring.

\begin{definition} \label{filtrationsonPidr}
The power series ring $R\langle\langle \adr, \bdr \rangle\rangle$ is equipped with three natural decreasing filtrations: the $A$-filtration $A^\bullet R\langle\langle \adr, \bdr \rangle \rangle$ is the decreasing filtration on the $\adr$-degree; the $B$-filtration $B^\bullet R\langle\langle\adr, \bdr \rangle \rangle$ is the decreasing filtration on the $\bdr$-degree; and the \emph{elliptic depth} filtration $D^\bullet R\langle\langle\adr, \bdr \rangle \rangle$ is the decreasing filtration on the $[\adr,\bdr]$-degree, defined by
\begin{align*}
D^r R\langle\langle\adr, \bdr \rangle \rangle &= 
\begin{cases}
R\langle\langle\adr, \bdr \rangle \rangle, & r=0 \\
\ker\left(R\langle\langle\adr, \bdr \rangle \rangle \to R[[\adr,\bdr]] \right), & r=1 \\
 \left\{\text{elements of } D^1\text{-degree } \geq r \right\}, & r\geq 2. 
\end{cases}
\end{align*}
It is clear that $D^\bullet \subseteq A^\bullet\cap B^\bullet$.
\end{definition}

\subsubsection{Filtrations on $\mathcal{U}_{1,1}$ and $\mathcal{U}_{1,\vec{1}}$} \label{lengthfiltration}

In \S \ref{specificfundamentalgroups} we gave an explicit description of the elements of $\mathcal{U}_{1,1} (R)$ and $\mathcal{U}_{1,\vec{1}}(R)$. These can be used to define a length filtration.

\begin{definition}
The groups $\mathcal{U}_{1,1} (R)$ and $\mathcal{U}_{1,\vec{1}}(R)$ are equipped with a decreasing \emph{length} filtration $L^\bullet$. It is the decreasing filtration induced by the total degree, where the generators $\mathbf{e}_{2n+2}$, $\mathbf{e}_f'$ and $\mathbf{e}_f''$ are each assigned degree $1$.
\end{definition}

\begin{remark} \label{lengthfiltrationremark}
The length filtration is so named because the coefficients of an element $u\in L^r$ are $R$-valued iterated integrals of length at most $r$.
\end{remark}

\subsubsection{Filtrations on $\Ugeom$ and $\ugeom$} \label{filtrationsUgeom}

The unipotent group $\Ugeom$ is the group scheme whose Lie algebra is the algebra $\ugeom$ of geometric derivations defined in \S \ref{ugeomsection}. It can also be defined as the image of $\mathcal{U}_{1,1}$ under the monodromy morphism $\mu$ of \S \ref{monodromysection}.
\begin{definition}
The group scheme $\Ugeom$ is equipped with a decreasing \emph{length} filtration $L^\bullet$ defined by $L^\bullet \Ugeom  = \mu (L^\bullet \mathcal{U}_{1,1})$. Equivalently, it can be defined as the decreasing filtration induced by the total degree, where $\varepsilon_{2n+2}^\vee$ and $\varepsilon_{2n+2}$ are assgined degree $1$ for all $n\geq 1$, and $\varepsilon_0^\vee$ and $\varepsilon_0$ are assigned degree $0$.
\end{definition}

\subsection{Coefficient spaces}\label{coeffspace}

De Rham fundamental groups are a convenient tool for packaging iterated integrals into generating series. The notion of coefficient space extracts the coefficients from such series.

\begin{definition}[Coefficient space] \label{coeffspacedef}
Let $s\in R\langle\langle Z\rangle\rangle$ be a formal power series. The coefficient space $\Co(s)$ is the vector space image
\begin{equation*}
\Co(s):=\im\left(s \colon \mathbb{Q}\langle Z\rangle \to R\right)
\end{equation*}
of the associated linear map under \eqref{seriesisovec}. It is a subspace of the $\mathbb{Q}$-vector space $R$.
\end{definition}

Extra structure on a series $s$ induces structure on $\Co(s)$. The ring $R\langle\langle Z\rangle\rangle$ can be equipped with the structure of a Hopf algebra by taking the coproduct $\Delta$ for which each $z \in Z$ is primitive. The isomorphism \eqref{seriesisovec} restricts to a group isomorphism
\begin{equation} \label{shufflegroupiso}
\Hom_{\Alg_{\mathbb{Q}}}(\mathbb{Q}\langle Z\rangle, R)\xrightarrow{\sim} \GrL\left( R\langle\langle Z\rangle\rangle\right).
\end{equation}
Therefore if $s$ is grouplike, $\Co(s)$ is a $\mathbb{Q}$-subalgebra of $R$.

\begin{remark}
Let $X$ be a scheme over $\mathbb{Q}$ and let $x\in X(\mathbb{Q})_{\bp}$. Let $\mathcal{U}$ be the unipotent radical of $\pi_1^{\rel, \dR} (X,x)$ with respect to some choice of semistable subcategory of $\Con (X)$. The affine ring $\mathcal{O} (\mathcal{U})$ is equipped with the shuffle product induced by that on iterated integrals. Choosing a set $Z$ of generators for $\mathcal{O}(\mathcal{U})$ determines a non-canonical isomorphism $\mathcal{O}(\mathcal{U})\cong \mathbb{Q}\langle Z \rangle$. Let $s\in \mathcal{U}(R)\cong \Hom(\mathbb{Q}\langle Z\rangle, R)$. By \eqref{shufflegroupiso}, it follows that $\Co(s)$ is a $\mathbb{Q}$-subalgebra of $R$.
\end{remark}

\subsubsection{Induced filtrations on coefficient spaces} \label{coeffspacefiltration}

As before, let $Z$ be any set of indeterminates. Suppose that $R\langle\langle Z \rangle\rangle$ is equipped with an exhaustive, separated, {decreasing} filtration $F^\bullet$. Let $s\in R\langle\langle Z \rangle\rangle$. To the pair $(s, F^\bullet)$ one can associate an \emph{increasing} filtration on $\Co(s)$ as follows:

\begin{definition}
Define an increasing filtration of subspaces $\Co_\bullet^F (s)$ of $\Co(s)$ by
\begin{equation*}
\Co_r^F (s) := \Co \left(\Fil_F^r (s) \right).
\end{equation*}
\end{definition}

\begin{proposition} \label{coefffiltrationmult}
If $s $ is grouplike, the filtration $\Co_\bullet ^F (s)$ is compatible with the algebra structure on $\Co(s)$ i.e. $\Co_m^F (s) \Co_n^F (s) \subseteq \Co_{m+n}^F (s)$.
\end{proposition}
\begin{proof}
The isomorphism \eqref{shufflegroupiso} implies that $s$ is a shuffle-algebra homomorphism $s \colon \mathbb{Q}\langle Z \rangle \to R$. Let $c_1\in \Co_m ^F (s)$ and $c_2 \in \Co_n ^F(s)$. By linearity, we may assume $c_i = s(u_i)$ are the coefficients of words $u_1, u_2 \in M(Z)$. Then
\begin{equation*}
c_1  c_2 = s(u_1) s(u_2) = s(u_1 \shuffle u_2),
\end{equation*}
where $\shuffle$ denotes the shuffle product. Let $F^\bullet \mathbb{Q}\langle Z \rangle$ be the associated decreasing filtration on the shuffle algebra on $Z$ described in Remark \ref{transferfiltration}. The assumption on $c_1$ and $c_2$ implies that $u_1 \in F^k \mathbb{Q}\langle Z \rangle$ for $k\leq m$ and $u_2 \in F^l \mathbb{Q}\langle Z \rangle$ for $l\leq n$. The shuffle product $u_1 \shuffle u_2 = \sum v$ is a sum of words $v\in F^{k+l} \mathbb{Q}\langle Z \rangle$, which implies that $c_1 c_2 = \sum s(v)\in \Co_{k+l}^F (s)\subseteq \Co_{m+n}^F (s)$.
\end{proof}

\begin{proposition} \label{coeffapplication}
Let $E$ be an $R$-module acting on $F^\bullet R\langle\langle Z \rangle\rangle$. Suppose that $L^\bullet E$ is a subfiltration of the induced filtration $F^\bullet E$. Let $e\in E$ and $s\in R\langle\langle Z\rangle\rangle$. Then
\begin{equation*}
\Co_r^F (e(s)) = \sum_{i + j = r} \Co_i^L (e) \Co_j^F (s).
\end{equation*}
\end{proposition}
\begin{proof}
Write $e = \sum_{k\geq 0} e_k$ and $s = \sum_{l\geq 0} s_l$, where $e_k \in L^k \backslash L^{k+1}$ and $s_l\in F^l\backslash F^{l+1}$. It is clear that
\begin{equation*}
\Fil_F^r (e (s)) = \sum_{\substack{k,l\geq 0 \\ k+l \leq r}} e_k (s_l),
\end{equation*}
because $e_k (s_l)\in F^{k+l} R\langle\langle Z\rangle\rangle$. The result follows by taking coefficient spaces.
\end{proof}

\subsubsection{Depth filtration and weight grading on motivic MZVs}

The filtrations on coefficient spaces defined in \S \ref{coeffspacefiltration} can be used to equip the algebra $\Zm$ of motivic multiple zeta values with natural filtrations.

Let $\Phi^{\mm}_{0 1}$ be the motivic Drinfeld associator defined in \S \ref{motivicdrinfeldassociators}. It is a generating series for motivic MZVs, and therefore $\Co(\Phi^{\mm}_{0 1}) = \Zm$. The weight and depth filtrations on $\mathcal{P}^{\mm}_{\MT}  \langle\langle \mathsf{x}_0, \mathsf{x}_1\rangle\rangle$ defined in Definition \ref{filtrationsonP1minus} can be used to equip $\Zm$ with weight and depth filtrations, as follows:

\begin{definition} \label{weightdepthMZV}
The algebra $\Zm$ has two increasing filtrations: the \emph{weight} filtration $\mathfrak{W}_\bullet \Zm$, and the \emph{depth} filtration $\mathfrak{D}_\bullet \Zm$. They are defined by
\begin{align*}
\mathfrak{W}_\bullet \Zm &:= \Co_\bullet^W (\Phi^{\mm}_{0 1}); \\
\mathfrak{D}_\bullet \Zm &:= \Co_\bullet^D (\Phi^{\mm}_{0 1}).
\end{align*}
\end{definition}
Since $\Phi^{\mm}_{0 1}$ is grouplike, Proposition \ref{coefffiltrationmult} implies that these filtrations are compatible with the shuffle product on $\Zm$. 

\begin{remark}
The weight (resp. depth) filtration is so named because its image under the period map is precisely the weight (resp. depth) filtration on numerical multiple zeta values. We recall that the weight of an admissible multiple zeta value $\zeta(k_1, \dots, k_r)$ is $k_1 +\dots + k_r$ and the depth is $r$.
\end{remark}

\begin{proposition} \label{associatorweights}
The weight filtration satisfies $\mathfrak{W}_\bullet \Zm = \Co_\bullet^W (\Phi^{\mm}_{i j})$ whenever $i, j \in \left\{0,1,\infty\right\}$ are distinct.
\end{proposition}
\begin{proof}
The series $\Phi_{i j}^{\mm}$ is obtained from $\Phi_{0 1}^{\mm}$ by making the change of variables $(\mathsf{x}_0, \mathsf{x}_1)\mapsto (\mathsf{x}_i, \mathsf{x}_j)$. This change of variables induces an automorphism $W^\bullet\xrightarrow{\sim} W^\bullet$, and therefore the coefficient space remains unchanged.
\end{proof}

The weight filtration $\mathfrak{W}_\bullet \Zm$ is induced from a \emph{grading}. This is particular to motivic multiple zeta values; numerical MZVs are only conjecturally graded.

The weight grading is defined as follows: let
\begin{equation*}
\mathcal{Z}_k^{\mm} := \langle \zeta^{\mm}(w): \deg (w) = k\rangle_\mathbb{Q}
\end{equation*}
be the subspace of motivic MZVs of weight $k$. Then $\mathcal{Z}_k^{\mm} \cong \mathfrak{W}_k /\mathfrak{W}_{k-1}$, and
\begin{align*}
\mathfrak{W}_r \Zm = \bigoplus_{k=0}^r \mathcal{Z}_{k}^{\mm}.
\end{align*}
The ideal of motivic MZVs of positive weight is
\begin{equation*}
\mathcal{Z}^{\mm}_{> 0} :=\bigoplus_{r > 0} \mathcal{Z}_r^{\mm}.
\end{equation*}

Let $\mathcal{P}:=\Zm [ \mathbb{L} ]$. Brown's result \cite{brownmtm} implies that $\mathcal{P}$ is the ring $\mathcal{P}^{\mm, +}_{\MT}$of effective motivic periods for mixed Tate motives over $\mathbb{Z}$, and that $\mathcal{P}^{\mm}_{\MT} \cong \mathcal{P}[\mathbb{L}^{-1}]$.

\subsubsection{Action of $S^{\mm}$}

The group scheme $SL_2^{\dR}$ acts on $\mathbb{Q}\langle\langle \adr, \bdr \rangle\rangle$ on the left, as described in \S  \ref{leftrightactions}. The element $S^{\mm}$ acts via
\begin{equation} \label{actionmotivicS}
S^{\mm} (\adr, \bdr) = (-\mathbb{L}^{-1} \bdr, \mathbb{L} \adr).
\end{equation}

\begin{proposition} \label{SL2filtrationswap}
Let $R$ be a $\mathbb{Q}[\mathbb{L}^\pm]$-algebra. The action of $S^{\mm}$ on $R\langle\langle \adr, \bdr \rangle\rangle$ induces isomorphisms of filtered $R$-algebras $A^\bullet R\langle\langle \adr, \bdr \rangle\rangle \rightleftarrows B^\bullet R\langle\langle \adr, \bdr \rangle\rangle$.
\end{proposition}
\begin{proof}
This is immediate from equation \eqref{actionmotivicS}.
\end{proof}

The following useful lemma is a consequence of Proposition \ref{SL2filtrationswap}.

\begin{lemma} \label{SL2filteredpiece}
Let $R$ be a $\mathbb{Q}[\mathbb{L}^\pm]$-algebra. Then
\begin{equation*}
{S^{\mm}}\circ  {\Fil_A^r} = {\Fil_B^r} \circ {S^{\mm}}.
\end{equation*}
The same formula also holds with $A$ and $B$ interchanged.
\end{lemma}
\begin{proof}
An element $w\in R\langle\langle \adr, \bdr \rangle\rangle$ is contained in $A^r$ iff $S^{\mm} (w)\in B^r$.
\end{proof}

\section{Relations between modular curves and the Tate curve} \label{modularandTate}

In this section we study the relationships between the fundamental groups of the moduli spaces $\M11$ and $\mathcal{M}_{1,\vec{1}}$, and the fundamental group of the infinitesimal punctured Tate curve $\Tatecurve$.

\subsection{The topological monodromy action} \label{topmonodromysection}

Let $E$ be an elliptic curve over a field $K\subseteq \mathbb{C}$; this determines a point $x = [E]\in \M11(K)$. By equipping $E^\times$ with a nonzero tangential basepoint $\vec{v}$ at the puncture we also determine a point $\hat{x}\in\mathcal{M}_{1,\vec{1}}(K)$ such that $\hat{x}\mapsto x$ under the morphism $\mathcal{M}_{1, \vec{1}} \to \M11$.

We can also take a tangential basepoint on $\M11$; e.g. $x = \vec{1}_\infty$. In this case the corresponding elliptic curve is the infinitesimal Tate curve $\mathcal{E}_{\partial/\partial q}$, which can be equipped with the tangential basepoint $\vec{v} = \vec{1}_O$ at the origin $O$.

Throughout this section we work with an arbitrary (possibly tangential) basepoint $x$. From \S \ref{monodromysection} onwards, however, we specialise $x = \vec{1}_\infty$, as our argument crucially relies on the limiting MHS on $\pi_1^{\dR}(\Tatecurve, \vec{1}_O)$ being mixed Tate.

\subsubsection{The outer monodromy action}

The fiber of $\mathcal{E}^\times \to\M11$ over $x$ is $E^\times$. The homotopy exact sequence of this fibration produces a short exact sequence
\begin{equation*}
1\to \pi_1^{\text{top}}(E^\times ,\vec{v}) \to \pi_1^{\text{top}} (\mathcal{E}^\times, [E^\times,\vec{v}])\to\pi_1^{\text{top}}(\M11,x)\to 1
\end{equation*}
exhibiting $\pi_1^{\text{top}} (\mathcal{E}^\times, [E^\times,\vec{v}])$ as an extension of $\Gamma$ by a free group on two generators \cite[Proposition $1.4$]{hainkzb}. Conjugation determines an outer action
\begin{equation} \label{outeraction}
\bar{\mu}_0 \colon \pi_1^{\text{top}}(\M11,x)\to \Out \pi_1^{\text{top}} (E^\times, \vec{v}).
\end{equation}

\begin{remark}
The outer action does not lift to a genuine action for the following geometric reason: the basepoint $x = [E]$ on $\M11$ can also be represented by a different \emph{model} of $E$, say $x = [E']$, corresponding to an isomorphism $E\cong E'$ defined over $K$. But there is no natural way to choose a tangential basepoint $\vec{v}$ on all models of $E$ simultaneously such that the $\pi_1^{\text{top}}(\M11,x)$-action respects this choice. The moduli space $\mathcal{M}_{1,\vec{1}}$ overcomes this issue because a basepoint $\hat{x}\in\mathcal{M}_{1,\vec{1}}$ determines an elliptic curve $E$ \emph{together} with a choice of tangential basepoint $\vec{v}$ on $E^\times$. This allows us to define an action of $\pi_1^{\text{top}}(\mathcal{M}_{1,\vec{1}}, \hat{x})$ on $\pi_1^{\text{top}}(E^{\times}, \vec{v})$.
\end{remark}

\subsubsection{The topological monodromy action}

We can identify $\pi_1^{\text{top}}(\mathcal{M}_{1,\vec{1}},\hat{x})\cong B_3$ with the mapping class group of $\Sigma^\times$, a genus $1$ surface with one puncture. The generators $t_A$ and $t_B$ can be identified with Dehn twists on $A,B\subseteq \Sigma$, where $A$ and $B$ are any two simple closed curves on $\Sigma$ that intersect transversally at one point. For example, $A$ and $B$ could be (the images of) the two generators $\alpha$ and $\beta$ for the fundamental group of $\Sigma^\times$. The \emph{left} action of $\pi_1^{\text{top}}(\mathcal{M}_{1,\vec{1}}, \hat{x})$ on $\pi_1^{\text{top}}(E^\times, \vec{v})$ via the Dehn twists $t_A$ and $t_B$ defines a homomorphism
\begin{equation} \label{topmonodromy}
\mu_0 \colon \pi_1^{\text{top}}(\mathcal{M}_{1,\vec{1}},\hat{x})\to \Aut \pi_1^{\text{top}} (E^\times, \vec{v}).
\end{equation}

\begin{lemma}
Let $A$ and $B$ denote the images of generators $\alpha$ and $\beta$ for $\pi_1^{\text{top}} (E^\times, \vec{v})$. Then $\mu_0$ is given by the following explicit action:
\begin{equation} \label{topmonodromyexplicit}
t_A(\alpha) = \alpha, \quad t_A(\beta) = \beta\alpha, \quad t_B(\alpha) = \alpha \beta^{-1}, \quad t_B(\beta) = \beta.
\end{equation}
\end{lemma}
\begin{proof}
This follows from the description of $t_A$ and $t_B$ as Dehn twists along the images of $\alpha$ and $\beta$ by considering the shape of $\alpha, \beta$ in the Jacobi uniformisation $E(\mathbb{C})\cong \mathbb{C}^\times/q^{\mathbb{Z}}$ \cite{hainnotes}. It can also be derived group-theoretically by viewing the free group $F_2 = \langle \alpha, \beta \rangle$ as a subgroup of the braid group $B_4$ on four strands, with $B_3 \subseteq B_4$ acting on $F_2$ via conjugation \cite[\S $9$]{nakamuraeisenstein}.
\end{proof}

\begin{remark}
This action fixes $\Theta = [\alpha, \beta]$. Geometrically, $\Theta$ corresponds to a small loop around the puncture of the elliptic curve. The punctured torus $\Sigma^\times$ is homotopy-equivalent to a torus with a small open disc removed, which we denote by $\Sigma^o$. Elements of the mapping class group $\pi_1^{\text{top}}(\mathcal{M}_{1,\vec{1}}, \hat{x})$ fix  $\partial \Sigma^o$, which is homotopic to the image of $\Theta$.
\end{remark}

\subsubsection{Main topological equation} \label{topologicalSbeta}
Let $\tilde{S} := (t_A t_B t_A )^{-1} \in B_3$. Under \eqref{topmonodromy}, this element acts by
\begin{equation*}
\tilde{S}(\alpha) = \alpha \beta\alpha^{-1}, \quad \tilde{S}(\beta) = \alpha^{-1}.
\end{equation*}
Up to conjugation and inverses, the element $\tilde{S}$ interchanges $\alpha$ and $\beta$. The essential idea of this paper is to use this ``modular inverter'' to show how the periods of $\alpha$ and $\beta$ are related to the periods of $\tilde{S}$.

\subsubsection{Induced morphism to $\Gamma$}

There is an induced action of $\pi_1^{\text{top}}(\mathcal{M}_{1,\vec{1}}, \hat{x})$ on
\begin{equation*}
 \pi_1^{\text{top}} (E^\times, \vec{v})^{\ab}\cong H_1(E(\mathbb{C}),\mathbb{Z})\cong \mathbb{Z}^2.
\end{equation*} 
This may be described by a homomorphism $\pi_1^{\text{top}}(\mathcal{M}_{1,\vec{1}}, \hat{x}) \to GL_2(\mathbb{Z})$. This morphism factors through $\Gamma$ because $\pi_1^{\text{top}}(\mathcal{M}_{1,\vec{1}}, \hat{x})$ fixes $\Theta$, and is thus orientation-preserving. We now describe this morphism explicitly.

Let $\adr$ (resp. $\bdr$) denote the image of $\alpha$ (resp. $\beta$) in $H_1(E^\times(\mathbb{C}), \mathbb{Z})$. The fundamental group $\pi_1^{\text{top}}(\mathcal{M}_{1,\vec{1}}, \hat{x})\cong B_3$ acts on $\adr$ and $\bdr$ by the abelianisation of the action \eqref{topmonodromyexplicit}. This can be written as a \emph{right} action on frames, given by
\begin{equation*}
(t_A(\adr),t_A(\bdr)) = (\adr,\bdr) \left( \begin{matrix} 1 & 1 \\ 0 & 1 \end{matrix}\right), \quad (t_B(\adr),t_B(\bdr)) = (\adr,\bdr) \left( \begin{matrix} 1 & 0 \\ -1 & 1 \end{matrix}\right).
\end{equation*}

\begin{definition} \label{fdef}
The abelianised action defines $f \colon B_3 \to \Gamma$, given by
\begin{equation} \label{fdefexplicit}
t_A \mapsto T = \left( \begin{matrix} 1 & 1 \\ 0 & 1 \end{matrix} \right), \quad t_B \mapsto (T S T)^{-1} = \left( \begin{matrix} 1  & 0 \\ -1 & 1 \end{matrix} \right).
\end{equation}
\end{definition}

Let $\tilde{S} =( t_A t_B t_A)^{-1}$ as above, and let $\tilde{T} := t_A$. Then $f(\tilde{S}) = S$ and $f(\tilde{T}) = T$. Since $\Gamma$ is generated by $S$ and $T$, $f$ is surjective.

\begin{proposition} \label{kernelf}
The kernel of $f$ is an infinite cyclic group generated by $\tilde{S}^4 = (t_A^{-1} t_B^{-1})^6$. 
\end{proposition}
\begin{proof}
It is easy to verify that $f(\tilde{S}^4) = I$. Set $\tilde{U} = \tilde{S}\tilde{T}$. From the presentation for the braid group in terms of the generators $t_A$ and $t_B$ given in \eqref{braidpresentation}, we obtain the alternate presentation $B_3 \cong \langle \tilde{S}, \tilde{U} \vert \tilde{S}^2 = \tilde{U}^3\rangle$. There is a well-known presentation for $\Gamma$ in terms of $S$ and $U = ST$, namely $\Gamma = \langle S, U \vert S^2 = U^3, S^4 = I\rangle$. From this it is clear the kernel is precisely generated by $\tilde{S}^4$.
\end{proof}

\begin{remark}[Geometric interpretation]
The kernel is isomorphic to $\mathbb{Z}$ because $f$ can be identified with the map on fundamental groups induced by the morphism $\mathcal{M}_{1,\vec{1}}\to\M11$, which is a principal $\mathbb{G}_m$-bundle. 

Recall that $\ker(f)$ is generated by $\tilde{S}^4$. One verifies that $\mu_0 (\tilde{S}^4) = \Ad_\Theta$. Therefore, $\ker(f)$ consists of elements acting on $ \pi_1^{\text{top}}(E^\times,\vec{v})$ as a power of $\Ad_\Theta$. 

The action of $\pi_1^{\text{top}}(\mathcal{M}_{1,\vec{1}}, \hat{x})$ on $\pi_1^{\text{top}}(E^\times, \vec{v})^{\ab} \cong \pi_1^{\text{top}}(E, \vec{v})$ factors through the fundamental group $\pi_1^{\text{top}}(\M11, x) \cong SL_2 (\mathbb{Z})$, and this in turn factors through the outer action \eqref{outeraction}. This gives a geometric interpretation for why elements in $\ker(f)$ must act on $\pi_1^{\text{top}}(E^\times, \vec{v})$ via inner automorphisms. The precise choice of $\Ad_\Theta$ corresponds to the fact that $\Theta\in\ker(\pi_1^{\text{top}}(E^\times, \vec{v}) \to \pi_1^{\text{top}}(E, \vec{v}))$.
\end{remark}

The situation is summarised in the following commutative diagram:
\begin{equation*}
\begin{tikzcd}
\ker(f)= \langle \tilde{S}^4\rangle \arrow[d, hook] \arrow[r,"\mu_0\vert_{\ker(f)}"] & \langle \Ad_\Theta \rangle \arrow[d,hook] \\
\pi_1^{\text{top}}(\mathcal{M}_{1,\vec{1}}, \hat{x}) \cong B_3\arrow[r, "\mu_0" ] \arrow[d, "f", swap]& \Aut \pi_1^{\text{top}}(E^\times,\vec{v}) \arrow[d] \\
\pi_1^{\text{top}}(\M11, x) \arrow[r, "\bar{\mu}_0" ] \arrow[d, "\sim", swap]& \Out  \pi_1^{\text{top}}(E^\times,\vec{v}) \arrow[d] \\
\Gamma \arrow[r, hook] & \Aut \left(\pi_1^{\text{top}}(E^\times,\vec{v})^{\ab} \right)\cong GL_2(\mathbb{Z})
\end{tikzcd}
\end{equation*}
The left column is exact at $B_3$ and the final vertical map in the right column is the natural map $\Out(G)\to \Aut(G^{\ab})$.

\begin{remark}
Although one cannot define an action of $\pi_1^{\text{top}}(\M11,x)$ on $\pi_1^{\text{top}}(E^\times,\vec{v})$, there is a way to define a ``motivic'' action of the de Rham relative completion $\pi_1^{\rel,\dR}(\M11,\vec{1}_\infty)$ on the de Rham fundamental group $\Pi_Y^{\dR}$, using the structure theory of the relative fundamental group of $\mathcal{M}_{1,\vec{1}}$. This is covered in \S \ref{monodromysection}.
\end{remark}

\subsection{The monodromy action for relative fundamental groups} \label{monodromysection}

The topological action of \S \ref{topmonodromysection} extends to an action of $\mathcal{G}_{1,\vec{1}}^{\dR}$ on $\Pi_Y^{\dR}$. This can be described by an $SL_2^{\dR}$-equivariant morphism
\begin{equation} \label{truemonodromy}
\mu \colon \mathcal{U}^{\dR}_{1,\vec{1}}\to\Aut(\Pi_Y^{\dR}).
\end{equation}
The monodromy action may we written using the decomposition $\mathcal{G}_{1,\vec{1}}^{\dR} \cong  \mathcal{U}_{1,\vec{1}}^{\dR} \rtimes SL_2^{\dR}$, given in \eqref{sdsplit}, as follows. Let $R$ be a $\mathbb{Q}$-algebra and let $u\in \mathcal{U}_{1,\vec{1}}^{\dR}(R)$, $\gamma\in SL_2^{\dR}(R)$ and $\pi\in\Pi_Y^{\dR}(R)$. Then
\begin{equation} \label{leftactiondef}
(u,\gamma)(\pi) = \mu(u)(\gamma(\pi)).
\end{equation}

\subsubsection{Induced action of $\mathcal{G}_{1,1}^{\dR}$}

Recall that $\mathcal{U}_{1,\vec{1}}^{\dR} \cong \mathcal{U}_{1,1}^{\dR}\times\mathbb{G}_a$. We obtain a morphism
\begin{equation*}
\mathcal{U}_{1,1}^{\dR}\hookrightarrow \mathcal{U}_{1,\vec{1}}^{\dR} \xrightarrow{\mu} \Aut (\Pi_Y^{\dR}),
\end{equation*}
which we also denote by $\mu$ by abuse of notation.

\begin{definition} \label{Ugeomdef}
Let $\mathcal{U}^{\geom}:=\im (\mu \colon \mathcal{U}_{1,1}^{\dR} \to \Aut (\Pi_Y^{\dR}))$. It is a pro-unipotent group scheme whose Lie algebra is the lower central series completion of $\ugeom$.
\end{definition}

\begin{remark} \label{errorterm}
Hain's result \cite[Theorem $15.4$]{hainhodge} implies that $\mu$ vanishes on the cuspidal generators $\mathbf{e}_f'$ and $\mathbf{e}_f ''$. It therefore depends only on its values at Eisenstein generators. These are determined by Proposition \ref{explicitmonodromy} below.
\end{remark}

\subsubsection{Explicit formula}

The monodromy homomorphism induces a morphism of fundamental Lie algebras, which may be written down explicitly.

\begin{proposition} \label{explicitmonodromy}
The monodromy morphism induces a morphism of Lie algebras, denoted $\mu \colon \mathfrak{u}_{1,\vec{1}} \to \ugeom\oplus {\mathbb{Q}\varepsilon_2}$. It vanishes on cuspidal generators and satisfies
\begin{equation*}
\mathbf{e}_{2n+2}\mathsf{X}^{k} \mathsf{Y}^{2n-k} \mapsto \frac{2 (2n-k)!}{\left[ (2n)!\right]^2} \ad\left(\varepsilon_0^\vee\right)^k \left(\varepsilon_{2n+2}^\vee \right)
\end{equation*}
\end{proposition}
\begin{proof}
By \cite[Theorem $15.4$]{hainhodge}, the cuspidal generators are contained in $\ker(\mu)$, and \cite[Theorem $15.7$]{hainhodge} provides the first formula in the case $k=0$. It remains to compute the action on a general Eisenstein generator $\mathbf{e}_{2n+2}\mathsf{X}^k \mathsf{Y}^{2n-k}$ for any $0\leq k\leq 2n$.

The $SL_2$-equivariance of $\mu \colon \mathcal{U}_{1,\vec{1}}^{\dR}\to \Aut(\Pi_Y^{\dR})$ implies that the induced map on Lie algebras is $\mathfrak{sl}_2$-equivariant. The two representations of $\mathfrak{sl}_2$ are identified via
\begin{equation*}
\mathsf{X}\frac{\partial}{\partial\mathsf{Y}}\mapsto \ad(\varepsilon_0^\vee), \quad \mathsf{Y}\frac{\partial}{\partial\mathsf{X}}\mapsto -\ad(\varepsilon_0).
\end{equation*}
The $\mathfrak{sl}_2$-action may be used to write a general Eisenstein generator in the form
\begin{equation*}
\mathbf{e}_{2n+2}\mathsf{X}^k \mathsf{Y}^{2n-k} = \frac{(2n-k)!}{(2n)!} \left(\mathsf{X} \frac{\partial}{\partial \mathsf{Y}}\right)^k \left(\mathbf{e}_{2n+2} \mathsf{Y}^{2n}\right).
\end{equation*}
Applying $\mu$ and using the $\mathfrak{sl}_2$-equivariance gives the result.

\end{proof}

\begin{remark}
Proposition \ref{explicitmonodromy} implies that $\mu(\mathbf{e}_2) = 2\varepsilon_2$, where  $\varepsilon_2 =\varepsilon_2^\vee = -\ad([\adr,\bdr])$.
\end{remark}

\subsubsection{Action of the length filtration}

Let $B^\bullet \Pi_Y^{\dR}$ be the filtration defined in Definition \ref{filtrationsonPidr}, and let $L^\bullet \mathcal{U}_{1,1}$ (resp. $L^\bullet \Ugeom$) be the length filtration defined in \S \ref{lengthfiltration} (resp. \S \ref{filtrationsUgeom}).

\begin{proposition} \label{lengthmonodromy}
Under the monodromy action, the length filtration $L^\bullet \mathcal{U}_{1,1}$ (resp. $L^\bullet \Ugeom$) is a subfiltration of the induced  filtration $B^\bullet \mathcal{U}_{1,1}$ (resp. $B^\bullet \Ugeom$).
\end{proposition}
\begin{proof}
Let $u\in L^r \mathcal{U}_{1,1}^{\dR}(R)$. By linearity, we may assume that $u$ is a word in the generators \eqref{liealgconcrete}. If $u$ contains any of the symbols $\mathbf{e}_f', \mathbf{e}_f''$ for a Hecke eigenform $f$ then $\mu(u) = 0$ by Proposition \ref{explicitmonodromy}. Therefore we may assume that $u$ only contains Eisenstein symbols; it can then be written in the form
\begin{equation*}
u = \mathsf{X}^{k_1}\mathsf{Y}^{2 n_1-k_1} \mathbf{e}_{2 n_1 + 2}\cdots  \mathsf{X}^{k_s}\mathsf{Y}^{2 n_s-k_s} \mathbf{e}_{2n_s + 2}, \quad \text{where } s\geq r \text{ and }  0\leq k_i\leq 2 n_i.
\end{equation*}
Applying $\mu$ gives
\begin{equation*}
\mu(u) = \delta_{2n_1 + 2}^{(k_1)} \circ \cdots \circ \delta_{2 n_s +2}^{(k_s)}\in L^s \Ugeom(R),
\end{equation*}
where $\delta_{2n + 2}^{(k)} := \ad(\varepsilon_0^\vee)^{k} (\varepsilon_{2n +2}^\vee )$. Each $\delta_{2n+2}^{(k)}$ raises the $\bdr$-degree by at least $1$, so the $s$-fold composition raises the $\bdr$-degree by at least $s\geq r$.
\end{proof}

Proposition \ref{lengthmonodromy} is useful because it enables the detection of multiple zeta values of a certain depth within the length filtration on iterated integrals on $\M11$.

\section{The canonical cocycle and the series $\Psi$} \label{motivicS}

Let $S\in \Gamma$ be the matrix
\begin{equation*}
S =
\left( \begin{matrix}
0 & -1 \\
1 & 0
\end{matrix} \right).
\end{equation*}
 Holomorphic multiple modular values are (regularised) iterated integrals along $S$ when viewed as an element of $\pi_1^{\text{top}}(\M11, \vec{1}_\infty)$ \cite[Definition $5.2$]{mmv}. To work with these formally, it is necessary to understand the image of $S$ under the map $\Gamma\to \mathcal{G}_{1,1}^{\dR}(\mathcal{P}^{\mm}_{\mathcal{H}})$.

Let $f \colon B_3\to \Gamma$ be as in Definition \ref{fdef}. The functoriality of relative completion produces a commutative diagram
\begin{equation} \label{B3toSL2Zdiagram}
\begin{tikzcd}
B_3 \arrow[d, "f", swap] \arrow[r] & \mathcal{G}_{1,\vec{1}}^{\dR}(\mathcal{P}^{\mm}_\mathcal{H}) \arrow[d] \arrow[r, "\sim"] & (\mathcal{U}_{1,\vec{1}}^{\dR}\rtimes SL_2^{\dR})(\mathcal{P}^{\mm}_\mathcal{H}) \arrow[d]\\
 \Gamma \arrow[r] &  \mathcal{G}_{1,1}^{\dR}(\mathcal{P}^{\mm}_\mathcal{H}) \arrow[r, "\sim"] & \left(\mathcal{U}_{1,1}^{\dR}\rtimes SL_2^{\dR}\right)(\mathcal{P}^{\mm}_\mathcal{H})
\end{tikzcd}
\end{equation}
in which the left horizontal maps are the universal comparisons \eqref{universaldrmap} and the right horizontal maps make use of the decompositions \eqref{vectordecomp} and \eqref{sdsplit}.

\subsection{The canonical cocycle}

Composition along the bottom row of Diagram \eqref{B3toSL2Zdiagram} defines quantities $(\mathcal{C}_\gamma^{\mm}, \gamma^{\mm})\in\mathcal{U}_{1,1}^{\dR}(\mathcal{P}^{\mm}_{\mathcal{H}}) \rtimes SL_2^{\dR} (\mathcal{P}^{\mm}_{\mathcal{H}})$ for each $\gamma\in\Gamma$. The element $\mathcal{C}_\gamma^{\mm}$ depends upon the choice of splitting $\mathcal{G}_{1, 1}^{\dR} \cong  \mathcal{U}_{1,1}^{\dR} \rtimes SL_2^{\dR}$.

\begin{definition}[Canonical cocycle]
The map $\gamma\mapsto \mathcal{C}_\gamma^{\mm}$ is called the \emph{canonical cocycle} $\mathcal{C}^{\mm} \in Z^1 (\Gamma, \mathcal{U}_{1,1}^{\dR}(\mathcal{P}^{\mm}_{\mathcal{H}}))$ \cite[Definition $15.4$]{mmv}.
\end{definition}

The element $\mathcal{C}_\gamma^{\mm}$ can be viewed as a generating series for motivic iterated integrals along $\gamma\in\pi_1^{\text{top}}(\M11, \vec{1}_\infty)$. This motivates the following definition, which may be viewed as an analogue of $\Zm = \Co(\Phi_{0 1}^{\mm})$.

\begin{definition}[Motivic multiple modular values]
The $\mathbb{Q}$-algebra $M^{\mm}$ of motivic multiple modular values \cite{mmv} is the $\mathbb{Q}$-subalgebra of $\mathcal{P}^{\mm}_{\mathcal{H}}$ generated by $\Co(\mathcal{C}_\gamma^{\mm})$ for all $\gamma\in\Gamma$.
\end{definition}

The cocycle property implies that $M^{\mm}$ is generated by $\Co(\mathcal{C}^{\mm}_S)$ and $\Co(\mathcal{C}^{\mm}_T)$. It is known that $\Co(\mathcal{C}_T^{\mm}) \subseteq \mathbb{Q}[\mathbb{L}]$ \cite[Lemma $15.6$]{mmv}. Therefore the interesting motivic multiple modular values are contained in $\Co (\mathcal{C}_S^{\mm})$. Some of these coefficients are motivic iterated integrals (along $S$) of differential forms of the shape
\begin{equation*}
f(q) \log(q)^b \frac{dq}{q},
\end{equation*}
where $f$ is a modular form of weight $2n+2\geq 4$ and $0\leq b\leq 2n$. Such integrals are called \emph{totally holomorphic} motivic multiple modular values \cite[Definition $5.1$]{brownihara}.

\begin{definition}[Motivic iterated Eisenstein integrals]
When all $f$ are Eisenstein series $\mathbb{G}_{2n+2}$, these integrals are called motivic \emph{iterated Eisenstein integrals}. They are motivic iterated integrals of the form
\begin{equation} \label{generalEisint}
\int_S^{\mm} [E_{2n_1+2}(b_1)\vert \dots \vert E_{2n_s+2}(b_s)].
\end{equation}
Here $E_{2n+2}(b)\in\mathcal{O}(\mathcal{U}_{1,1}^{\dR})$ is dual to $\mathbf{e}_{2n+2}\mathsf{X}^a\mathsf{Y}^b\in\mathfrak{u}_{1,1}$, where $a+b = 2n$, and $[E_{2n_1+2}(b_1)\vert \dots \vert E_{2n_s+2}(b_s)]\in\mathcal{O}(\mathcal{U}_{1,1}^{\dR})$ is the element of the bar complex representing the iterated integral. The length of this integral is $s$ and the total modular weight is $N:=\sum_{i=1}^s (2n_i+2)$. 
\end{definition}

\begin{remark}[Betti and de Rham normalisations] \label{bettidrintegrals}
It is important to note that the iterated integrals considered here are de Rham normalised. The Betti normalisation differs by replacing $\log(q)$ by $\mathbb{L} \tau$ and $dq/q$ by $\mathbb{L} d\tau$. Therefore the de Rham integral
\begin{equation*}
\int^{\mm} [E_{2n_1 + 2} (b_1) \vert\cdots \vert E_{2n_s + 2} (b_s)] 
\end{equation*}
may be written in the Betti normalisation as
\begin{equation*}
\mathbb{L}^{b_1 + \dots + b_s + s}\int^{\mm} \mathbb{G}_{2n_1 + 2} (\tau_1) \tau_1^{b_1} d\tau_1 \cdots \mathbb{G}_{2n_s + 2} (\tau_s) \tau_s^{b_s} d\tau_s.
\end{equation*}
\end{remark}

\subsection{The series $\Psi$}

We now study the top row of Diagram \eqref{B3toSL2Zdiagram}. Recall that $\mathcal{U}_{1,\vec{1}}^{\dR}\cong\mathcal{U}_{1,1}^{\dR}\times\mathbb{G}_a$ where $\Lie(\mathbb{G}_a)$ is spanned by an element $\mathbf{e}_2$. This implies that the map $B_3 \to (\mathcal{U}_{1,\vec{1}}^{\dR}\rtimes SL_2^{\dR})(\mathcal{P}^{\mm}_\mathcal{H})$ given by composition along the top row may be written
\begin{equation*}
    \sigma \mapsto \left(\exp(r^{\mm}(\sigma)\mathbf{e}_2) \mathcal{C}_{f(\sigma)}^{\mm}, f(\sigma)^{\mm}\right).
\end{equation*}
Here $r^{\mm} \colon B_3 \to \mathcal{P}^{\mm}_{\mathcal{H}}$ is a motivic lift of a cocycle $r$ constructed by Matthes \cite{matthesquasi, matthesE2}. It takes values in the additive subgroup $\mathbb{Q}\mathbb{L}$ because $\mathbf{e}_2$ spans a copy of $\mathbb{Q}(1)$. We discuss this cocycle in \S \ref{G2Lvalue}.

\begin{definition}[Series $\Psi$]
Let $\tilde{S} = (t_A t_B t_A)^{-1}\in B_3$, so that $f(\tilde{S}) = S$. Composition along the top row of Diagram \eqref{B3toSL2Zdiagram} sends $\tilde{S}\mapsto (\Psi, S^{\mm})$, where $S^{\mm}$ is the de Rham normalisation of $S$ defined in \S\ref{SL2reps} and $\Psi = \exp(r^{\mm}(\tilde{S}) \mathbf{e}_2) \mathcal{C}_S^{\mm}$.
\end{definition}

\begin{remark}[Integrals along $S$ and $\tilde{S}$] \label{intsE2}
Let $\eta := r^{\mm}(\tilde{S})$. We have $\eta = \Psi[\mathbf{e}_2] = \int_{\tilde{S}}^{\mm} E_2 (0)$. We compute $\eta = \mathbb{L}/8$ in Corollary \ref{qvalue}.

The algebra $M^{\mm}$ is generated by $\Co(\mathcal{C}_S^{\mm})$ and $\Co(\mathcal{C}_T^{\mm}) = \mathbb{Q}[\mathbb{L}]$. Since $\Psi = \exp(\eta\mathbf{e}_2) \mathcal{C}_S^{\mm} $, where $\eta\in \mathbb{Q}\mathbb{L}$ and $\mathcal{C}_S^{\mm}$ is invertible, it follows that $M^{\mm} = \Co(\Psi)$. Elements of $\Co(\Psi)$ can be decomposed into elements of $\Co(\mathcal{C}_S^{\mm})$ multiplied by powers of $\eta$ as follows.

The projection $\mathcal{U}_{1,\vec{1}}^{\dR}\cong \mathcal{U}_{1,1}^{\dR}\times \mathbb{G}_a \to \mathcal{U}_{1,1}^{\dR}$ sending $\mathbf{e}_2 \mapsto 0$ is dual to an inclusion
\begin{equation*}
j \colon \mathcal{O}(\mathcal{U}_{1,1}^{\dR})\hookrightarrow \mathcal{O}(\mathcal{U}_{1,\vec{1}}^{\dR})\cong  \mathcal{O}(\mathcal{U}_{1,1}^{\dR})\otimes \mathbb{Q}[E_2 (0)]
\end{equation*}
satisfying $\Psi \circ j = \mathcal{C}_S^{\mm} \in \Hom (\mathcal{O}(\mathcal{U}_{1,1}^{\dR}), \mathcal{P}^{\mm}_{\mathcal{H}})$. This implies that $\int_{\tilde{S}}^{\mm} E_{2n+2}(b) = \int_S^{\mm} E_{2n+2}(b)$ whenever $n>0$. The integral of any element of $\mathcal{O}(\mathcal{U}_{1,\vec{1}}^{\dR})$ can therefore be decomposed as an integral along $S$ multiplied by a power of $\eta=\int_{\tilde{S}}^{\mm} E_2(0)\in\mathbb{Q}\mathbb{L}$.
\end{remark}

\subsection{The coefficients of $\Co(\mu(\Psi))$} \label{coeffsmuPsi}
Proposition \ref{explicitmonodromy} implies that $\mu$ vanishes on words involving a cuspidal symbol $\mathbf{e}_f'$ or $\mathbf{e}_f''$. This means that the coefficients of the series $\mu(\Psi)$ (and therefore $\mu(\mathcal{C}_S^{\mm})$) are \emph{a priori} motivic iterated Eisenstein integrals. However, both of these series may also be viewed as elements of $\Aut (\Pi_Y^{\dR})$, which is a pro-object of $\MT$ \cite{brownzetaelements, umem}. Therefore $\Co(\mu(\Psi))$ and $\Co(\mu(\mathcal{C}_S^{\mm}))$ are also $\mathbb{Q}$-subalgebras of $\mathcal{P}^{\mm}_{\MT} \cong \Zm [\mathbb{L}^\pm]$. In other words, the coefficient spaces of $\mu(\Psi)$ and $\mu(\mathcal{C}_S^{\mm})$ consist of all linear combinations of motivic iterated Eisenstein integrals that are equal to motivic MZVs (allowing for powers of $\mathbb{L})$.

The main goal of this paper is to show that $\Co(\mu(\Psi))$ is as large as possible within $\mathcal{P}^{\mm}_{\MT}$. We describe the strategy in more detail at the start of \S \ref{mainargument}.

\section{Relations between $\P1minus$ and the Tate curve} \label{genus0andTate}

In this section we study the relationships between the fundamental groups of $X=\P1minus$ and $Y = \Tatecurve$. This establishes a link between multiple zeta values and the periods of the infinitesimal punctured Tate curve.

\subsection{The Hain morphism} \label{hainmorphismsection}

The Hain morphism \cite[\S $12.2$, \S $16$-$18$]{hainkzb} \cite[\S $3.3$]{brownzetaelements} is a morphism of de Rham fundamental groups $\Pi_X^{\dR} \to \Pi_Y^{\dR}$. It is obtained by pulling back the universal elliptic Knizhnik-Zamolodchikov-Bernard (KZB) connection \cite{hainkzb, enriquezkzb} to $X$. One obtains a version of the Knizhnik-Zamolodchikov (KZ) connection \cite{kz} with modified residues at the poles $\left\{0,1,\infty\right\}$.

\begin{definition}[Hain morphism] \label{haindef}
The Hain morphism is a morphism of affine group schemes $\phi \colon \Pi_X^{\dR}\to\Pi_Y^{\dR}$ over $\mathbb{Q}$. It is equivalent to the continuous Lie algebra homomorphism $\phi \colon \Lie(\mathsf{x}_0, \mathsf{x}_1)^\wedge  \to \Lie(\adr,\bdr)^\wedge$ given by
\begin{align*}
 \mathsf{x}_0 & \mapsto \frac{\ad(\bdr)}{e^{\ad(\bdr)} - 1}(\adr) = \sum_{k\geq 0} \frac{B_k}{k!} \ad(\bdr)^k (\adr) \\
 \mathsf{x}_1 &  \mapsto -[\adr,\bdr].
\end{align*}
\end{definition}

\begin{remark}
A remarkable consequence of the motivic theory is that although there is no algebraic morphism $X\to Y$ defined over $\mathbb{Q}$, there is a topological map $X(\mathbb{C})\hookrightarrow Y(\mathbb{C})$ that still induces the Hain morphism of de Rham fundamental groups over $\mathbb{Q}$. The Hain morphism may be defined as follows: let $\nabla$ be the connection on the bundle over $X(\mathbb{C})$ obtained as the pullback of the universal KZB connection via $X(\mathbb{C})\hookrightarrow Y(\mathbb{C})\hookrightarrow \mathcal{E}^\times (\mathbb{C})$, and let $p\in\left\{0,1,\infty\right\}$. Then $\phi(\mathsf{x}_p) = \res_p (\nabla)$. 
\end{remark}

The following basic fact about $\phi$ follows from the proof of \cite[Lemma $3.3$]{brownzetaelements}.
\begin{proposition} \label{hainmorphisminjective}
The Hain morphism is injective.
\end{proposition}

\subsubsection{Dictionary between filtrations}

The Hain morphism relates the weight and depth filtrations on $R\langle\langle \mathsf{x}_0, \mathsf{x}_1 \rangle \rangle$ to the $A$ and $B$ filtrations on $R\langle\langle \adr, \bdr \rangle \rangle$. This is made precise in Lemma \ref{hainfiltration}, which generalises \cite[Lemma $3.3$]{brownzetaelements}.

\begin{lemma} \label{hainfiltration}
The Hain morphism induces \emph{strict} morphisms of filtered $R$-algebras
\begin{align*}
W^\bullet R\langle\langle \mathsf{x}_0, \mathsf{x}_1\rangle\rangle &\to A^\bullet R \langle\langle \adr,\bdr \rangle\rangle\\
D^\bullet R \langle\langle \mathsf{x}_0, \mathsf{x}_1\rangle\rangle &\to B^\bullet R \langle\langle \adr,\bdr \rangle\rangle.
\end{align*}
\end{lemma}
\begin{proof}
Let $w\in R\langle\langle \mathsf{x}_0, \mathsf{x}_1\rangle\rangle$. By linearity we may assume that $w\in M(\mathsf{x}_0, \mathsf{x}_1)$. Therefore we can write
\begin{equation} \label{wordform}
w = \mathsf{x}_0^{k_1} \mathsf{x}_1^{l_1} \dots \mathsf{x}_0^{k_n} \mathsf{x}_1^{l_n}, \quad \text{where } k_1, l_n\geq 0 \text{ and } k_2, \dots, k_n, l_1, \dots, l_{n-1} \geq 1.
\end{equation}
Definition \ref{haindef} implies that $\phi(\mathsf{x}_0) = \adr + d$, where $d\in D^1\mathbb{L}(\adr,\bdr)^\wedge$, and $\phi(\mathsf{x}_1) = -[\adr, \bdr]$. Applying $\phi$ to \eqref{wordform} gives
\begin{equation} \label{wordformphi}
\phi(w) = (-1)^l (\adr+d)^{k_1} [\adr,\bdr]^{l_1} \cdots (\adr+d)^{k_n} [\adr,\bdr]^{l_n}
\end{equation}
where $l := l_1 + \dots + l_n$. We have $\adr+d\in A^1\cap B^0$ and $[\adr,\bdr]\in D^1 \subseteq A^1\cap B^1$. Using this and comparing \eqref{wordform} and \eqref{wordformphi} gives
\begin{equation*}
w\in W^r R\langle\langle \mathsf{x}_0, \mathsf{x}_1\rangle\rangle \iff k_1 + \dots + k_n + l_1 + \dots + l_n \geq r \iff \phi(w)\in A^r R\langle\langle \adr, \bdr\rangle\rangle.
\end{equation*}
Similarly, we have
\begin{equation*}
w\in D^r R\langle\langle \mathsf{x}_0, \mathsf{x}_1\rangle\rangle \iff  l_1 + \dots + l_n \geq r \iff \phi(w)\in B^r R\langle\langle \adr, \bdr\rangle\rangle.
\end{equation*}
\end{proof}

\subsection{The series $\alpha^{\mm}$ and $\beta^{\mm}$} \label{alphabetasection}

Recall that $\Pi_Y = \langle\alpha, \beta\rangle$. Under the natural map $\Pi_Y \to \Pi_Y^{\dR} (\mathcal{P}^{\mm}_{\mathcal{H}})$, these elements are sent to power series
\begin{equation*}
\alpha^{\mm}, \beta^{\mm}\in\Pi_Y^{\dR}(\mathcal{P}^{\mm}_{\mathcal{H}})\cong \GrL (\mathcal{P}^{\mm}_{\mathcal{H}}\langle\langle \adr, \bdr \rangle\rangle).
\end{equation*}

Using the Hain morphism, it is possible to explicitly write down these series in terms of exponential series and the associators $\Phi^{\mm}_{i j}$ of \S  \ref{motivicdrinfeldassociators}. 

\begin{lemma} \label{alphabetaformula}
The series $\alpha^{\mm}$ and $\beta^{\mm}$ are given explicitly by
\begin{align*}
\alpha^{\mm} &= \phi(\Phi^{\mm}_{1 \infty}) e^{-\mathbb{L} \phi(\mathsf{x}_\infty)} \phi(\Phi_{\infty 1}^{\mm}),\\
\beta^{\mm} &= e^{-\frac{\mathbb{L}}{2} \phi(\mathsf{x}_1)} \phi(\Phi^{\mm}_{1 0}) e^{-\bdr} \phi (\Phi^{\mm}_{\infty 1}).
\end{align*}
\end{lemma}
\begin{proof}
These formulae are obtained from a topological model for $Y(\mathbb{C})$ in terms of a quotient of $X(\mathbb{C})$ \cite[\S $16$]{hainkzb}.\footnote{This is an infinitesimal version of the Jacobi uniformisation $E(\mathbb{C}) \cong \mathbb{C}^\times / q^{\mathbb{Z}}$ of an elliptic curve as $q \to 0$. Here $q = \exp(2\pi i \tau)$ and $E(\mathbb{C})\cong  \mathbb{C}/(\mathbb{Z}\oplus\mathbb{Z}\tau)$. This construction is not algebraic.} This induces a homomorphism $\phi_0 \colon  \Pi_X \to \Pi_Y$ of topological fundamental groups, and therefore a commutative diagram
\begin{equation*}
\begin{tikzcd}
\Pi_X \arrow[r, "\phi_0"] \arrow[d] & \Pi_Y \arrow[d] \\
\Pi_X^{\dR}(\mathcal{P}^{\mm}_{\MT}) \arrow[r, "\phi", swap] & \Pi_Y^{\dR}(\mathcal{P}^{\mm}_{\MT})
\end{tikzcd}
\end{equation*}
where the vertical arrows are the natural morphisms \eqref{universaldrmap}; they land in the $\mathcal{P}^{\mm}_{\MT}$-points of the respective fundamental groups because the affine rings of each are ind-objects in $\MT$. Define
\begin{equation} \label{topgammadef}
\gamma := \dch_{1\infty} \cdot \gamma_{\infty}^{-1} \cdot \dch_{\infty 1}\in\Pi_X.
\end{equation} 
Hain shows that $\phi_0 (\gamma) = \alpha$ \cite[\S $16$]{hainkzb}. Under the map $\Pi_X \to \Pi_X^{\dR}(\mathcal{P}^{\mm}_{\MT})$, we have
\begin{equation} \label{gammamotdef}
\gamma\mapsto \gamma^{\mm}:=\Phi^{\mm}_{1 \infty} e^{-\mathbb{L} \mathsf{x}_\infty}\Phi_{\infty 1}^{\mm}.
\end{equation}
It follows that $\alpha^{\mm} = \phi(\gamma^{\mm})$.

The element $\beta$ is not contained within $\im(\phi_0)$. Instead, it is constructed in two stages. First take $\delta_1 :=(\gamma_1^{-, +})^{-1} \cdot \dch_{1 0} \in \pi_1^{\text{top}}(X, \vec{1}_1, -\vec{1}_0)$ and $\delta_2 := \dch_{\infty 1} \in \pi_1^{\text{top}}(X, \vec{1}_\infty, -\vec{1}_1)$, where ${\gamma_p^{\pm, \mp}}$ represents the homotopy class of a counterclockwise semicircle from $\pm \vec{1}_p$ to $\mp \vec{1}_p$ on $X(\mathbb{C})$. Their images in the $\mathcal{P}^{\mm}_{\MT}$-points of the respective de Rham path torsors are $\delta_1^{\mm} = \exp(-\mathbb{L}\mathsf{x}_1/2) \Phi_{1 0}^{\mm}$ and $\delta_2^{\mm} = \Phi_{\infty 1}^{\mm}$.

The boundary circles around the punctures at $0$ and $\infty$ are then identified to obtain a space homotopic to $Y(\mathbb{C})$ (this identification is done without twisting; see \cite[\S $18.1$]{hainkzb}). The element $\beta$ is the result of first traversing $\delta_1$, then identifying the boundary circles, and then traversing $\delta_2$. The identification inserts a factor of $e^{-\bdr}$ \cite[\S $18.1$]{hainkzb}, giving $\beta^{\mm} = \delta_1^{\mm} e^{-\bdr} \delta_2^{\mm}$.
\end{proof}

\begin{remark}[Hexagon equation]
The element $\gamma$ is homotopic to the path
\begin{equation*}
{\gamma_1^{+, -}} \cdot {\dch_{1 0}} \cdot {\gamma_0} \cdot {\dch_{0 1}} \cdot {\gamma_1^{-, +}},
\end{equation*}
where ${\gamma_p^{\pm, \mp}}$ is defined as in the proof of Lemma \ref{alphabetaformula}. For this reason, or equivalently by making use of the hexagon equation \cite{drinfeld} for $\Phi_{i j}^{\mm}$, we may also write
\begin{equation*}
\alpha^{\mm} = e^{\frac{\mathbb{L}}{2} \phi(\mathsf{x}_1)} \phi (\Phi_{1 0}^{\mm}) e^{\mathbb{L} \phi(\mathsf{x}_0)} \phi (\Phi_{0 1}^{\mm})e^{\frac{\mathbb{L}}{2} \phi(\mathsf{x}_1)}.
\end{equation*}
\end{remark}

Lemma \ref{alphabetaformula} implies that the coefficients of the series $\alpha^{\mm}$ and $\beta^{\mm}$ are contained within the subalgebra $\mathcal{P} = \mathcal{P}^{\mm, +}_{\MT}$ of effective periods of mixed Tate motives over $\mathbb{Z}$. This follows because the MHS on $ \Pi_Y^{\dR}$ is mixed Tate \cite[\S $13.6$]{mmv}.

Formulae similar to those given in Lemma \ref{alphabetaformula} appear in various forms in the literature. They are the initial values of solutions to the universal elliptic KZB connection at the cusp \cite[Proposition $4.9$, Theorem $4.11$]{enriquezkzb}, \cite[Theorem $4.3$]{nils}, \cite[\S $3.1$]{lochakmatthesschneps}. In this case, their images under the period map are sometimes denoted $A_\infty$ and $B_\infty$ and comprise part of the data an \emph{elliptic associator} \cite{ellipticassociators}. The reader is warned that conventions differ regarding path multiplication and choice of tangential basepoints on $Y$.

Similar formulae also appear in the profinite context \cite[Theorem $3.4$]{nakamuratangential} when describing the Galois action on the \'{e}tale fundamental group of the punctured Tate curve.

\subsection{Bounds on coefficients of $\beta^{\mm}$ and $(\alpha^{\mm})^{-1}$} \label{boundssection}

The formula $\tilde{S}(\beta) = \alpha^{-1}$ given in \S \ref{topologicalSbeta}, together with the functoriality of relative completion, implies that
\begin{equation} \label{mainrelativeequation}
    \mu(\Psi)(S^{\mm}(\beta^{\mm})) = (\alpha^{\mm})^{-1}.
\end{equation}
This is an equation holding between noncommutative formal power series in the letters $\adr, \bdr$. By the discussion in \S \ref{coeffsmuPsi} the coefficients of $\mu(\Psi)$ are precisely those motivic iterated Eisenstein integrals that are contained in $\mathcal{P}^{\mm}_{\MT}$. The coefficients of $S^{\mm}(\beta^{\mm})$ and $(\alpha^{\mm})^{-1}$ are contained in $\mathcal{P}^{\mm}_{\MT}$ by Lemma \ref{alphabetaformula}.

Our main result, Theorem \ref{maintheorem}, essentially says that $\Co(\mu(\Psi))$ is as large as possible within $\mathcal{P}^{\mm}_{\MT}$. Its proof uses the equation \eqref{mainrelativeequation} to detect coefficients of $\mu(\Psi)$ by \emph{comparing} $\Co(S^{\mm}(\beta^{\mm}))$ with $\Co((\alpha^{\mm})^{-1})$. The ``difference'' between these is the contribution from $\Co(\mu(\Psi))$ by equation \eqref{mainrelativeequation}.

More precisely, in this section we compute bounds on $\Co_r^A (\beta^{\mm})$ and $\Co_r^B ((\alpha^{\mm})^{-1})$ in terms of multiple zeta values of a certain weight or depth. In \S \ref{mainargument} we use Proposition \ref{lengthmonodromy} to transfer this to information about $\Co_r^L (\mu(\Psi))$.

\subsubsection{The $A^r$-filtered pieces of $\beta^{\mm}$}

The following lemma gives an upper bound on the weights of motivic MZVs that may occur in $\Co_r^A(\beta^{\mm})$.

\begin{lemma} \label{betacoeffs}
For all $r\geq 0$ we have
\begin{equation*}
\Co_r^A(\beta^{\mm}) \subseteq \sum_{0\leq k\leq r} \left(\bigoplus_{0\leq j\leq k} \mathbb{L}^j\mathbb{Q} \right) \mathfrak{W}_{r-k}\Zm.
\end{equation*}
\end{lemma}
\begin{proof}
Recall that $\beta^{\mm} = e^{-\frac{\mathbb{L}}{2} \phi(\mathsf{x}_1)} \phi(\Phi^{\mm}_{1 0}) e^{-\bdr} \phi (\Phi^{\mm}_{\infty 1})$. We have
\begin{align} \label{betafiltration}
\Fil_A^r(\beta^{\mm}) &= \sum_{r_1 + \dots + r_4 = r} \Fil_A^{r_1} \left(e^{-\frac{\mathbb{L}}{2} \phi(\mathsf{x}_1)} \right) \Fil_A^{r_2} \left(\phi(\Phi^{\mm}_{1 0}) \right) \Fil_A^{r_3} \left( e^{-\bdr} \right) \Fil_A^{r_4} \left( \phi (\Phi^{\mm}_{\infty 1}) \right) \nonumber \\
& = \sum_{r_1 + r_2 + r_3 = r} \phi\left( \Fil_W^{r_1} \left(e^{-\frac{\mathbb{L}}{2} \mathsf{x}_1} \right)\right) \phi\left(\Fil_W^{r_2} \left(\Phi^{\mm}_{1 0} \right)\right) e^{-\bdr} \phi\left(\Fil_W^{r_3} \left(\Phi^{\mm}_{\infty 1} \right)\right).
\end{align}
The second line follows by applying Lemma \ref{hainfiltration} and noting that $e^{-\bdr}\in A^0$.

The Hain morphism is injective and defined over $\mathbb{Q}$. This implies that $\Co(\phi(s)) =\Co(s)$ for any $s\in \mathcal{P}\langle\langle \mathsf{x}_0, \mathsf{x}_1\rangle\rangle$. It therefore suffices to compute the coefficient spaces of the $W$-filtered pieces of the factors in \eqref{betafiltration}.

It is easy to see that $\Co_{r}^W (e^{-\mathbb{L} \mathsf{x}_1/2})  =\bigoplus_{j=0}^{r} \mathbb{L}^j\mathbb{Q}$ and $\Co\left(e^{-b}\right) =\mathbb{Q}$. Proposition \ref{associatorweights} states that for any $i \neq j$ we have $\Co_r^W(\Phi^{\mm}_{i j}) = \mathfrak{W}_r \Zm$. Therefore
\begin{equation} \label{coeffsum}
\Co_r^A (\beta^{\mm}) = \sum_{r_1+r_2+r_3 = r} {\left(\bigoplus_{j=0}^{r_1} \mathbb{L}^j\mathbb{Q}\right)} \cdot {\mathfrak{W}_{r_2} \Zm} \cdot {\mathfrak{W}_{r_3}\Zm}.
\end{equation}
Finally we use that $\mathfrak{W}_r \Zm \cdot\mathfrak{W}_s \Zm \subseteq \mathfrak{W}_{r+s} \Zm$, which gives the result.
\end{proof}

\subsubsection{The $B^r$-filtered pieces of $(\alpha^{\mm})^{-1}$}

In this section we compute a lower bound on $\Co_r^B\left((\alpha^{\mm})^{-1}\right)$ in terms of the depth filtration on motivic MZVs.

\begin{lemma} \label{alphacoeffs}
For all $r\geq 0$ we have $\mathbb{L}\mathfrak{D}_r \Zm \subseteq \Co_r^B(\alpha^{\mm})^{-1})$.
\end{lemma}

\begin{remark}[Proof strategy]
While the proof may appear complicated, the strategy is simple: compute an expression (equation \eqref{vcommapprox}) for $\Fil_D^r((\gamma^{\mm})^{-1})$ modulo products of MZVs of positive weight, where $\gamma^{\mm}$ was defined in \eqref{gammamotdef}. The formula has a leading term consisting of a generating series for motivic MZVs of the form $\zeta^{\mm}(w)$ where $w$ has depth precisely $r$, followed by a correction term whose coefficients are motivic MZVs of depth strictly less than $r$. This explicitly demonstrates that $\Co_r^D((\gamma^{\mm})^{-1})$ contains all MZVs of depth at most $r$, multiplied by $\mathbb{L}$. We then use that $\phi(\gamma^{\mm}) = \alpha^{\mm}$ and Lemma \ref{hainfiltration} to conclude.
\end{remark}

\begin{proof}[Proof of Lemma \ref{alphacoeffs}]
Let $\gamma^{\mm} \in \Pi_X^{\dR}(\mathcal{P})$ be as in \eqref{gammamotdef}. Recall that $\phi(\gamma^{\mm}) = \alpha^{\mm}$; therefore, by Lemma \ref{hainfiltration}, it suffices to show that 
\begin{equation} \label{toprove}
\mathbb{L}\mathfrak{D}_r \Zm \subseteq \Co_r^D ((\gamma^{\mm})^{-1}).
\end{equation}
This is trivially true if $r=0$, so from now on we assume that $r\geq 1$. Define $I :=  \mathcal{Z}_{>0}^{\mm} \mathcal{Z}_{>0}^{\mm}+\mathbb{L} \mathcal{Z}_{>0}^{\mm} \mathcal{Z}_{>0}^{\mm}$. This is an ideal of $\mathcal{P}$, and the induced homomorphism $\Pi_X^{\dR}(\mathcal{P})\to \Pi_X^{\dR}(\mathcal{P}/I)$ sends $(\gamma^{\mm})^{-1}$ to an element $\bar\gamma$.

We may regard $\Co_r^D(\bar{\gamma})$ as the subspace of $\Co_r^D((\gamma^{\mm})^{-1}))$ spanned by coefficients that are not products of two positive-weight MZVs. Therefore it suffices to prove \eqref{toprove} with $(\gamma^{\mm})^{-1}$ replaced by $\bar\gamma$, using the $D$-filtration defined on $(\mathcal{P}/I)\langle\langle \mathsf{x}_0, \mathsf{x}_1 \rangle\rangle$. We show this by direct calculation.

To begin, write $\Phi_{1\infty}^{\mm} = 1+u$ and $\Phi_{\infty 1}^{\mm} = 1+v$, where $u,v\in \mathcal{Z}_{> 0}^{\mm}\langle\langle \mathsf{x}_0, \mathsf{x}_1 \rangle\rangle$. The property $\Phi_{1\infty}^{\mm} \Phi_{\infty 1}^{\mm} = 1$ implies that $u\equiv -v \pmod{\mathcal{Z}^{\mm}_{>0}\mathcal{Z}^{\mm}_{>0}}$. Expanding in $(\mathcal{P}/I)\langle\langle \mathsf{x}_0,\mathsf{x}_1 \rangle\rangle$ gives
\begin{equation} \label{gammaapprox1b}
\bar{\gamma}\equiv (1-v) e^{\mathbb{L}\mathsf{x}_\infty} (1+v) \equiv e^{\mathbb{L}\mathsf{x}_\infty} + \mathbb{L}\ad(\mathsf{x}_\infty)( v)  \pmod{I}.
\end{equation}

We now compute $v$. Recall that $ \Phi_{\infty 1}^{\mm} = \tau_{0 \infty}(\Phi_{0 1}^{\mm})$, where
\begin{equation*}
\tau_{0 \infty} (\mathsf{x}_0) = \mathsf{x}_\infty = -(\mathsf{x}_0 + \mathsf{x}_1), \quad \tau_{0\infty}(\mathsf{x}_1) = \mathsf{x}_1.
\end{equation*}
It is clear from the formula that $\tau_{0\infty}$ preserves the depth filtration and that the associated depth-graded morphism is multiplication by $(-1)^{(\mathsf{x}_0\text{-degree})} = (-1)^{\weight-\depth}$. Filtering $\Phi_{0 1}^{\mm}$ by the depth gives
\begin{equation*}
\Phi_{0 1}^{\mm} = 1+\sum_{k\geq 1} \sum_{\depth(w)=k} \zeta^{\mm}(w)w.
\end{equation*}
We obtain
\begin{align*}
v &= \tau_{0\infty} (\Phi_{0 1}^{\mm}) - 1 \\
&\equiv \sum_{\depth(w)=r} (-1)^{\weight(w)-\depth(w)} \zeta^{\mm}(w) w + t \pmod{D^{r+1}\Zm \langle\langle \mathsf{x}_0, \mathsf{x}_1 \rangle\rangle},
\end{align*}
where $t\in\mathfrak{D}_{r-1}\Zm\langle\langle \mathsf{x}_0, \mathsf{x}_1 \rangle\rangle$ is the contribution from applying $\tau_{0 \infty}$ to $D^{r-1}\backslash D^r$.

Using that $\mathsf{x}_\infty \equiv -\mathsf{x}_0 \pmod{D^1}$, the final term in \eqref{gammaapprox1b} may be approximated:
\begin{equation} \label{vcommapprox}
\ad(\mathsf{x}_\infty)( v) \equiv -\ad(\mathsf{x}_0) \left( \sum_{\depth(w)=r} (-1)^{\weight(w)-r} \zeta^{\mm}(w) w \right) +t' \pmod{D^{r+1}}
\end{equation}
where $t'\in\mathfrak{D}_{r-1}\Zm\langle\langle \mathsf{x}_0, \mathsf{x}_1 \rangle\rangle$ is the contribution from applying $\tau_{0 \infty}$ to $D^{r-1}\backslash D^r$.

Consider the words in this sum beginning with the letter $\mathsf{x}_1$. As $w$ ranges over all words of depth $r$, the coefficients of these words range over the $\mathbb{Q}$-vector space of all admissible motivic MZVs of depth $r$ (see Remark \ref{admissibledef}). This vector space is equal to the space of all motivic MZVs of depth $r$ by shuffle-regularisation.

Induction on all $s< r$ implies that $\Co(t')$ contains $\mathbb{L}\mathfrak{D}_{r-1} \Zm$. Equations \eqref{gammaapprox1b} and \eqref{vcommapprox} imply that $\Co_r^D(\bar{\gamma})$ contains $\mathbb{L}\mathfrak{D}_{r}\Zm$, which completes the proof.
\end{proof}

\section{MZVs and iterated Eisenstein integrals} \label{mainargument}

This section is devoted to the proof the following theorem:
\begin{theorem} \label{maintheorem}
Every motivic multiple zeta value of weight $n$ and depth $r$ is a $\mathbb{Q}$-linear combination of motivic iterated Eisenstein integrals of the form \eqref{generalEisint} of length $s \leq r$ and total modular weight $N\leq n+s$, multiplied by $\mathbb{L}^m$, where $m := n - s - \sum b_i\geq 0$.
\end{theorem}

The main idea behind the proof of Theorem \ref{maintheorem} is to use the monodromy action of \S  \ref{monodromysection} to detect coefficients of $\mu(\mathcal{C}_S^{\mm})$ by comparing the coefficients of $\alpha^{\mm}$ and $\beta^{\mm}$. As relative completion is functorial, the topological equation $\tilde{S}(\beta) = \alpha^{-1}$ of \eqref{topmonodromyexplicit} induces the following equation at the level of $\mathcal{P}^{\mm}_{\mathcal{H}}$-points of $\Pi_Y^{\dR}$:
\begin{equation*}
(\Psi, S^{\mm})(\beta^{\mm}) :=\mu(\Psi)\left(S^{\mm}(\beta^{\mm})\right) =  (\alpha^{\mm})^{-1}.
\end{equation*}
This equation relates the coefficients of $(\alpha^{\mm})^{-1}$ and $S^{\mm}(\beta^{\mm})$, which are motivic MZVs, to the coefficients of $\mu(\Psi)$, which are motivic iterated Eisenstein integrals.

This equation can be filtered with respect to $B^\bullet \Pi_Y^{\dR} (\mathcal{P}^{\mm}_{\MT})$. By combining this with the bounds given in Lemmas  \ref{betacoeffs} and \ref{alphacoeffs}, we will show that for all $r\geq 0$
\begin{equation*}
\mathfrak{D}_r \Zm \subseteq \langle \text{iterated Eisenstein integrals of length} \leq r \rangle [\mathbb{L}^\pm].
\end{equation*}
The proof concludes by relating the modular weights to the MZV weights using the Hodge filtration on $\mathcal{O}(\mathcal{U}_{1,1}^{dR})$. This fixes the powers of $\mathbb{L}$ that can occur in such a linear combination. In particular, we show that a precise \emph{nonnegative} power of $\mathbb{L}$ is required for each term in this linear combination.

\begin{remark}
The weight filtration $\mathfrak{W}_\bullet \Zm$ is given by the length filtration of iterated integrals on $X=\P1minus$. Theorem \ref{maintheorem} therefore implies that there is a large reduction in the length of the iterated integral expressing a given multiple zeta value in passing from $X$ to $\M11$. 
\end{remark}

\subsection{Main argument}

The topological equation $\tilde{S}(\beta) = \alpha^{-1}$  of \S \ref{topologicalSbeta} induces the equation
\begin{equation} \label{motivicSactionbeta}
\mu(\Psi)\left(S^{\mm}(\beta^{\mm})\right) = (\alpha^{\mm})^{-1}.
\end{equation}

The proof of Theorem \ref{maintheorem} crucially relies on Proposition \ref{inequalitytheorem} below, which uses this equation to derive an inclusion relating the depth filtration on motivic MZVs to the length filtration on motivic iterated Eisenstein integrals.

\begin{proposition}  \label{inequalitytheorem}
Let $r\geq 0$ and let $H_r := \Co_r^L (\mu(\Psi))$. There is an inclusion
\begin{equation*}
\mathfrak{D}_r \Zm \subseteq H_r [\mathbb{L}^\pm].
\end{equation*}
\end{proposition}
\begin{proof}
By Proposition \ref{lengthmonodromy}, the length filtration $L^\bullet \Ugeom$ is a subfiltration of the induced filtration $B^\bullet \Ugeom$. Therefore, we may apply $\Co_r^B$ to both sides of \eqref{motivicSactionbeta} and apply Proposition \ref{coeffapplication} to obtain
\begin{align*}
\Co_r^B ((\alpha^{\mm})^{-1}) &= \sum_{i+j = r} \Co_i^L (\mu(\Psi)) \cdot \Co_j^B (S^{\mm} (\beta)) \\
& = \sum_{i+j = r} H_i \cdot \Co (S^{\mm} (\Fil_A^j (\beta^{\mm}))).
\end{align*}
We use Lemma \ref{SL2filteredpiece} in going from the first to the second line.

We now combine Lemmas \ref{betacoeffs} and \ref{alphacoeffs}  with the previous formula to obtain
\begin{equation} \label{refinedbound}
\begin{split}
\mathbb{L}\mathfrak{D}_r\Zm &\subseteq \Co_r^B ((\alpha^{\mm})^{-1}) \\ 
&=\sum_{i + j = r} H_i \cdot \Co (S^{\mm} (\Fil_A^j (\beta^{\mm}))) \\
&\subseteq \sum_{i + j = r} H_i \cdot \Co_j^A (\beta^{\mm})[\mathbb{L}^\pm] \\
& \subseteq \sum_{i + j = r} H_i \cdot \left(\sum_{0\leq k \leq j} \left( \bigoplus_{0\leq l \leq k} \mathbb{L}^l \mathbb{Q}\right) \mathfrak{W}_{j-k}\Zm\right)[\mathbb{L}^\pm] \\
&\subseteq \sum_{i+j = r} H_i \cdot \left(\sum_{0\leq k \leq j} \mathfrak{W}_{j-k}\Zm \right)[\mathbb{L}^\pm] \\
&\subseteq \sum_{i+j = r} H_i \cdot \mathfrak{W}_j\Zm [\mathbb{L}^\pm].
\end{split}
\end{equation}
The first line uses Lemma \ref{alphacoeffs}. To go from the second to third line, recall from \eqref{actionmotivicS} that $S^{\mm}(\adr, \bdr) = (-\mathbb{L}^{-1}\bdr, \mathbb{L} \adr)$. To go from the third to the fourth line, apply Lemma \ref{betacoeffs}. The remaining simplifications follow by allowing multiplication by arbitrary powers of $\mathbb{L}$.

We now use induction to show that $\mathfrak{D}_r\Zm \subseteq H_r [\mathbb{L}^\pm]$ for all $r$. The case $r=0$ is trivially true, and the case $r = 1$ is known by Brown's formula for the length $1$ part of $\mathcal{C}_S^{\mm}$ \cite[\S $15.4$]{mmv}. So let us assume that for all $k\leq r$ we have $\mathfrak{D}_k\Zm \subseteq H_k [\mathbb{L}^\pm]$. Then apply \eqref{refinedbound} and the induction hypothesis to obtain

\begin{align*}
\mathbb{L} \mathfrak{D}_{r+1}\Zm &\subseteq \sum_{\substack{i + j = r+1 \\ j > 0}} H_i \cdot \mathfrak{W}_j\Zm [\mathbb{L}^\pm] + H_{r+1} [\mathbb{L}^\pm] \\
& \subseteq \sum_{i + j = r} H_i \cdot H_j [\mathbb{L}^\pm] + H_{r+1} [\mathbb{L}^\pm] \\
&   \subseteq H_{r+1} [\mathbb{L}^\pm].
\end{align*}
The transition from the first to the second line follows from the trivial inclusion $\mathfrak{W}_j\Zm \subseteq \mathfrak{D}_{j-1}\Zm $, coming from the fact that every motivic MZV can be written as a linear combination of admissible MZVs, and then applying the induction hypothesis. Now invert the leading $\mathbb{L}$ to obtain $\mathfrak{D}_{r+1}\Zm \subseteq H_{r+1} [\mathbb{L}^\pm]$. This proves the inductive step and completes the proof.
\end{proof}

We are now able to prove Theorem \ref{maintheorem}.

\begin{proof}[Proof of Theorem \ref{maintheorem}]

By Proposition \ref{inequalitytheorem}, there is an inclusion $\mathfrak{D}_r \Zm \subseteq H_r[\mathbb{L}^\pm]$ for every $r\geq 0$. By construction, the space $H_r$ consists of motivic periods of $\mathcal{O}(\Ugeom)$. By Remark \ref{lengthfiltrationremark} and the fact that the $\mathcal{H}$-subobject $\mathcal{O}(\Ugeom) \subseteq \mathcal{O}(\mathcal{U}_{1,1}^{\mathcal{H}})$ is mixed Tate, the $\mathbb{Q}$-subspace $H_r \subseteq \mathcal{P}^{\mm}_{\mathcal{H}}$ is contained within
\begin{equation*}
    \langle \int_{\tilde{S}}^{\mm} [E_{2n_1 + 2}(b_1) \vert \cdots\vert E_{2n_s + 2}(b_s)] :s\leq r \rangle_{\mathbb{Q}} \cap \mathcal{P}^{\mm}_{\MT},
\end{equation*}
which uses the inclusion $\mathcal{P}^{\mm}_{\MT} \hookrightarrow \mathcal{P}^{\mm}_{\mathcal{H}}$ given in \eqref{motivicperiodinclusion}. We then use Remark \ref{intsE2} to split such integrals into products of nonnegative powers of $\mathbb{L}$ and iterated Eisenstein integrals along $S \in \Gamma \cong \pi_1^{\text{top}}(\M11, \vec{1}_\infty)$. This proves the first part of Theorem \ref{maintheorem}.

We now turn to determining the weight. The weight filtration is canonically split by the Hodge filtration in the mixed Tate case. This means that for any (ind-) object $V \in \MT$ there is a canonical isomorphism
\begin{equation*}
    \gr_n^M V^{\dR} \cong M_{2n} V^{\dR} \cap F^n V^{\dR}.
\end{equation*}
Let $\zeta^{\mm}(w)\in \mathfrak{W}_n \Zm$. It follows that $w\in M_{2n}\mathcal{O}({_0 \Pi_1^{\dR}}) \cap F^n \mathcal{O}({_0 \Pi_1^{\dR}})$.

Suppose further that $\zeta^{\mm} (w) \in \mathfrak{D}_r \Zm$. By the previous part of the proof we may write $\zeta^{\mm}(w)$ as a $\mathbb{Q}[\mathbb{L}^\pm]$-linear combination of the form
\begin{equation} \label{mpeq}
\zeta^{\mm}(w) = \sum \lambda \left[\mathcal{O}(\mathcal{U}^{\mathcal{H}}_{1,1}), [E_{2n_1 + 2}(b_1) \vert \cdots\vert E_{2n_s + 2}(b_s)], S \right]^{\mm}
\end{equation}
where each $s\leq r$ and $a_i + b_i = 2n_i $ for all $1\leq i\leq s$.\footnote{Each term in this linear combination has its own value of $s$ and choice of $n_i, a_i, b_i$.} The element $E_{2n_i + 2}(b_i)\in\mathcal{O}(\mathcal{U}_{1,1}^{\dR})$ is dual to $\mathbf{e}_{2n_i+2} \mathsf{X}^{a_i} \mathsf{Y}^{b_i}$, and spans $\mathbb{Q}(-b_i - 1)$. The coefficient $\lambda$ lies in $\mathbb{Q}\mathbb{L}^m$ for some $m\in\mathbb{Z}$. Equation \eqref{mpeq} uses the fact that $\mu$ is a morphism of MHS \cite[Proposition $15.1$]{hainhodge} and the functoriality relations for motivic periods to write motivic periods of $\mathcal{O}(\Ugeom)$ as periods of $\mathcal{O}(\mathcal{U}_{1,1}^{\mathcal{H}})$.

Let $v :=  [E_{2n_1+2 }(b_1) \vert \cdots \vert E_{2n_s+2}( b_s)]$. By construction, $v$ is contained within the mixed Tate Hopf subalgebra $\mathcal{O}(\Ugeom) \subseteq \mathcal{O}(\mathcal{U}_{1,1}^{\dR})$. Hence it must be contained within the de Rham realisation of a mixed Tate motive that is an iterated extension of the objects $\mathbb{Q}(0)$, $\mathbb{Q}(-b_1 - 1)$, $\dots$, $\mathbb{Q}(-b_s-1)$. It follows that
\begin{equation*}
\lambda v\in M_{2(m+s+b_1+\dots+b_s)} \mathcal{O}(\Ugeom) \cap F^{m+s+b_1+\dots + b_s} \mathcal{O}(\Ugeom).
\end{equation*}
Equation \eqref{mpeq} implies that each $\lambda v$ is also contained in $M_{2n} \mathcal{O}(\Ugeom) \cap F^n \mathcal{O}(\Ugeom)$. Using the splitting of the weight filtration above and comparing the degrees of the graded pieces we deduce that
\begin{equation} \label{weightcomp}
n = m+s+\sum_{i=1}^s b_i.
\end{equation}
We now show that $m\geq \sum a_i$, which implies the result.

The highest weight vector $\mathbf{e}_{2n+2}\mathsf{Y}^{2n}$ is contained in an effective motive, for all $n\geq 0$. Therefore the motivic integral of $E_{2n+2} (2n)$ along any element of the fundamental group is an effective motivic period.\footnote{See \cite[\S $3.3$]{motivicperiods} for a definition of effective $\mathcal{H}$-periods.} The other generators $\mathbf{e}_{2n+2}\mathsf{X}^k\mathsf{Y}^{2n-k}$ are obtained from it by applying the $\mathfrak{sl}_2$-generator $\mathsf{X}\partial/\partial\mathsf{Y}$ $k$ times. Since $\mathsf{X}$ spans $\mathbb{Q}(0)$ and $\mathsf{Y}$ spans $\mathbb{Q}(1)$, the motivic integral of $E_{2n+2}(2n-k)$ along any path is the quotient of an effective motivic period by $\mathbb{L}^k$. This is also true of iterated integrals of such forms. In order for $\Psi$ to be homogeneous of weight $0$, the motivic period $[\mathcal{O}(\mathcal{U}^{\mathcal{H}}_{1,1}), v, S]^{\mm}$ therefore must be of the form
\begin{equation*}
[\mathcal{O}(\mathcal{U}^{\mathcal{H}}_{1,1}), v, S]^{\mm} = \frac{\kappa}{\mathbb{L}^{\sum a_i}},
\end{equation*}
where $\kappa\in\mathcal{P}^{\mm}_{\mathcal{H}}$ is effective.

As $\zeta^{\mm}(w)$ is effective, each term on the right hand side of \eqref{mpeq} is also effective. It follows that each $\lambda\in\mathbb{Q}\mathbb{L}^m$ where $m\geq \sum a_i$. In particular $m\geq 0$, which implies that $\mathfrak{D}_r \Zm\subseteq H_r [\mathbb{L}]$. Combining the bound for $m$ with \eqref{weightcomp} gives
\begin{equation*}
n = m+s + \sum_{i=1}^s b_i \geq N - s,
\end{equation*}
where the total modular weight is
\begin{equation*}
N=\sum_{i=1}^s (2n_i + 2) = \sum_{i=1}^s (a_i + b_i + 2).
\end{equation*}
This gives the bound $N\leq n+s$ and completes the proof.
\end{proof}

\section{Galois-theoretic consequences}

\subsection{Modular generator for $\MT$}

The benefit of working at the level of motivic periods throughout is that it gives access to structure results for the category of mixed Tate motives over $\mathbb{Z}$. The following corollary of Theorem \ref{maintheorem} can be seen as a ``modular'' analogue of Brown's main result in \cite{brownmtm}.

\begin{theorem} \label{motivetheorem}
The group $G^{\dR}_{\MT}$ acts faithfully on the affine ring $\mathcal{O}(\Ugeom)$.
\end{theorem}

The central idea of the proof is to exhibit a specific element $\varepsilon \in \Ugeom (\mathcal{P}^{\mm}_{\MT})$ with sufficiently many coefficients to be able to distinguish elements of $G^{\dR}_{\MT} (\mathbb{Q})$. By Theorem \ref{maintheorem} the element $\varepsilon = \mu(\mathcal{C}_S^{\mm})$ has this property, since its coefficients contain all motivic multiple zeta values.

\begin{proof}[Proof of Theorem \ref{motivetheorem}]
By Proposition \ref{motivicremark} the affine ring $\mathcal{O}(\Ugeom)$ is an ind-object of $\MT$, so it is equipped with an action of the motivic Galois group $G^{\dR}_{\MT} \cong U_{\MT}^{\dR} \rtimes \mathbb{G}_m$. The Galois action on $\mathcal{O}(\Ugeom)$ is equivalent to the Galois action on the group of points
\begin{equation*}
\Ugeom (\mathcal{P}^{\mm}_{\MT})\cong \Hom (\mathcal{O}(\Ugeom), \mathcal{P}^{\mm}_{\MT}).
\end{equation*}
For $\varepsilon\in \Ugeom (\mathcal{P}^{\mm}_{\MT})$ and $g\in G^{\dR}_{\MT}(\mathbb{Q})$, this action is given by
\begin{equation*}
\left(\mathcal{O}(\Ugeom)\xrightarrow{\varepsilon} \mathcal{P}^{\mm}_{\MT} \right) \mapsto \left(\mathcal{O}(\Ugeom)\xrightarrow{g} \mathcal{O}(\Ugeom)\xrightarrow{\varepsilon} \mathcal{P}^{\mm}_{\MT} \right).
\end{equation*}
This, in turn, is equivalent to the Galois action on $\Ugeom (\mathcal{P}^{\mm}_{\MT})$ via the action on \emph{coefficients}; by viewing $\varepsilon$ as a series $\varepsilon = \sum_w \varepsilon_w w$ with $\varepsilon_w \in \mathcal{P}^{\mm}_{\MT}$ and $w$ ranging over a basis for $\mathcal{O}(\Ugeom)$, it is defined by $g(\varepsilon) = \sum_w g(\varepsilon_w) w$.

Let $\varepsilon = \mu(\mathcal{C}_S^{\mm})\in \Ugeom(\mathcal{P}^{\mm}_{\MT})$, and let $g\in U^{\dR}_{\MT}(\mathbb{Q})$. Suppose that $g(\varepsilon) = \varepsilon$. Then for each basis element $w\in \mathcal{O}(\Ugeom)$ we must have
\begin{equation} \label{gE}
g(\varepsilon_w) = \varepsilon_w.
\end{equation}
Since $g\in U^{\dR}_{\MT}(\mathbb{Q})$ it also follows that $g(\mathbb{L}) = \mathbb{L}$.

Theorem \ref{maintheorem} implies that there is a $G_{\MT}^{\dR}$-equivariant inclusion $\Zm\hookrightarrow\Co(\varepsilon)[\mathbb{L}]$. This means that any $\zeta^{\mm}(w) \in \Zm$ can be written as a linear combination
\begin{equation*}
\zeta^{\mm}(w) = \sum_w \lambda_w \varepsilon_w
\end{equation*}
where $w$ ranges over a set of basis elements for $\mathcal{O}(\Ugeom)$ and $\lambda_w\in\mathbb{Q}\mathbb{L}^{m_w}$ are zero for all but finitely many $w$, where the integer $m_w\geq 0$ is determined as in Theorem \ref{maintheorem}. Since $g$ fixes $\mathbb{L}$ it follows that $g(\lambda_w) = \lambda_w$ for each $w$.

Equation \eqref{gE} therefore implies that $g(\zeta^{\mm}(w)) = \zeta^{\mm}(w)$. But the action of $U^{\dR}_{\MT}$ on $\Zm$ is faithful by Brown's theorem, so $g = \id$. This proves that the action of $U^{\dR}_{\MT}$ on $\Ugeom(\mathcal{P}^{\mm}_{\MT})$ is free, and hence faithful. The compatibilities outlined above then imply that the action of $U_{\MT}^{\dR}$ on $\mathcal{O}(\Ugeom)$ is faithful.

The affine ring $\mathcal{O}(\Ugeom)$ is an ind-object in $\MT$. Its $M$-weight filtration is canonically split by the Hodge filtration. The multiplicative group $\mathbb{G}_m$ acts faithfully on the associated $M$-graded. It follows that the full motivic Galois group $G^{\dR}_{\MT} = U^{\dR} \rtimes \mathbb{G}_m$ also acts faithfully on $\mathcal{O}(\Ugeom)$.
\end{proof}

\begin{remark}
Equivalently, Theorem \ref{motivetheorem} says that the subcategory $\langle \mathcal{O}(\Ugeom)\rangle_\otimes$ generates the whole of $\MT$. Another equivalent statement is that the Lie algebra $\ugeom = \Lie(\Ugeom)$ also generates $\MT$.

It is \emph{a priori} surprising that these ``modular'' generators exist; instead of being built directly from the cohomology of projective spaces they emerge from the relative completion of the fundamental group of $\M11$ and its action on $\Pi_Y^{\dR}$.
\end{remark}

\subsection{Oda's conjecture} \label{odasection}

Oda's conjecture \cite{oda} was proven by Takao \cite{takao}, building on the work of Ihara, Matsumoto and Nakamura, as well as Brown's result \cite{brownmtm}. Although it is a statement about relative pro-$\ell$ completions of mapping class groups and their associated Galois action, the conjecture implies the following statement at the motivic level.

Let $L:=\Lie \Pi_Y^{\dR}\cong \Lie(\adr, \bdr)^\wedge$. It is a pro-object of $\MT$ and thus a $\mathfrak{k}$-module, where $\mathfrak{k} = \Lie(U^{\dR}_{\MT})$. Consequently, there is a homomorphism
\begin{equation*}
\rho\colon \mathfrak{k}\to \Der(L).
\end{equation*}
Let $\mathfrak{u}_{1,\vec{1}}:=\Lie \mathcal{U}_{1,\vec{1}}^{\dR}$. The monodromy action of \S \ref{monodromysection} induces a homomorphism
\begin{equation*}
\mu\colon \mathfrak{u}_{1,\vec{1}} \to \Der(L),
\end{equation*}
whose image is $\ugeom\oplus \mathbb{Q}\varepsilon_2$.

Proposition \ref{motivicremark} implies that $\ugeom$ is equipped with a $\mathfrak{k}$-action described by a homomorphism $\tilde{\rho} \colon \mathfrak{k} \to \Der(\ugeom)$. It is compatible with the $\mathfrak{k}$-action on $L$ via $\rho$ in the sense that $\tilde{\rho}(\sigma) = \ad(\rho(\sigma))$. Therefore, the image of $\rho$ normalises $\ugeom$.\footnote{This means that $[\im(\rho), \ugeom]\subseteq \ugeom$, or equivalently that $\sigma \mapsto \ad(\rho(\sigma))\vert_{\ugeom}$ defines a Lie algebra homomorphism $\mathfrak{k}\to \Der(\ugeom)$.} The Oda conjecture implies that $\rho$ induces an injection $\mathfrak{k} \hookrightarrow N(\ugeom)/\ugeom$, where $N$ denotes the normaliser. Since $\mathfrak{k}\subseteq N(\ugeom)/\ugeom$ is free, its adjoint action on $N(\ugeom)/\ugeom$ is faithful.

Theorem \ref{motivetheorem} can be regarded as an ``orthogonal'' statement to this result in the sense that the action of $\mathfrak{k}$ on $\ugeom$ is also faithful. This is surprisingly nontrivial, for the following reason. In \cite{brownzetaelements}, Brown showed that $\mathfrak{k}$ acts faithfully on $L$ and that there is a choice of generators $\sigma_{2n+1}\in\mathfrak{k}$ that act via
\begin{equation}\label{nakamurader}
\rho(\sigma_{2n+1}) \equiv \varepsilon_{2n+2}^\vee \pmod{W_{-2n-3}}.
\end{equation}
If $\ugeom$ had no relations, \eqref{nakamurader} would trivially imply that $\tilde{\rho}$ is injective, which is equivalent to the statement that $\mathfrak{k}$ acts faithfully on $\ugeom$. But $\ugeom$ is \emph{not} free, due to the Pollack relations, which pose a potential obstruction to a faithful Galois action. Theorem \ref{motivetheorem} implies that $\tilde{\rho}$ is still injective despite this obstruction.

\section{Example in depth $1$}

In this section we compare our results with Brown's formula for $\Fil_L^1 (\mathcal{C}_S^{\mm})$ \cite[\S $15.4$]{mmv}. Brown's formula explicitly demonstrates that all odd motivic zeta values appear in $\Co (\mu(\mathcal{C}_S^{\mm}))$. This comparison allows us to compute the motivic period $\eta$ appearing in the formula for $\Psi$ explicitly.

Doing so requires several computational tools, such as a depth $1$ version of the Baker-Campbell-Hausdorff formula, which has a connection to Euler's formula for the even zeta values $\zeta(2n)$. We also give an interpretation of the value of $\eta$ in terms of a cocycle for the braid group constructed by Matthes \cite{matthesquasi, matthesE2}.

It therefore seems worthwhile to include an account of the depth $1$ case here.

\subsection{Brown's computation of $\mathcal{C}_S^{\mm}$ in length $1$}

In \cite[\S $15.4$]{mmv}, Brown gives a formula for the length-$1$ part of the value at $S\in \Gamma$ of the motivic canonical cocycle $\mathcal{C}^{\mm}$. In the de Rham normalisation, his formula is
\begin{equation}\label{brownformula}
\Fil_L^1 (\mathcal{C}_S^{\mm}) = 1 + \sum_{n\geq 1} \mathbf{e}_{2n+2} \left[\frac{(2n)!}{2}\zeta^{\mm}(2n+1) \left(\frac{\mathsf{X}^{2n}}{\mathbb{L}^{2n}} - \mathsf{Y}^{2n} \right) + \mathbb{L} e_{2n+2, S}^0 \right] 
+ C
\end{equation}
where:
\begin{itemize}
\item $e_{2n+2, S}^0$  is a homogeneous polynomial of degree $2n$ in $\mathsf{X}$ and $\mathbb{L}\mathsf{Y}$. The degree of each individual variable in each term is odd. It is the $\mathbb{Q}(\mathbb{L})$-rational part of the cocycle for $\Gamma$ attached to $\mathbb{G}_{2n+2}$, evaluated at $S$.
\item $C\in\ker(\mu)$ is a contribution from cuspidal generators.
\end{itemize}

\subsection{Explicit calculation}

Recall that $\Psi = \exp(\eta \mathbf{e}_2) \mathcal{C}_S^{\mm}$, where $\eta \in\mathbb{Q}\mathbb{L}$. The element $\Psi$ satisfies $\mu(\Psi)(S^{\mm} (\beta^{\mm})) = (\alpha^{\mm})^{-1}$, and Proposition \ref{explicitmonodromy} implies that $\mu(\mathbf{e}_2) = 2\varepsilon_2$. Consequently, we have
\begin{equation*}
e^{2\eta \varepsilon_2}\mu(\mathcal{C}_S^{\mm})(S^{\mm}(\beta^{\mm})) = (\alpha^{\mm})^{-1} .
\end{equation*}
We now filter this equation modulo $B^2$. By Lemma \ref{SL2filteredpiece} and the fact that $\varepsilon_2 \in L^1$, we expand the exponential to obtain
\begin{align} \label{qformula}
\Fil_B^1 \left((\alpha^{\mm})^{-1}\right) & \equiv \Fil_L^1 (\mu(\mathcal{C}_S^{\mm})) \left( \Fil_B^1 (S^{\mm}(\beta^{\mm}))\right) + 2\eta\varepsilon_2(\Fil_L^0(\mathcal{C}_S^{\mm})(\Fil^0_B (S^{\mm}(\beta^{\mm}))) \nonumber \\
& \equiv \Fil_L^1 (\mu(\mathcal{C}_S^{\mm})) \left(S^{\mm} (\Fil_A^1 (\beta^{\mm}))\right) + 2\eta\varepsilon_2(e^{-\mathbb{L}a}) \pmod{B^2}.
\end{align}
This formula determines $\eta = \mathbb{L}/8$; in order to show this, we compare explicit expressions for the power series $\Fil_L^1 \left(\mu(\mathcal{C}_S^{\mm})\right)$, $\Fil_A^1 (\beta^{\mm})$ and $\Fil_B^1 \left((\alpha^{\mm})^{-1}\right)$.

\subsubsection{Formula for $\Fil_L^1 \left(\mu(\mathcal{C}_S^{\mm})\right)$}

\begin{proposition} \label{Zseries}
Let $Z:= 1 - \sum_{n\geq 1} \zeta^{\mm} (2n+1) \varepsilon_{2n+2}^\vee$. Then for all $w\in\mathbb{Q}\langle\langle a,b\rangle\rangle$ we have
\begin{equation*}
\Fil_L^1 (\mu(\mathcal{C}_S^{\mm}))(w)\equiv Z(w) \pmod{B^2}.
\end{equation*}
\end{proposition}
\begin{proof}
Using \eqref{brownformula}, Proposition \ref{explicitmonodromy} and Lemma \ref{highestlowest}, we compute
\begin{equation*}
\Fil_L^1(\mu(\mathcal{C}_S^{\mm})) = 1 + \sum_{n\geq 1} \zeta^{\mm}(2n+1) \left[ \left(\frac{\varepsilon_{2n+2}}{\mathbb{L}^{2n}} - \varepsilon_{2n+2}^\vee \right) + \mathbb{L} \mu\left(e_{2n+2, S}^0(\mathsf{X},\mathbb{L}\mathsf{Y})\mathbf{e}_{2n+2} \right) \right].
\end{equation*}
For any $w\in\mathbb{Q}\langle\langle \adr,\bdr\rangle\rangle$ we have
\begin{equation*}
\mu\left(e_{2n+2, S}^0\mathbf{e}_{2n+2}\right)(w)\equiv 0 \pmod{A^2\cap B^2}.
\end{equation*}
This is for degree reasons; the description of the polynomial $e_{2n+2, S}^0$ above implies that $\mu\left(e_{2n+2, S}^0 \mathbf{e}_{2n+2} \right)$ is a linear combination of derivations $\delta_{2n+2}^{(k)} = \ad(\varepsilon_0^\vee)^k\left( \varepsilon_{2n+2}^\vee \right)$
for $1\leq k\leq 2n-1$. Any such $\delta_{2n+2}^{(k)}$ with $k$ in this range raises both the $\adr$- and $\bdr$-degree by at least $2$.

Similarly, the definition of $\varepsilon_{2n+2}$ implies that $\varepsilon_{2n+2}(w)\in B^2$ for any $w$. Therefore the only terms in $\Fil_L^1(\mu(\mathcal{C}_S^{\mm}))$ that act nontrivially modulo $B^2$ are those involving $\varepsilon_{2n+2}^\vee$. The definition of $Z$ then implies that $\Fil_L^1(\mu(\mathcal{C}_S^{\mm}))(w)\equiv Z(w)\pmod{B^2}$ for all $w\in\mathbb{Q}\langle\langle \adr,\bdr\rangle\rangle$.
\end{proof}

\subsubsection{Formulae for $\Fil_A^1\left(\beta^{\mm}\right)$ and $\Fil_B^1\left((\alpha^{\mm})^{-1}\right)$}

Lemma \ref{alphabetaformula} gives explicit formulae for $\alpha^{\mm}$ and $\beta^{\mm}$. These can be used to compute the filtered pieces of these series.

\begin{proposition} \label{fil1}
We have
\begin{align*}
\Fil_A^1\left(\beta^{\mm} \right) &\equiv e^{\frac{\mathbb{L}}{2}[\adr,\bdr]} e^{-\bdr} \pmod{A^2}\\
\Fil_B^1 \left((\alpha^{\mm})^{-1} \right) &\equiv  e^{-\mathbb{L}\adr + \frac{\mathbb{L}}{2}[\adr,\bdr]} + \ad\left(e^{-\mathbb{L}\adr} \right)\left(t\right)\pmod{B^2} 
\end{align*}
where 
\begin{equation*}
t := \sum_{n\geq 2} (-1)^n \zeta^{\mm}(n)\ad(\adr)^n (\bdr).
\end{equation*}
\end{proposition}
\begin{proof}
The first statement follows immediately from the expression
\begin{equation*}
\beta^{\mm} = e^{-\frac{\mathbb{L}}{2} \phi(\mathsf{x}_1)} \phi(\Phi^{\mm}_{1 0}) e^{-\bdr} \phi(\Phi_{\infty 1}^{\mm}).
\end{equation*}
Lemma \ref{hainfiltration} implies that $\Fil_A^1\left(\phi(\Phi_{i j}^{\mm})\right) \equiv \phi \left(\Fil_W^1 (\Phi_{i j}^{\mm}) \right) \pmod{A^2}$. The associators are grouplike and have no linear term, so $\Fil_W^1(\Phi_{i j}^{\mm})=\Fil_W^0(\Phi_{i j}^{\mm})  = 1$. Thus the only contributions to $\Fil_A^1 \left( \beta^{\mm}\right)$ come from the exponential factors. The result follows by recalling that $\phi(\mathsf{x}_1) = -[\adr,\bdr]$.

To prove the second statement, recall that $\alpha^{\mm} = \phi (\gamma^{\mm})$, where $\gamma^{\mm} = \Phi_{1 \infty}^{\mm} e^{-\mathbb{L} \mathsf{x}_\infty} \Phi_{\infty 1}^{\mm}$. By Lemma \ref{hainfiltration} we have
\begin{equation*}
\Fil_B^1 \left((\alpha^{\mm})^{-1} \right) \equiv \phi\left(\Fil_D^1 \left((\gamma^{\mm})^{-1} \right)\right) \pmod{B^2}.
\end{equation*}
We therefore compute $(\gamma^{\mm})^{-1}$ modulo $D^2$ and then apply $\phi$.

The series $\Phi_{\infty 1}^{\mm}$ is obtained from $\Phi_{0 1}^{\mm}$ by exchanging, in each word, all occurrences of $\mathsf{x}_0$ for $\mathsf{x}_\infty = - (\mathsf{x}_0 + \mathsf{x}_1)$. The formula for $\Phi_{0 1}^{\mm}$ in depth $1$ \cite[\S $6.7$]{delignegoncharov} gives
\begin{align*}
\Phi_{\infty 1}^{\mm} &\equiv 1 + \sum_{n\geq 2} \zeta^{\mm}(n) \ad(\mathsf{x}_\infty)^{n-1} (\mathsf{x}_1) \pmod{D^2} \\
& \equiv 1-\sum_{n\geq 2} (-1)^n \zeta^{\mm}(n) \ad(\mathsf{x}_0) ^{n-1} (\mathsf{x}_1) \pmod{D^2}.
\end{align*}
Let $r := \sum_{n\geq 2} (-1)^n \zeta^{\mm}(n) \ad(\mathsf{x}_0) ^{n-1} (\mathsf{x}_1)\in D^1$. The property $\Phi_{1 \infty}^{\mm} \Phi_{\infty 1}^{\mm} = 1$ implies that $\Phi_{1\infty}^{\mm}\equiv e^r \pmod{D^2}$. Therefore
\begin{align*}
(\gamma^{\mm})^{-1}  &\equiv \Ad_{e^r}\left(e^{\mathbb{L}\mathsf{x}_\infty}\right)  = e^{\ad(r)}\left( e^{\mathbb{L} \mathsf{x}_\infty}\right)   \equiv e^{\mathbb{L} \mathsf{x}_\infty} + \ad\left(r\right)\left(e^{\mathbb{L} \mathsf{x}_\infty} \right) \pmod{D^2}.
\end{align*}
We now apply $\phi$ to this formula. Note that $\phi(r)\equiv - t \pmod{B^2}$, where $t\in B^1$ is defined in the statement of the proposition. This gives
\begin{align*}
(\alpha^{\mm})^{-1} = \phi \left(( \gamma^{\mm})^{-1}\right) & \equiv e^{\mathbb{L} \phi(\mathsf{x}_\infty)} - \ad(t)( e^{\mathbb{L} \phi(\mathsf{x}_\infty)}) \pmod{B^2} \\
& \equiv  e^{-\mathbb{L}\adr + \frac{\mathbb{L}}{2}[\adr,\bdr]} + \ad\left(e^{-\mathbb{L}\adr} \right)\left(t\right)\pmod{B^2}.
\end{align*}
The second line follows by noting that $\ad(e^{\mathbb{L}\phi(\mathsf{x}_\infty)})(t)\equiv \ad(e^{-\mathbb{L}\adr})(t)\pmod{B^2}$, because $t\in B^1$.
\end{proof}

\subsection{Formal Lie algebraic computations}

Before proving Theorem \ref{depth1check} we record the following elementary formal computations.

\begin{lemma} \label{adjointformula}
Let $X,Y$ be elements in a Lie algebra $\mathfrak{g}$ over a field of characteristic zero, and let $k\geq 1$. Then in the universal enveloping algebra $U(\mathfrak{g})$ we have
\begin{equation*} 
\ad(X^k)(Y) = \sum_{r=1}^k X^{r-1} \ad(X)(Y) X^{k-r}.
\end{equation*}
\end{lemma}
\begin{proof}
Use the relation $\ad(X)(Y) = XY - YX$ in $U(\mathfrak{g})$ and induct on $k$.
\end{proof}

We also need the following ``depth-$1$'' version of the Baker-Campbell-Hausdorff formula. A similar formula was used in \cite{nils} to compute the meta-abelian logarithm of the elliptic associator.
\begin{lemma}[Truncated BCH formula] \label{bch}
Let $X$ and $Y$ be elements in a Lie algebra $\mathfrak{g}$ over a field of characteristic zero. In the universal enveloping algebra we have
\begin{equation*}
e^X e^Y \equiv e^{X+Y} + \ad\left(e^X \right) \left(\sum_{n\geq 1} \frac{B_n^+}{n!} \ad(X)^{n-1} (Y) \right) \pmod{Y^2},
\end{equation*}
where $B_n^+ = (-1)^n B_n$ is the sequence of Bernoulli numbers with $B_1^+ = 1/2$.
\end{lemma}
\begin{proof}
The classical BCH formula \cite[Corollary $3.24$]{freeliealgebras} gives
\begin{equation*}
\log(e^X e^Y) \equiv X + Y + \sum_{n\geq 1} \frac{B_n}{n!} \ad(Y)^n (X) \pmod{X^2}.
\end{equation*}
Taking exponentials, inverting both sides and replacing $(X,Y)$ by $(-Y, -X)$ gives
\begin{equation*}
e^X e^Y \equiv \exp\left(X + Y + \tilde{Y}\right)  \pmod{Y^2}, \quad \text{where } \tilde{Y} :=  \sum_{n\geq 1} \frac{B_n^{+}}{n!} \ad(X)^n (Y).
\end{equation*}
Expanding the exponential, we obtain
\begin{align*}
e^{X+Y+\tilde{Y}} &= \sum_{k\geq 0} \frac{1}{k!}\left(X+Y+\tilde{Y}\right)^k \\
& \equiv \sum_{k\geq 0} \frac{(X+Y)^k}{k!} + \sum_{k\geq 1} \frac{1}{k!} \sum_{r=1}^k (X+Y)^{r-1} \tilde{Y} (X+Y)^{k-r}\pmod{Y^2} \\
&\equiv e^{X+Y} + \sum_{k\geq 1}\frac{1}{k!} \sum_{r=1}^k X^{r-1}\tilde{Y}X^{k-r} \pmod{Y^2} \\
&=e^{X+Y} + \sum_{n, k\geq 1} \frac{B_n^+}{n!}\frac{1}{k!} \sum_{r=1}^k X^{r-1} \ad(X)^n (Y) X^{k-r} \\
&=e^{X+Y} + \sum_{n, k\geq 1} \frac{B_n^+}{n!}\frac{1}{k!} \ad(X^k)\left(\ad(X)^{n-1}(Y) \right) \\
& = e^{X+Y} + \ad\left(e^X - 1\right)\left( \sum_{n\geq 1} \frac{B_n^+}{n!} \ad(X)^{n-1}(Y) \right) \\
& = e^{X+Y} + \ad\left(e^X\right)\left( \sum_{n\geq 1} \frac{B_n^+}{n!} \ad(X)^{n-1}(Y) \right).
\end{align*}
In going from the fourth to the fifth line we use Lemma \ref{adjointformula}, where $Y$ has been replaced by $\ad(X)^{n-1} (Y)$.
\end{proof}

\subsection{Main calculation}

In this section we will prove the following result:

\begin{theorem} \label{depth1check}
We have
\begin{equation*}
 \Fil_B^1 \left((\alpha^{\mm})^{-1}\right) \equiv \Fil_B^1 \left(\mu(\mathcal{C}_S^{\mm})(S^{\mm} (\beta^{\mm})) \right) + \frac{\mathbb{L}}{4} \varepsilon_2 \left(S^{\mm}(\Fil_A^0(\beta^{\mm}))\right) \pmod{B^2}.
\end{equation*}
\end{theorem}

\begin{remark} \label{depth1checkremark}
Note that $S^{\mm} (\Fil_A^0 (\beta^{\mm}))=e^{-\mathbb{L}\adr}$ and $S^{\mm} (\Fil_A^1 (\beta^{\mm}))=e^{\mathbb{L}[\adr, \bdr]/2} e^{-\mathbb{L}\adr}$. To simplify notation, write $t = t_{\even} - t_{\odd}$, where
\begin{equation*} \label{teventodd}
t_{\even}:=\sum_{n\geq 1 } \zeta^{\mm} (2n) \ad(\adr)^{2n} (\bdr), \quad t_{\odd} := \sum_{n\geq 1} \zeta^{\mm}(2n+1) \ad(\adr)^{2n+1} (\bdr).
\end{equation*}
\end{remark}

\begin{proposition} \label{Zactingonbeta1}
We have
\begin{equation*} 
\Fil_B^1 \left(\mu(\mathcal{C}_S^{\mm})(S^{\mm} (\beta^{\mm})) \right) \equiv e^{\frac{\mathbb{L}}{2}[\adr,\bdr]} e^{-\mathbb{L}\adr} - \ad\left(e^{-\mathbb{L}\adr} \right)\left(t_{\odd} \right) \pmod{B^2}.
\end{equation*}
\end{proposition}
\begin{proof}
This is a straightforward calculation, working modulo $B^2$:
\begin{align*}
\Fil_B^1 \left(\mu(\mathcal{C}_S^{\mm})(S^{\mm} (\beta^{\mm})) \right) & \equiv \Fil_L^1(\mu(\mathcal{C}_S^{\mm}))(S^{\mm} (\Fil_A^1(\beta^{\mm}))) \\
& \equiv Z\left(e^{\frac{\mathbb{L}}{2}[\adr,\bdr]}e^{-\mathbb{L}\adr}\right) \\
 &\equiv  e^{\frac{\mathbb{L}}{2}[\adr,\bdr]} Z(e^{-\mathbb{L}\adr})  \\
& \equiv e^{\frac{\mathbb{L}}{2}[\adr,\bdr]} \left(e^{-\mathbb{L}\adr} - \sum_{n\geq 1} \zeta^{\mm}(2n+1) \varepsilon_{2n+1}^\vee (e^{-\mathbb{L}\adr}) \right) \\
& \equiv e^{\frac{\mathbb{L}}{2}[\adr,\bdr]} e^{-\mathbb{L}\adr} - \sum_{n,k\geq 1}\frac{(-\mathbb{L})^k}{k!} \zeta^{\mm}(2n+1) \sum_{r=1}^k \adr^{r-1}\ad(\adr)^{2n+2}(\bdr) \adr^{k-r} \\
& = e^{\frac{\mathbb{L}}{2}[\adr,\bdr]} e^{-\mathbb{L}\adr} - \ad\left(e^{-\mathbb{L}\adr} \right)\left(t_{\odd}\right). 
\end{align*}
The first line follows from Lemma \ref{lengthmonodromy}. The second line follows from Propositions \ref{Zseries} and \ref{fil1}. In the third line, the prefactor $e^{\mathbb{L}[\adr,\bdr]/2}$ pulls through the action of $Z$ because every term in $Z$ is a geometric derivation $\varepsilon_{2n+2}^\vee$, all of which annihilate $[\adr,\bdr]$. The fourth line follows from the definition of $Z$ given in Proposition \ref{Zseries}. The fifth line follows by computing $\varepsilon_{2n+2}^{\vee} \left(e^{-\mathbb{L}\adr} \right)$ explicitly. We also reduce the product of the prefactor and the sum over $n$ modulo $B^2$, because $e^{\mathbb{L}[\adr,\bdr]/2} -1\in B^1$. The sixth line follows by applying Lemma \ref{adjointformula} with $X = \adr$ and $Y = \ad(\adr)^n(\bdr)$.
\end{proof}

We may now apply the BCH formula of Lemma \ref{bch} to write $e^{\frac{\mathbb{L}}{2}[\adr,\bdr]} e^{-\mathbb{L}\adr}$ in terms of $e^{-\mathbb{L}\adr + \frac{\mathbb{L}}{2}[\adr,\bdr]}$.

\begin{proposition} \label{appliedbch}
We have
\begin{equation*}
e^{\frac{\mathbb{L}}{2}[\adr,\bdr]} e^{-\mathbb{L}\adr} \equiv e^{-\mathbb{L}\adr+\frac{\mathbb{L}}{2}[\adr,\bdr]} + \ad\left(e^{-\mathbb{L}\adr}\right) \left(t_{\even} \right) -\frac{ \mathbb{L}}{4}\varepsilon_2\left(e^{-\mathbb{L}\adr}\right) \pmod{B^2}.
\end{equation*}
\end{proposition}
\begin{proof}
To apply Lemma \ref{bch} we must work with the conjugate $e^{-\mathbb{L}\adr}e^{\mathbb{L}[\adr,\bdr]/2}$ instead of $e^{\mathbb{L}[\adr,\bdr]/2} e^{-\mathbb{L}\adr}$. They can be compared by the formula
\begin{equation*}
e^Y e^X \equiv e^X e^Y - \ad(e^X) (Y) \pmod{Y^2}
\end{equation*}
with $X = - \mathbb{L}\adr$ and $Y = \mathbb{L} [\adr, \bdr]/2$, to obtain\footnote{Recall that $\varepsilon_2 = -\ad([\adr,\bdr])$ and that the degree in $Y$ corresponds to the $B$-filtration.}
\begin{equation} \label{local1}
e^{\mathbb{L}[\adr,\bdr]/2}e^{-\mathbb{L}\adr} \equiv e^{-\mathbb{L}\adr} e^{\mathbb{L}[\adr,\bdr]/2} - \frac{\mathbb{L}}{2} \varepsilon_2 (e^{-\mathbb{L}\adr}) \pmod{B^2},
\end{equation}
Now apply the BCH formula of Lemma \ref{bch}. This gives
\begin{equation} \label{local2}
e^{-\mathbb{L}\adr}e^{\frac{\mathbb{L}}{2}[\adr,\bdr]} \equiv e^{-\mathbb{L}\adr+\frac{\mathbb{L}}{2}[\adr,\bdr]} + \ad\left(e^{-\mathbb{L}\adr}\right) \left(t_{\even} \right) +\frac{\mathbb{L}}{4}\varepsilon_2\left(e^{-\mathbb{L}\adr}\right) \pmod{B^2}.
\end{equation}
The term involving $\varepsilon_2$ appears by writing the $n=1$ term of the sum over $n$ in Lemma \ref{bch} separately and then rearranging. The remaining terms in the sum over $n$ are only nonzero when $n$ is even due the Bernoulli numbers. This sum can be shown to equal $t_{\even}$ by applying Euler's formula for the even zeta values
\begin{equation*}
\zeta^{\mm}(2n) = - \frac{B_{2n} \mathbb{L}^{2n}}{2(2n)!}.
\end{equation*}
Combining \eqref{local1} and \eqref{local2} gives the result.
\end{proof}

\begin{proof}[Proof of Theorem \ref{depth1check}]
By combining Proposition \ref{Zactingonbeta1} and Proposition \ref{appliedbch}, we obtain
\begin{equation*}
\Fil_B^1 \left(\mu(\mathcal{C}_S^{\mm})(S^{\mm} (\beta^{\mm})) \right) + \frac{\mathbb{L}}{4} \varepsilon_2\left(e^{-\mathbb{L}\adr}\right) \equiv e^{-\mathbb{L}\adr + \frac{\mathbb{L}}{2}[\adr,\bdr]} + \ad\left(e^{-\mathbb{L}\adr}\right)\left( t \right)\pmod{B^2}.
\end{equation*}
But the expression on the right hand side is precisely $\Fil_B^1((\alpha^{\mm})^{-1})$, as computed in Proposition \ref{fil1}.
\end{proof}

\subsection{Interpretation of $\eta$} \label{G2Lvalue}

Comparing Theorem \ref{depth1check} to equation \eqref{qformula} gives a precise value for $\eta$:

\begin{corollary} \label{qvalue}
We have $\eta = \mathbb{L}/8$.
\end{corollary}

In this section we explain the significance of this value and how it may be computed by alternative analytic means.

The element $\mathbf{e}_2$ is dual to the cohomology class of a form $\psi_2 \in \Omega^1 (\mathcal{M}_{1,\vec{1}})$ \cite[\S $14$]{hainhodge}. The form $\psi_2$ may be written as the following form defined on the partial cover $\mathbb{C}^\times \times \mathfrak{H} \to \mathcal{M}_{1,\vec{1}}$,
\begin{equation*}
\underline{\mathbb{G}_2}(\xi, \tau) = 2\pi i \mathbb{G}_2 (\tau) d\tau - \frac{1}{2}\frac{d\xi}{\xi},
\end{equation*}
by checking it is preserved under the $\Gamma$-action $(\xi, \tau)\mapsto ((c\tau+d)^{-1}\xi, \gamma(\tau))$. Here $\mathbb{G}_2 (\tau) := -1/24 + \sum_{n \geq 1} \sigma(n)q^n$ is the Eisenstein series of weight $2$. It is a quasimodular form for $\Gamma$ of weight $2$, and satisfies the modular transformation property
\begin{equation*}
    \mathbb{G}_2(\gamma(\tau)) = (c\tau + d)^2 \mathbb{G}_2(\tau) + \frac{i c(c\tau + d)}{4 \pi}.
\end{equation*}

The universal cover of $\mathcal{M}_{1,\vec{1}}$ factors as $\mathbb{C}\times \mathfrak{H}\to \mathbb{C}^\times \times \mathfrak{H} \to \mathcal{M}_{1,\vec{1}}$, where the first map is $\exp\times\id$ and the second identifies $\Gamma$-equivalent points. The coordinates on the universal cover are $(z, \tau)$. Pulling back $\underline{\mathbb{G}_2} (\xi, \tau)$ to the universal cover produces the $B_3$-invariant form
\begin{equation*}
\underline{\mathbb{G}_2} (z, \tau) = 2\pi i \mathbb{G}_2 (\tau) d\tau - \frac{1}{2}dz \in \Omega^1 (\mathbb{C}\times\mathfrak{H}).
\end{equation*}
The integral of $\underline{\mathbb{G}_2}(z,\tau)$ over $\tilde{S}$ is the coefficient of $\mathbf{e}_2$ in $\per(\Psi)$, which is
\begin{equation} \label{numerology}
 \int_{\tilde{S}}\underline{\mathbb{G}_2}(\xi, \tau) = \per(\Psi)[\mathbf{e}_2] =  \per(\eta) =  \frac{2\pi i}{8}.
\end{equation}

This equals the value, at $\tilde{S}$, of a cocycle $r \in Z^1 (B_3, \mathbb{C})$ constructed by Matthes \cite{matthesquasi}. Using our conventions, which differ slightly from those of Matthes,\footnote{Matthes' convention for the $SL_2(\mathbb{Z})$-action on $\mathbb{C}^\times \times \mathfrak{H}$ differs from the one used here by the reciprocal in the first factor. This only changes the cocycle values by a possible sign. He also defines the cocycle in terms of $E_2 = -24\mathbb{G}_2$.} it is defined as the map $r\colon B_3\to \mathbb{C}$ given by
\begin{equation} \label{E2cocycledef}
\gamma\mapsto  \int_{\gamma^{-1} (0,\vec{1}_\infty)}^{(0, \vec{1}_\infty)} \underline{\mathbb{G}_2} (z,\tau),
\end{equation}
where $(0, \vec{1}_\infty)$ is a tangential basepoint on $\mathbb{C}\times\mathfrak{H}$ and elements of $B_3$ act on the universal cover as deck transformations. It can be viewed as the limiting $n=0$ case of a family of cocycles $r_{2n+2} \in Z^1 (B_3, V_{2n}\otimes\mathbb{C})$. These are defined for $n>0$ by sending $\gamma\in B_3$ to value of the period polynomial associated to $\mathbb{G}_{2n+2}$ at $f(\gamma)\in \Gamma$, where $f\colon B_3\to \Gamma$ is defined in Definition \ref{fdef}.

The braid group is a central extension
\begin{equation*}
1\to \mathbb{Z} \to B_3 \xrightarrow{f} \Gamma\to 1.
\end{equation*}
Here the infinite cyclic group is the group of automorphisms of the covering $\exp \times \id\colon \mathbb{C}\times \mathfrak{H} \to \mathbb{C}^\times \times \mathfrak{H}$. This may be identified with the fundamental group of $\mathbb{C}^\times$, which is generated by a counterclockwise loop around $0$ denoted by $\sigma$.

Proposition \ref{kernelf} states that $\ker(f)$ is generated by $\tilde{S}^4$, and hence $\tilde{S}^4 = \sigma^{\pm 1}$. The braid group acts trivially on $\mathbb{C}$ which implies that $r\in \Hom(B_3, \mathbb{C})$. It follows that
\begin{equation} \label{rvalues}
r(\tilde{S}) = \pm \frac{r(\sigma)}{4}.
\end{equation}
It therefore suffices to compute $r(\sigma)$. Matthes' original computation \cite{matthesE2} is unpublished, so we record it again here using our own conventions.

\begin{proposition}
We have
\begin{equation*}
    r(\sigma) = -\frac{2\pi i}{2}.
\end{equation*}
\end{proposition}
\begin{proof}
We simply compute the action on the universal covering space and then apply \eqref{E2cocycledef}. The generator $\sigma$ acts on $\mathbb{C}\times \mathfrak{H}$ via $\sigma(z, \tau) = (z+2\pi i, \tau)$. Since it fixes the second coordinate, we have
\begin{equation*}
    r(\sigma) = \int_{(-2\pi i, \vec{1}_\infty)}^{(0, \vec{1}_\infty)} \left(2\pi i \mathbb{G}_2 (\tau)d\tau-\frac{1}{2} dz \right) = -\frac{1}{2}\int_{-2\pi i}^0 dz = -\frac{2\pi i}{2}.
\end{equation*}
\end{proof}

By \eqref{rvalues} it follows that $r(\tilde{S}) = \pm (2\pi i)/8$. Comparing to \eqref{numerology} shows that the two values agree up to a choice of sign. This reason for this is because the coefficient of $\mathbf{e}_{2n+2}$ in $\per(\Fil_L^1(\Psi))$ is $r_{2n+2} (\tilde{S})$, for $n> 0$, and the coefficient of $\mathbf{e}_2$ is the value at $\tilde{S}$ of a cocycle for $B_3$ associated to $\mathbb{G}_2$ that must be compatibly normalised in order for $\Psi$ to be grouplike. Our choice of normalisation, together with the fact that $H^1 (B_3, \mathbb{C})$ is one-dimensional, implies that this cocycle is $r$.

\section{Coefficients in linear combinations of iterated Eisenstein integrals} \label{coefficientsection}

The proof of Theorem \ref{maintheorem} is nonconstructive. However, it is possible to determine some of the coefficients in a linear combination of iterated Eisenstein integrals equal to a given multiple zeta value using additional information. One such source of information is the $f$-alphabet decomposition \cite{browndecomposition}\cite[\S $22$]{mmv}. This assigns an element of the shuffle algebra $\mathbb{Q}[\mathbb{L}^\pm]\otimes \mathbb{Q}\langle f_{2n+1}:n\geq 1 \rangle$ to a given motivic multiple zeta value, or to a ``mixed Tate'' linear combination of motivic iterated Eisenstein integrals. Although this depends on some choices, the highest coradical-degree piece in this decomposition is canonical.

In this section we use this idea to determine the coefficient of the highest-length iterated Eisenstein integrals in the decompositions of some example MZVs. This uses both the rich combinatorics of the $f$-alphabet decomposition for motivic MZVs and the modular theory of the canonical cocycle $\mathcal{C}^{\mm}$. 

\subsection{$f$-alphabet decomposition}

In this section we define the notion of a decomposition for (mixed Tate) motivic periods and give a formula for the decomposition of motivic iterated Eisenstein integrals to leading order in the coradical filtration $C_\bullet \mathcal{P}^{\mm}_{\mathcal{H}}$ \cite[\S $2.5$, \S $3.8$]{motivicperiods}.

\begin{definition}[$f$-alphabet decomposition]
An $f$-alphabet decomposition is an isomorphism of $\mathbb{G}_m$-modules $\varphi\colon \mathcal{P}^{\mm}_{\MT} \xrightarrow{\sim} \mathbb{Q}[\mathbb{L}^\pm] \otimes \mathbb{Q}\langle f_3, f_5, \dots \rangle$. It is \emph{normalised} if $\varphi(\zeta^{\mm}(2n+1)) = f_{2n+1}$ for all $n\geq 1$.
\end{definition}

There is a more general notion of a decomposition map \cite[Definition $3.10$]{motivicperiods}, which is a canonical isomorphism of $S^{\dR}_{\mathcal{H}}$-modules $\Phi\colon \gr^C \mathcal{P}^{\mm}_{\mathcal{H}} \xrightarrow{\sim} \mathcal{P}^{\mm}_{\mathcal{H}^{ss}}\otimes {\mathbb{Q}\langle \gr_1^C \mathcal{O}(U^{\dR}_{\mathcal{H}}) \rangle}$. Here $G_{\mathcal{H}}^{\dR} = \Aut^{\otimes}_{\mathcal{H}} (\omega^{\dR}_{\mathcal{H}})$, $U_{\mathcal{H}}^{\dR}$ is its unipotent radical, $S^{\dR}_{\mathcal{H}}$ its reductive quotient, and $\mathcal{H}^{ss} \hookrightarrow \mathcal{H}$ the full Tannakian subcategory of semisimple objects. See \cite[\S $3$]{motivicperiods}.

The fully faithful functor $\omega^{\mathcal{H}}$ embeds $\MT \hookrightarrow \mathcal{H}$ and induces an inclusion $\mathcal{P}^{\mm}_{\MT} \hookrightarrow \mathcal{P}^{\mm}_{\mathcal{H}}$ (see \S \ref{tannakianperiods} and Remark \ref{motivicperiodinclusionremark}). As described in \cite[\S $5.4$]{motivicperiods} this restricts to a canonical isomorphism of $S^{\dR}_{\MT} (= \mathbb{G}_m)$-modules 
\begin{equation*}
    \Phi\colon \gr^C \mathcal{P}^{\mm}_{\MT} \xrightarrow{\sim} \mathbb{Q}[\mathbb{L}^\pm] \otimes \mathbb{Q} \langle f_3, f_5, \dots \rangle.
\end{equation*}
A choice of splitting of the coradical filtration $C_\bullet \mathcal{P}^{\mm}_{\MT}$ (for example, using the Hoffman elements \cite{brownmtm}) determines an isomorphism $\mathcal{P}^{\mm}_{\MT} \cong \gr^C \mathcal{P}^{\mm}_{\MT}$ and hence an isomorphism $\varphi$ as above. Different choices of splittings for the coradical filtration will determine different choices of $\varphi$; the map $\Phi = \gr^C \varphi$, however, is canonical.

The coradical filtration on motivic periods is connected to the length filtration on iterated integrals as follows:

\begin{remark} \label{lengthcoradical}
By \cite[Theorem $22.2$]{mmv} the element $\mathcal{C}_S^{\mm} = \int_S^{\mm} \in \mathcal{U}_{1,1}^{\dR} (\mathcal{P}^{\mm}_{\mathcal{H}})$ may be interpreted as a morphism of filtered $\mathbb{Q}$-algebras
\begin{equation*} 
    \mathcal{C}_S^{\mm} = \int_S^{\mm} \colon L_\bullet \mathcal{O}(\mathcal{U}_{1,1}^{\dR}) \to C_\bullet \mathcal{P}^{\mm}_{\mathcal{H}},
\end{equation*}
where the length filtration $L_\bullet \mathcal{O}(\mathcal{U}_{1,1}^{\dR})$ is defined by
\begin{equation*}
    L_r\mathcal{O}(\mathcal{U}_{1,1}^{\dR})=\langle [E_{2n_1+2}(0)\vert \dots \vert E_{2n_s+2}(0)]: s \leq r \rangle.
\end{equation*}
\end{remark}

Remark \ref{lengthcoradical} is compatible with the following formula for the $f$-alphabet decomposition of motivic iterated Eisenstein integrals.

\begin{lemma}\label{falphabet}
Let $n_1, \dots, n_s > 0$ and let $\varphi$ be any normalised $f$-alphabet decomposition. Let $\Phi = \gr^C \varphi$. Then $\Phi \colon \Co(\mu(\Psi)) \to \mathbb{Q}[\mathbb{L}^\pm] \otimes \mathbb{Q}\langle f_3, f_5, \dots \rangle$ and is given explicitly on the integrals of the lowest and highest weight vectors by
\begin{align*}
\int_S^{\mm}  [E_{2n_1+2}(0)\vert \dots \vert E_{2n_s+2}(0)] & \mapsto \left(-1 \right)^s \frac{(2n_1)! \cdots (2n_s)!}{2^s} \frac{f_{2n_1 + 1} \cdots f_{2n_s + 1}}{\mathbb{L}^{2n_1 + \dots + 2n_s}}\\
\int_S^{\mm}  [E_{2n_1+2}(2n_1)\vert \dots \vert E_{2n_s+2}(2n_s)] & \mapsto \frac{(2n_1)! \cdots (2n_s)!}{2^s} f_{2n_s + 1} \cdots f_{2n_1 + 1}.
\end{align*}
\end{lemma}
\begin{proof}
We prove the equality
\begin{equation} \label{Sactionintegrals}
\int_S^{\mm}  [E_{2n_1+2}(0)\vert \dots \vert E_{2n_s+2}(0)] = \frac{(-1)^s}{\mathbb{L}^{2n_1 + \dots + 2n_s}} \int_S^{\mm}  [E_{2n_s+2}(2n_s)\vert \dots \vert E_{2n_1+2}(2n_1)],
\end{equation}
which can be shown by making use of the action of $SL_2$ on $\mathcal{O}(\mathcal{U}_{1,1}^{\dR})$. The result follows by applying  \cite[Theorem $22.2$]{mmv} to the right hand side of \eqref{Sactionintegrals}, although some modifications must be made (namely, the coefficient needs to be inverted and the order of letters reversed).

To proceed with showing \eqref{Sactionintegrals}, define elements $v_{\mathsf{X}}, v_{\mathsf{Y}}\in\mathcal{O}(\mathcal{U}_{1,1}^{\dR})$ by $v_{\mathsf{X}}:= [E_{2n_1+2}(0)\vert \dots \vert E_{2n_s+2}(0)]$ and $v_{\mathsf{Y}}:= [E_{2n_1+2}(2n_1)\vert \dots \vert E_{2n_s+2}(2n_s)]$. They are dual to the elements $\mathbf{e}_{2n_1+2} \mathsf{X}_1^{2n_1} \cdots \mathbf{e}_{2n_s+2} \mathsf{X}_s^{2n_s} , \mathbf{e}_{2n_1+2} \mathsf{Y}_1^{2n_1} \cdots \mathbf{e}_{2n_s+2} \mathsf{Y}_s^{2n_s} \in \mathcal{U}_{1,1}^{\dR} (\mathbb{Q})$ respectively. The action of $S^{\mm}$ on $\mathcal{U}_{1,1}^{\dR}(\mathcal{P}^{\mm}_{\mathcal{H}})$ sends
\begin{equation*}
\mathbf{e}_{2n_1+2} \mathsf{X}_1^{2n_1} \cdots \mathbf{e}_{2n_s+2} \mathsf{X}_s^{2n_s} \mapsto \mathbb{L}^{2n_1 + \dots + 2n_s} \mathbf{e}_{2n_1+2} \mathsf{Y}_1^{2n_1} \cdots \mathbf{e}_{2n_s+2} \mathsf{Y}_s^{2n_s}.
\end{equation*}
The dual action of $S^{\mm}$ on $\mathcal{O}(\mathcal{U}_{1,1}^{\dR})\otimes \mathcal{P}^{\mm}_{\mathcal{H}}$ therefore satisfies
\begin{equation}\label{Sactiondual}
S^{\mm}\colon v_{\mathsf{Y}} \mapsto \mathbb{L}^{2n_1 + \dots + 2n_s} v_{\mathsf{X}}.
\end{equation}
Applying $\mathcal{C}_S^{\mm} \in \mathcal{U}_{1,1}^{\dR} (\mathcal{P}^{\mm}_{\mathcal{H}}) = \Hom(\mathcal{O}(\mathcal{U}_{1,1}^{\dR}), \mathcal{P}^{\mm}_{\mathcal{H}})$ to both sides of \eqref{Sactiondual} gives
\begin{equation} \label{coeffsofCs}
\mathcal{C}_S^{\mm}(v_{\mathsf{Y}}\vert _{S^{\mm}}) = \mathbb{L}^{2n_1 + \dots + 2n_s} \mathcal{C}_S^{\mm} (v_{\mathsf{X}}).
\end{equation}

By definition, $\mathcal{C}_S^{\mm}(v_{\mathsf{Y}}\vert _{S^{\mm}}) = (S^{\mm}(\mathcal{C}_S^{\mm}))(v_{\mathsf{Y}})$. The cocycle equation for $\mathcal{C}^{\mm}$ associated to the equation $S^2 = -I$ implies that $S^{\mm}(\mathcal{C}_S^{\mm}) = (\mathcal{C}_S^{\mm})^{-1}$. The element $(\mathcal{C}_S^{\mm})^{-1}$ is the generating series for motivic iterated integrals along $S^{-1}$, where we use the choice of splitting that defines $\mathcal{C}_S^{\mm}$. Equation \eqref{coeffsofCs} therefore reads
\begin{equation*}
\int_{S^{-1}}^{\mm} v_{\mathsf{Y}} = \mathbb{L}^{2n_1 + \dots + 2n_s} \int_S^{\mm} v_{\mathsf{X}}.
\end{equation*}
Applying the reversal of paths formula gives
\begin{equation*}
\int_{S}^{\mm} \text{reverse}(v_{\mathsf{Y}}) = (-1)^s \mathbb{L}^{2n_1 + \dots + 2n_s} \int_S^{\mm} v_{\mathsf{X}},
\end{equation*}
where $\text{reverse}(v_{\mathsf{Y}}):=[E_{2n_s+2}(2n_s)\vert \dots \vert E_{2n_1+2}(2n_1)]$. This is precisely equation \eqref{Sactionintegrals}.
 Taking $f$-alphabet decompositions gives the result.
\end{proof}

\begin{remark} \label{falphabetlemmaremark}
Lemma \ref{falphabet} is actually valid for all $n_1, \dots, n_s > 0$, even when the words $[E_{2n_1+2}(0)\vert \dots \vert E_{2n_s+2}(0)], [E_{2n_1+2}(2n_1)\vert \dots \vert E_{2n_s+2}(2n_s)] \in \mathcal{O}(\mathcal{U}_{1,1}^{\dR})$ are not contained in the mixed Tate subalgebra $\mathcal{O}(\Ugeom) \subseteq \mathcal{O}(\mathcal{U}_{1,1}^{\dR})$. This implies that the leading-order term in the coradical filtration of \emph{any} motivic iterated Eisenstein integral is an element of $\mathcal{P}^{\mm}_{\MT}$. With these more general integrals, however, non-mixed Tate motivic periods can appear lower down in the coradical filtration.
\end{remark}

\subsection{Cocycle equations}

To determine coefficients in linear combinations of iterated Eisenstein integrals we must understand relations between such integrals. The \emph{cocycle equations} for $\mathcal{C}^{\mm} \in Z^1 (\Gamma, \mathcal{U}_{1,1}^{\dR} (\mathcal{P}^{\mm}_{\mathcal{H}}))$ determine many of these. Set $U = S T \in \Gamma$; then
\begin{equation*}
    U^{\mm} = S^{\mm} T^{\mm} = \left( \begin{matrix} 0 & -\mathbb{L} \\ \mathbb{L}^{-1} & 1 \end{matrix} \right) \in SL_2^{\dR}(\mathcal{P}^{\mm}_{\MT}).
\end{equation*}
It is well-known that $S^2 = U^3 = -I$. Furthermore $\mathcal{C}_{-I}^{\mm} = 1$ because all modular forms for $\Gamma$ have even weight. The canonical cocycle therefore satisfies
\begin{align}
    \mathcal{C}_S^{\mm}\vert_{S^{\mm}} \cdot \mathcal{C}_S^{\mm} & = 1 \label{Seq}\\
    \mathcal{C}_U^{\mm} \vert_{(U^{\mm})^2} \cdot  \mathcal{C}_U^{\mm} \vert_{U^{\mm}} \cdot \mathcal{C}_U^{\mm} & = 1, \label{Ueq}
\end{align}
where $\mathcal{C}_U^{\mm} = \mathcal{C}_{S T}^{\mm} = \mathcal{C}_S^{\mm}\vert_{T^{\mm}}\cdot \mathcal{C}_T^{\mm}$.

Taking the coefficient of $\mathbf{e}_{2n_1+2} \mathsf{X}_1^{2n_1 - b_1} \mathsf{Y}_1^{b_1} \cdots \mathbf{e}_{2n_s+2}\mathsf{X}_s^{2n_s - b_s} \mathsf{Y}_s^{b_s}$ in equations \eqref{Seq} and \eqref{Ueq} produces relations between iterated Eisenstein integrals.

Equation \eqref{Ueq}, when written out in terms of $\mathcal{C}_S^{\mm}$ and $\mathcal{C}_T^{\mm}$, becomes rapidly complex as the length $s$ increases. However, since $\Co(\mathcal{C}_T^{\mm}) = \mathbb{Q}[\mathbb{L}]$, we may focus instead on only the leading-order terms in the coradical filtration. Doing so gives
\begin{equation} \label{Umodeq}
    \mathcal{C}_S^{\mm}\vert_{(TSTST)^{\mm}} \cdot \mathcal{C}_S^{\mm} \vert_{(TST)^{\mm}} \cdot \mathcal{C}_S^{\mm} \vert_{T^{\mm}} = 1 \in \mathcal{U}_{1,1}^{\dR}(\gr^C \mathcal{P}^{\mm}_{\mathcal{H}}).
\end{equation}
For our purposes this is just as useful and easier to compute; for example, taking the coefficient of a length $2$ word in \eqref{Ueq} produces a relation with $21$ terms, whereas taking the coefficient of the same word in \eqref{Umodeq} produces a relation modulo $C_1 \mathcal{P}^{\mm}_{\mathcal{H}}$ with only $6$ terms.

\subsection{Examples}

We now consider the two example equations given in the introduction to this paper and show how their coefficients may be determined using the above machinery. The essential idea in each case is to first use Theorem \ref{maintheorem} to obtain the (fairly small) finite set of iterated Eisenstein integrals that may appear in a linear combination equal to the given MZV. We then use relations between iterated Eisenstein integrals coming from the cocycle equations to exhibit linear dependencies in this set and reduce its size further. Finally we form a general linear combination of these remaining terms, apply Lemma \ref{falphabet} and compare coefficients.

\begin{remark}[Notation]
To ease notation we define
\begin{equation*}
    I_{2n_1+2, \dots, 2n_s +2}^{b_1, \dots, b_s} := \int_S^{\mm} [E_{2n_1+2}(b_1)\vert \dots \vert E_{2n_s+2}(b_s)].
\end{equation*}
\end{remark}

\begin{example} \label{example1}
Consider the expression for $\zeta(3)$ given in the introduction:
\begin{align*}
\zeta(3) &= -(2\pi i)^3 \int_0^{\vec{1}_\infty} \mathbb{G}_4(\tau) d\tau \\
&= -(2\pi i)^2 \int_S E_4(0).
\end{align*}
We illustrate how to derive this equation from the motivic theory. Theorem \ref{maintheorem} implies that $\zeta^{\mm}(3)$ can be written as a $\mathbb{Q}$-linear combination of the motivic periods
$\mathbb{L}^2 I_4^0$, $\mathbb{L} I_4^1$,   $I_4^2$ and $\mathbb{L}^3$. By Remark \ref{lengthcoradical} we work modulo $C_0 \mathcal{P}^{\mm}_{\MT}$, which kills $\mathbb{L}^3 \in C_0 \mathcal{P}^{\mm}_{\MT}$. We shortly explain why there can be no $\mathbb{L}^3$ term in the full linear combination, so in fact it makes no difference whether we work modulo $C_0 \mathcal{P}^{\mm}_{\MT}$ or not in this case.

Taking the coefficient of $\mathbf{e}_4 \mathsf{X}^2$ in \eqref{Seq} implies that $I_4^2 = -\mathbb{L}^2 I_4^0$. Taking the coefficient of $\mathbf{e}_4 \mathsf{X}\mathsf{Y}$ in \eqref{Umodeq} implies that $I_4^1 \equiv 0 \pmod{C_0 \mathcal{P}^{\mm}_{\MT}}$.\footnote{This concisely expresses that the term involving a zeta value is a coboundary and the image of the cocycle of $\mathbb{G}_{4}$ in cohomology is rational up to a power of $\mathbb{L}$ \cite[Equation $(7.8)$]{mmv}.}  Therefore we may write $\zeta^{\mm}(3) \equiv A \mathbb{L}^2 I_4^0 \pmod{C_0 \mathcal{P}^{\mm}_{\MT}}$, for some $A \in\mathbb{Q}$.
In fact, this expression holds ``on the nose'', without working modulo $C_0 \mathcal{P}^{\mm}_{\MT}$, because the undetermined term in $C_0 \mathcal{P}^{\mm}_{\MT}$ must have weight $3$ and so must be a multiple of $\mathbb{L}^3$, which is anti-invariant under the real Frobenius $F_\infty$. Therefore
\begin{equation} \label{zeta3ex}
\zeta^{\mm}(3) = A \mathbb{L}^2 I_4^0, \quad \text{for some } A \in\mathbb{Q}.
\end{equation}
We now apply the normalised decomposition map $\varphi$ to \eqref{zeta3ex}. Lemma \ref{falphabet} gives
\begin{equation*}
\varphi\left(I_4^0\right) = -\frac{2!}{2} \frac{f_3}{\mathbb{L}^2} + \lambda \mathbb{L}^3, \quad \text{for some } \lambda \in \mathbb{Q}.
\end{equation*}
Comparing coefficients in \eqref{zeta3ex} determines $A = -1$ and $\lambda = 0$.
\end{example}

\begin{remark}
A similar calculation may be used to show that
\begin{equation*}
    \zeta^{\mm}(2n+1) = -\frac{2}{(2n)!} \mathbb{L}^{2n} I_{2n+2}^0 = \frac{2}{(2n)!} I_{2n+2}^{2n},
\end{equation*}
which recovers the completed $L$-values $\Lambda(\mathbb{G}_{2n+2},1)$ and $\Lambda(\mathbb{G}_{2n+2},2n+1)$. Combining this with the relation obtained by taking the coefficient of $\mathbf{e}_{2n+2} \mathsf{X}^{2n-b}\mathsf{Y}^b$ in equation \eqref{Umodeq} also shows that
\begin{equation*}
    \mathbb{L}^{2n-b} I_{2n+2}^b \equiv (\delta_{2n-b} - \delta_b) \frac{(2n)!}{2} \zeta^{\mm}(2n+1) \pmod{C_0 \mathcal{P}^{\mm}_{\MT}},
\end{equation*}
where $\delta_c = \delta_{c,0}$ is the Kronecker delta. In particular, this implies that the ``middle'' values $I_{2n+2}^b$ for $1\leq b \leq 2n-1$ are powers of $\mathbb{L}$, which is of course also well-known. From this point onward we therefore do not distinguish between single integrals of Eisenstein series and the associated single zeta value.
\end{remark}

The second example gives an expression for $\zeta^{\mm}(3,5) \in \mathfrak{D}_2 \Zm \backslash \mathfrak{D}_1 \Zm$ as a double Eisenstein integral.

\begin{example} \label{example2}
We have the following expression for $\zeta(3,5)$ \cite[Example $7.2$]{brownihara}\footnote{There is a difference in normalisation between Brown's formulae and ours; namely $\Lambda(\mathbb{G}_{k_1}, \dots, \mathbb{G}_{k_s};b_1, \dots, b_s) = i^{b_1 +\dots + b_s} (2\pi i)^{-(b_1 + \dots + b_s)} \int_S [E_{k_1} (b_1-1) \vert \dots \vert E_{k_s} (b_s-1)]$.}:
\begin{align*}
\zeta(3,5) &= -\frac{5}{12} (2\pi i)^8 \int_S \mathbb{G}_6 (\tau_1) d\tau_1 \mathbb{G}_4 (\tau_2) d\tau_2 + \frac{503}{2^{13} 3^5 5^2 7}(2\pi i)^8 \\
& = -\frac{5}{12} (2\pi i)^6 \int_S [E_6(0)\vert E_4(0)] + \frac{503}{2^{13} 3^5 5^2 7}(2\pi i)^8.
\end{align*}
We explain how the coefficient $-5/12$ of the longest iterated Eisenstein integral may be computed abstractly. We first determine the possible such integrals that may appear in a linear combination giving $\zeta^{\mm} (3,5)$. Theorem \ref{maintheorem} implies that these must be iterated integrals of length at most $2$ and total modular weight at most $10$. By Remark \ref{lengthcoradical} we may work modulo $C_1 \mathcal{P}^{\mm}_{\MT}$, giving
\begin{equation} \label{zeta35initial}
    \zeta^{\mm}(3,5) \equiv \sum_{\substack{(n_1, n_2) \in \left\{(1,1), (1,2), (2,1)\right\} \\ 0\leq b_i \leq 2n_i}} \lambda I_{2n_1+2, 2n_2+2}^{b_1, b_2} \pmod{C_1 \mathcal{P}^{\mm}_{\MT}},
\end{equation}
where $\lambda \in \mathbb{Q}\mathbb{L}^{6-b_1 - b_2}$. The motivic MZV $\zeta^{\mm}(3,5)$ has weight $8$, and the weight $8$ subspace\footnote{Recall that the weight filtration is a grading on motivic MZVs.} of $C_1 \mathcal{P}^{\mm}_{\MT}$ is spanned by $\zeta^{\mm}(8)$, which is a rational multiple of $\mathbb{L}^8$. Thus, the only term undetermined modulo $C_1 \mathcal{P}^{\mm}_{\MT}$ is a multiple of $\mathbb{L}^8 \in C_0 \mathcal{P}^{\mm}_{\MT}$. 

Let $\varphi$ be some choice of normalised $f$-alphabet decomposition and $\Phi = \gr^C \varphi$. The decomposition algorithm \cite[\S $5$]{browndecomposition} implies that
\begin{equation} \label{Phizeta35}
\Phi (\zeta^{\mm}(3,5)) \equiv -5 f_5 f_3 \pmod{C_0 \mathcal{P}^{\mm}_{\MT}}.
\end{equation}
The longest word in this decomposition is $f_5 f_3$, with coefficient $-5$. The coaction on $\mathbb{Q}\langle f_3, f_5, \dots \rangle \otimes_{\mathbb{Q}} \mathbb{Q}[\mathbb{L}^\pm]$ is deconcatenation, and we compute
\begin{equation} \label{coactionzeta35}
    \Delta(\Phi(\zeta^{\mm}(3,5)) = -5\otimes f_5 f_3 - 5 f_5 \otimes f_3 - 5 f_5 f_3 \otimes 1.
\end{equation}
In particular we have $\Delta'(\Phi (\zeta^{\mm}(3,5))) \equiv -5 f_5 \otimes f_3 \neq 0 \pmod{C_1 \mathcal{P}^{\mm}_{\MT}}$, where $\Delta' := \Delta - \id\otimes 1 - 1\otimes \id$.

We must now compute the same expression for the iterated Eisenstein integrals on the right hand side of \eqref{zeta35initial}. To leading order in the coradical filtration $C_\bullet \mathcal{P}^{\mm}_{\mathcal{H}}$, the coaction on motivic iterated Eisenstein integrals is deconcatenation \cite[Theorem $22.2$]{mmv}. This can be written explicitly as follows:
\begin{align*}
     \Delta \left(I_{2n_1+2, 2n_2 +2}^{b_1, b_2}  \right) & \equiv 1\otimes I_{2n_1+2, 2n_2 +2}^{b_1, b_2}  + I_{2n_1+2, 2n_2 +2}^{b_1, b_2}  \otimes 1 \\
     & + I_{2n_1 +2}^{b_1}\otimes I_{2n_2+2}^{b_2}   \pmod{C_1 \mathcal{P}^{\mm}_{\mathcal{H}}}.
\end{align*}
Note that $I_{2n+2}^b\in \mathcal{P}^{\mm}_{\MT}$ is a rational multiple of:
\begin{equation*}
    \begin{cases}
    \mathbb{L}^{b-2n} \zeta^{\mm}(2n+1)\in C_1 \mathcal{P}^{\mm}_{\MT}, & b = 0 \text{ or } 2n; \\
    \mathbb{L}^{b+1} \in C_0 \mathcal{P}^{\mm}_{\MT}, & \text{otherwise}.
    \end{cases}
\end{equation*}
Therefore $\Delta'(I_{2n_1+2, 2n_2 +2}^{b_1, b_2}) \equiv 0 \pmod{C_1 \mathcal{P}^{\mm}_{\mathcal{H}}}$ unless $(b_1, b_2)$ is one of $(0,0)$, $(0, 2n_2)$, $(2n_1, 0)$ or $(2n_1, 2n_2)$. In any of these cases we have
\begin{equation*}
    \Delta'(\Phi(I_{2n_1+2, 2n_2 +2}^{b_1, b_2})) =(-1)^{\delta_{b_1} + \delta_{b_2}} \frac{(2n_1)! (2n_2)!}{4}\mathbb{L}^{b_1+b_2 - 2n_1 - 2n_2} (f_{2n_1 +1} \otimes f_{2n_2+1}).
\end{equation*}

Thus the only iterated Eisenstein integrals that contribute nontrivially modulo $C_1 \mathcal{P}^{\mm}_{\mathcal{H}}$ to the sum on the right hand side of \eqref{zeta35initial} are $I_{2n_1 + 2,2n_2 + 2}^{b_1, b_2}$ with $\left\{n_1,n_2\right\} = \left\{1,2\right\}$ and $b_i \in \left\{0, 2n_i\right\}$. As described in Lemma \ref{falphabet}, the action of $S^{\mm}$ gives an equality $I_{6,4}^{b_1, b_2} = \mathbb{L}^{6 - 2b_1 -2 b_2}I_{4,6}^{2-b_2, 4-b_1}$ so that we may fix $(n_1,n_2) = (2,1)$. The sum in \eqref{zeta35initial} therefore consists of at most 4 distinct integrals.

We may simplify further using the cocycle equations. Taking the coefficients of $\mathbf{e}_6 \mathsf{X}_1^4 \mathbf{e}_4 \mathsf{X}_2^2$ and $\mathbf{e}_6 \mathsf{X}_1^4 \mathbf{e}_4 \mathsf{Y}_2^2$ in \eqref{Seq} respectively produces the two equations
\begin{align}
    \mathbb{L}^6 I_{6,4}^{0,0} + I_{6,4}^{4,2} - 12\zeta^{\mm}(3) \zeta^{\mm}(5) &= 0; \label{symmetricS}\\
    \mathbb{L}^4 I_{6,4}^{0,2} + \mathbb{L}^2 I_{6,4}^{4,0} + 12\zeta^{\mm}(3) \zeta^{\mm}(5) &= 0. \label{antisymmetricS}
\end{align}
Taking the coefficients of the same words in \eqref{Umodeq} produces the equations
\begin{align} 
    2 \mathbb{L}^6 I_{6,4}^{0,0}  + \mathbb{L}^4 I_{6,4}^{0,2} + \mathbb{L}^2 I_{6,4}^{4,0} + 2 I_{6,4}^{4,2} - 12\zeta^{\mm}(3) \zeta^{\mm}(5) &\equiv 0 \pmod{C_1 \mathcal{P}^{\mm}_{\mathcal{H}}}; \label{Ucoeff1} \\
    \mathbb{L}^6 I_{6,4}^{0,0} + 2 \mathbb{L}^4 I_{6,4}^{0,2} + \mathbb{L}^2 I_{6,4}^{4, 0}   + I_{6,4}^{4,2} &\equiv 0 \pmod{C_1 \mathcal{P}^{\mm}_{\mathcal{H}}}. \label{Ucoeff2}
\end{align}
Equation \eqref{Ucoeff1} is implied by combining \eqref{symmetricS} with \eqref{antisymmetricS}, which shows that it holds even without working modulo $C_1 \mathcal{P}^{\mm}_{\mathcal{H}}$. Though this gives no interesting new information for this calculation, it does imply the existence of a further relation in $C_1 \mathcal{P}^{\mm}_{\mathcal{H}}$. What is more useful is \eqref{Ucoeff2}.
Combining it with equations \eqref{symmetricS} and \eqref{antisymmetricS} implies that
\begin{equation} \label{Ueqcoeff3}
    I_{6,4}^{0,2} \equiv 0 \pmod{C_1 \mathcal{P}^{\mm}_{\mathcal{H}}}.
\end{equation}
The combination of \eqref{symmetricS}, \eqref{antisymmetricS} and \eqref{Ueqcoeff3}, together with Theorem \ref{maintheorem}, implies that $\zeta^{\mm}(3,5)$ may be written as a $\mathbb{Q}$-linear combination
\begin{equation} \label{zeta35nearlydone}
    \zeta^{\mm}(3,5) \equiv A \mathbb{L}^6 I_{6,4}^{0,0} + B \zeta^{\mm}(3)\zeta^{\mm}(5) \pmod{C_0 \mathcal{P}^{\mm}_{\MT}}, \quad A,B \in \mathbb{Q}.
\end{equation}
Applying Lemma \ref{falphabet} to \eqref{zeta35nearlydone} gives
\begin{align*}
    -5 f_5 f_3 & \equiv  \Phi(\zeta^{\mm}(3,5)) \\
    & \equiv \Phi(A \mathbb{L}^6 I_{6,4}^{0,0} + B \zeta^{\mm}(3)\zeta^{\mm}(5)) \\
    & = 12 A f_5 f_3 + B f_3 \shuffle f_5 \\
    & = (12 A + B) f_5 f_3 + B f_3 f_5 \pmod{C_0 \mathcal{P}^{\mm}_{\MT}}.
\end{align*}
Comparing coefficients determines $A = -5/12$ and $B = 0$. We obtain
\begin{equation*}
    \zeta^{\mm}(3,5) = -\frac{5}{12} \mathbb{L}^6 I_{6,4}^{0,0} + \lambda \mathbb{L}^8, \quad \text{for some } \lambda\in \mathbb{Q}.
\end{equation*}

\end{example}

\subsubsection{The general case and further questions} \label{futuresection}

Theorem \ref{maintheorem} implies that any $\zeta^{\mm}(w) \in \Zm$ may be written as a linear combination of the form
\begin{equation} \label{generalcaselincomb}
\zeta^{\mm}(w) = \sum_i A_i \mathbb{L}^{m_i} \int_S^{\mm} v_i,
\end{equation}
where $v_i\in\mathcal{O}(\mathcal{U}_{1,1}^{\dR})$ and $A_i\in\mathbb{Q}$. Theorem \ref{maintheorem} also places constraints on the elements $v_i$ and values $m_i \in \mathbb{Z}$ in terms of the weight $n$ and depth $r$ of $\zeta^{\mm}(w)$.

The left hand side of \eqref{generalcaselincomb} has an $f$-alphabet decomposition. In many cases (e.g. for the family of Hoffman MZVs \cite{brownmtm}) this decomposition may be computed algorithmically up to a rational multiple of $f_n$ \cite{browndecomposition}, though in general this decomposition is not unique. However the left hand side has a coradical filtration and, in particular, the longest term in the $f$-alphabet decomposition with respect to the coradical filtration is canonically determined by this algorithm.

The right hand side also has an $f$-alphabet decomposition. The space of iterated Eisenstein integrals on the right hand side has a length filtration\footnote{This coincides with the coradical filtration.} with respect to which the longest part in the $f$-alphabet decomposition is canonically determined by Lemma \ref{falphabet}. By comparing the $f$-alphabet decompositions of each side of \eqref{generalcaselincomb} we may then read off the coefficients $A_i$ of the leading order terms in the coradical or length filtrations.

By extending the formula given in Lemma \ref{falphabet} for the $f$-alphabet decomposition for motivic iterated Eisenstein integrals past the highest length term it may be possible to determine more of the coefficients $A_i$. In the most optimistic scenario all coefficients except that of $\zeta^{\mm}(n)$ may be determined by this method. This raises the following question:

\begin{problem}
Is there a recursive procedure to completely decompose motivic iterated Eisenstein integrals in the $f$-alphabet in a similar way to the procedure given for motivic MZVs defined in \cite{browndecomposition}?
\end{problem}

A related area of potential future study is suggested by Theorem \ref{motivetheorem}, which implies that the action of $\mathfrak{k} = \Lie(U_{\MT}^{\dR})$ on $\ugeom$ is faithful. The first few terms of this action with respect to the filtration $W_\bullet \ugeom$ were computed in \cite{brownzetaelements}. 

\begin{problem}
Describe the action of $\mathfrak{k}$ on $\ugeom$ completely.
\end{problem}

The proof of Theorem \ref{mainargument} implies that the words $v_i$ appearing on the right hand side of \eqref{generalcaselincomb} are elements of the Hopf subalgebra $\mathcal{O}(\Ugeom)\subseteq \mathcal{O}(\mathcal{U}_{1,1}^{\dR})$. The value of such a word under the homomorphism $\int_S^{\mm} = \mu(\mathcal{C}_S^{\mm}) \in \Hom (\mathcal{O}(\Ugeom), \mathcal{P}^{\mm}_{\mathcal{H}})$ is an element of $\mathcal{P}^{\mm}_{\MT}$ because $\mathcal{O}(\Ugeom)$ is an ind-object of $\MT$. Applying the period map and \cite{brownmtm} implies that $\int_S v_i = \mu(\mathcal{C}_S)(v_i) \in \mathcal{Z}[(2\pi i)^\pm] \subseteq \mathbb{C}$.

We expect that a form of converse holds as long as one considers integrals of not just a single word $v_i$ but rather the whole (finite) subset of words obtained by varying the values of $b_1, \cdots, b_s$ within $v_i$ in the allowed ranges $0 \leq b_j \leq 2n_j$. This may be formalised as follows: for $n_1,\dots, n_s > 0$ define a linear map 
\begin{equation*}
    [E_{2n_1+2}\vert \dots \vert E_{2n_s+2}] \colon V_{2n_1}^{\dR} \otimes \cdots \otimes V_{2n_s}^{\dR} \to \mathcal{O}(\mathcal{U}_{1,1}^{\dR})
\end{equation*}
by the formula
\begin{equation*}
    \mathsf{X}_1^{2n_1 - b_1} \mathsf{Y}_1^{b_1} \otimes \cdots \otimes \mathsf{X}_s^{2n_s - b_s} \mathsf{Y}_s^{b_s} \mapsto [E_{2n_1+2}(b_1)\vert \dots \vert E_{2n_s+2}(b_s)].
\end{equation*}
It is $SL_2$-equivariant. Precomposing with any choice of $SL_2$-equivariant linear map $g \colon V_{2n}^{\dR} \to V_{2n_1}^{\dR} \otimes \cdots \otimes V_{2n_s}^{\dR}$ defines an $SL_2$-equivariant linear map
\begin{equation} \label{Eispullback}
    V_{2n}^{\dR} \xrightarrow{g} V_{2n_1}^{\dR} \otimes \cdots \otimes V_{2n_s}^{\dR} \xrightarrow{[E_{2n_1+2}\vert \dots \vert E_{2n_s+2}]} \mathcal{O}(\mathcal{U}_{1,1}^{\dR}).
\end{equation}
Let $\mathcal{O}(\mathcal{U}_E^{\dR}) \subseteq \mathcal{O}(\mathcal{U}_{1.1}^{\dR})$ be the Hopf subalgebra\footnote{Note that $\mathcal{O}(\mathcal{U}_E^{\dR})$ is not a natural ind-object of $\mathcal{H}$.} generated by Eisenstein words, so that $\mathcal{O}(\Ugeom) \subseteq \mathcal{O}(\mathcal{U}_E^{\dR}) \subseteq \mathcal{O}(\mathcal{U}_{1.1}^{\dR})$. Every element of $\Hom_{SL_2} (V_{2n}^{\dR}, \mathcal{O}(\mathcal{U}_E^{\dR}))$ is a $\mathbb{Q}$-linear combination of maps of the form \eqref{Eispullback}. With this in mind we pose the following conjecture.

\begin{conjecture} \label{moduleMZVs}
Let $f \in \Hom_{SL_2} (V_{2n}^{\dR}, \mathcal{O}(\mathcal{U}_E^{\dR}))$. Then $\im(f) \subseteq \mathcal{O}(\Ugeom)$ if and only if $\im(\mathcal{C}_S \circ f) \subseteq \mathcal{Z}[(2\pi i)^\pm]$, where we view $\mathcal{C}_S \in \mathcal{U}_{1,1}^{\dR} (\mathbb{C}) = \Hom(\mathcal{O}(\mathcal{U}_{1,1}^{\dR}), \mathbb{C})$.
\end{conjecture}

The implications of Conjecture \ref{moduleMZVs} are somewhat subtle. For example, there are elements of $\mathcal{O}(\mathcal{U}^E)$ that are (conjecturally) \emph{not} contained in $\mathcal{O}(\Ugeom)$ whose image under $\mathcal{C}_S$ is contained in $\mathcal{Z}[(2\pi i)^\pm]$. Brown gives two numerical examples illustrating this in \cite[Examples $7.3$ and $7.5$]{brownihara}. Their conjectural motivic versions are
\begin{align}
       \frac{\mathbb{L}^{10}}{2^5 \cdot 3^3 \cdot 5} I_{6,8}^{0,0} - \frac{3 \mathbb{L}^4}{2^3 \cdot 691} I_{4,10}^{2,4} &\overset{?}{=} \zeta^{\mm}_{5,7}, \label{nonex1} \\ 
       \frac{\mathbb{L}^{10}}{2^6 \cdot 3^2 \cdot 5 \cdot 7} I_{4,10}^{0,0} + \frac{\mathbb{L}^4}{2^3 \cdot 691} I_{4,10}^{2,4} &\overset{?}{=} \zeta^{\mm}_{3,9}. \label{nonex2}
\end{align}
Here $\zeta^{\mm}_{5,7}, \zeta^{\mm}_{3,9} \in C_2 \mathcal{P}^{\mm}_{\MT}$ are motivic MZVs whose $f$-alphabet decompositions are $f_5 f_7 \pmod{\mathbb{Q} \mathbb{L}^{12}}$ and $f_3 f_9 \pmod{\mathbb{Q} \mathbb{L}^{12}}$ respectively. They are only well-defined up to addition of a rational multiple of $\mathbb{L}^{12}$. Despite lying in coradical degree $2$ they cannot be expressed by motivic MZVs of depth less than $4$.

A verification of \eqref{nonex1} and \eqref{nonex2} would imply that the linear combinations
\begin{align*}
    \frac{(2 \pi i)^{10}}{2^5 \cdot 3^3 \cdot 5} [E_6(0)\vert E_8 (0)] &- \frac{3 (2 \pi i)^4}{2^3 \cdot 691} [E_4(2)\vert E_{10}(4)] \\
    \frac{(2 \pi i)^{10}}{2^6 \cdot 3^2 \cdot 5 \cdot 7} [E_4(0)\vert E_{10}(0)] &+ \frac{(2 \pi i)^4}{2^3 \cdot 691} [E_4(2)\vert E_{10}(4)],
\end{align*}
which are elements of $\mathcal{O}(\mathcal{U}_E^{\dR}) \otimes_{\mathbb{Q}} \mathbb{Q}[2 \pi i]$, evaluate to elements of $\mathcal{Z}[(2\pi i)^\pm]$ under $\mathcal{C}_S$. Conjecturally, however, the individual terms appearing in these linear combinations are not contained in $\mathcal{O}(\Ugeom) \otimes_{\mathbb{Q}} \mathbb{Q}[2 \pi i]$ since their values under $\mathcal{C}_S$ should also involve the noncritical $L$-value $\Lambda(\Delta;12)$ of the cusp form $\Delta \in S_{12}(\Gamma)$, as well as a ``new'' period $c(\Delta;12)$ \cite[\S $7.2$]{brownihara}. By Remark \ref{falphabetlemmaremark} the motivic versions of these periods associated to $\Delta$ must occur in $C_1\mathcal{P}^{\mm}_{\mathcal{H}}$. They happen to cancel out in the linear combinations on the left hand sides of \eqref{nonex1} and \eqref{nonex2} to give the mixed Tate periods $\zeta^{\mm}_{5,7}$ and $\zeta^{\mm}_{3,9}$.

A further implication is that equations \eqref{nonex1} and \eqref{nonex2} are inaccessible from Theorem \ref{maintheorem}, which only makes use of periods of $\mathcal{O}(\Ugeom)$. This is evidenced by the expressions \eqref{nonex1} and \eqref{nonex2} for the depth $4$ motivic MZVs $\zeta^{\mm}_{5,7}$ and $\zeta^{\mm}_{3,9}$ as linear combinations of iterated Eisenstein integrals of length $s\leq 2$, while Theorem \ref{maintheorem} only gives the weaker bound $s \leq 4$.

This suggests that it is possible to reduce the length $s$ of the iterated Eisenstein integrals appearing on the right hand side of \eqref{generalcaselincomb} substantially below the bound $s \leq \depth \zeta^{\mm}(w)$ afforded by Theorem \ref{maintheorem} if we allow integrals of elements in the full space $\mathcal{O}(\mathcal{U}_E^{\dR})$ rather than just the mixed Tate subspace $\mathcal{O}(\Ugeom)$. In this larger space, cuspidal correction terms in lower coradical degree may account for these ``depth defects''.

This also raises an interesting question about relations between iterated Eisenstein integrals. Theorem \ref{maintheorem} implies that the depth $4$ MZVs appearing on the right hand sides of \eqref{nonex1} and \eqref{nonex2} may be written as linear combinations of iterated Eisenstein integrals of length at most $4$ that, moreover, come from $\mathcal{O}(\Ugeom)$. Equating with the iterated Eisenstein integrals on the respective left hand sides produces relations between iterated Eisenstein integrals along $S$ of potentially different lengths. It is natural to ask whether these relations may be proven using known relations (e.g. the cocycle relations). We therefore conclude with a final suggestion for future investigation:

\begin{problem}
Consider the relations in $M^{\mm}$ arising as follows: let
\begin{equation} \label{moregeneral}
    \sum_{i} \int_S^{\mm} w_i = \kappa^{\mm} \in \mathcal{P}^{\mm}_{\MT} = \Zm[\mathbb{L}^\pm]
\end{equation}
be a mixed Tate linear combination of motivic iterated Eisenstein integrals, where each $w_i \in \mathcal{O}(\mathcal{U}_E^{\dR}) \backslash \mathcal{O}(\Ugeom)$. Theorem \ref{maintheorem} produces a new expression
\begin{equation} \label{lessgeneral}
    \kappa^{\mm} = \sum_j \int_S^{\mm} v_j
\end{equation}
where each $v_i \in \mathcal{O}(\Ugeom)$. Equating \eqref{moregeneral} and \eqref{lessgeneral} produces the following relation between mixed Tate linear combinations of iterated Eisenstein integrals:
\begin{equation} \label{newrelation}
    \sum_i \int_S^{\mm} w_i = \sum_j \int_S^{\mm} v_j.
\end{equation}
Does the relation \eqref{newrelation} have a geometric origin?
\end{problem}

\bibliographystyle{amsplain}
\bibliography{ReferenceFile.bib}{}

\end{document}